\documentclass[leqno]{amsart}
\usepackage{latexsym}
\usepackage{amsmath}
\usepackage{amssymb}
\usepackage{amsfonts}
\usepackage{amsxtra}
\usepackage[all, 2cell, knot]{xy} \UseAllTwocells \SilentMatrices
\usepackage{xspace}
\usepackage{amsmath}
\usepackage{amstext}
\usepackage{amsfonts}
\usepackage[mathscr]{euscript}
\usepackage{amscd}

\newcommand{\bbD}{\mathbb{D}}
\newcommand{\bbE}{\mathbb{E}}
\newcommand{\bbX}{\mathbb{X}}
\newcommand{\bbY}{\mathbb{Y}}
\newcommand{\bbComp}{\mathbb{C}\mbox{\rm omp}}
\newcommand{\bbSpan}{\mathbb{S}\mbox{\rm pan}}
\newcommand{\Comp}{\mbox{\rm Comp}}

\newcommand{\calA}{\mathcal{A}}
\newcommand{\calB}{\mathcal{B}}
\newcommand{\calC}{\mathcal{C}}
\newcommand{\calH}{\mathcal{H}}
\newcommand{\calJ}{\mathcal{J}}
\newcommand{\calU}{\mathcal{U}}
\newcommand{\calV}{\mathcal{V}}

\newcommand{\bfC}{\mathbf{C}}
\newcommand{\bfD}{\mathbf{D}}
\newcommand{\bfE}{\mathbf{E}}

\newcommand{\Hom}{\mbox{\rm Hom}}
\newcommand{\Cyl}{\mbox{\rm Cyl}}
\newcommand{\Ps}{\mbox{\rm Ps}}
\newcommand{\St}{\mathit{St}}
\newcommand{\Bic}{\mbox{\bf Bic}}
\newcommand{\smBic}{\mbox{\bf\scriptsize Bic}}
\newcommand{\smDbl}{\mbox{\bf\scriptsize Dbl}}
\newcommand{\Bicat}{\mbox{\sf Bicat}}
\newcommand{\NHom}{\Bicat_{\mbox{\scriptsize\rm icon}}}
\newcommand{\Dbl}{\mbox{\bf Dbl}}
\newcommand{\W}{\{W\}}
\newcommand{\CW}{\bfC\W}
\newcommand{\Tm}{\mbox{\sf GTa}_{2}}
\newcommand{\Ta}{\mbox{\sf Ta}_{2}}
\newcommand{\GWGDbl}{\mbox{\sf GWGDbl}}
\newcommand{\DblCat}{\mbox{\sf DblCat}}
\newcommand{\PsTalg}{\mbox{\sf Ps-}T\mbox{\sf -alg}}
\newcommand{\ps}{\mbox{\scriptsize\rm ps}}

\newcommand{\WGDbl}{\mbox{\sf WGDbl}}
\newcommand{\smWGDbl}{\mbox{\sf\scriptsize WGDbl}}

\newcommand{\Cat}{\mbox{\sf Cat}}
\newcommand{\Set}{\mbox{\sf Set}}
\newcommand{\smBicat}{\mbox{\sf\scriptsize Bicat}}

\newcommand{\tiund}[1]{{\times}_{#1}}
\newcommand{\ovl}[1]{\overline{#1}}
\newcommand{\ovll}[1]{\overset{=}{#1}}
\newcommand{\pt}{\partial}
\newcommand{\dop}{\Delta^{\mbox{\scriptsize\rm op}}}
\newcommand{\Rar}{\Rightarrow}
\newcommand{\bsim}{/\!\!\sim}
\newcommand{\pro}[3]{#1\tiund{#2}\overset{#3}{\cdots}\tiund{#2}#1}
\newcommand{\tens}[2]{#1\,\tiund{#2}\,#1}
\newcommand{\JC}{\calJ_{\bfC}}

\newcommand{\supar}[1]{\overset{#1}{-\!\!-\!\!\!\rightarrow}}
\newcommand{\suparle}[1]{\overset{#1}{\leftarrow\!\!\!-\!\!-}}
\newcommand{\rw}{\rightarrow}

\newcommand{\id}{\mathrm{id}}

\newcommand{\vep}{\varepsilon}

\newcommand{\pr}{\mathrm{pr}}
\newcommand{\tp}{\mathsf{Top}}

\newtheorem{thm}{Theorem}[section]
\newtheorem{cor}[thm]{Corollary}
\newtheorem{corollary}[thm]{Corollary}
\newtheorem{theorem}[thm]{Theorem}
\newtheorem{prop}[thm]{Proposition}
\newtheorem{proposition}[thm]{Proposition}

\newtheorem{lma}[thm]{Lemma}
\newtheorem{lemma}[thm]{Lemma}

\newtheorem*{thma}{Theorem A}
\newtheorem*{thmb}{Theorem B}
\newtheorem*{thmc}{Theorem C}

\theoremstyle{definition} 
\newtheorem{dfn}[thm]{Definition}
\newtheorem{definition}[thm]{Definition}
\newtheorem{eg}[thm]{Example}

\newtheorem{rmks}[thm]{Remarks}
\newtheorem{rmk}[thm]{Remark}
\newtheorem{remark}[thm]{Remark}
\newtheorem{notation}[thm]{Notation}

\begin{document}

\title {Bicategories and Weakly Globular Double Categories}

\author{Simona Paoli}
\address{Department of Mathematics, 
University of Leicester,
LE17RH, UK}
\email{sp424@le.ac.uk}

\author{Dorette Pronk}
\address{Department of Mathematics and Statistics, Dalhousie University, Halifax, NS, B3H 4R2, Canada}
\email{pronk@mathstat.dal.ca}


\begin{abstract}
This paper introduces the notion of weakly globular double categories, a particular class of strict double categories,
 as a way to model weak 2-categories; it explores its use in defining a double category of fractions, and shows 
that the sub-2-category of groupoidal weakly globular double categories forms an algebraic model of homotopy 2-types.
\end{abstract}

\maketitle



\section{Introduction}\label{int}

In this paper we introduce the notion of weakly globular double categories, 
a particular class of strict double categories,
as a way to model weak 2-categories. We show that weakly globular double categories 
provide an algebraic model of homotopy 2-types and explore their use in defining a double category of fractions. 

\subsection{Weak Globularity}
Several different models of weak $n$-categories exist in the literature 
(for an overview, see \cite{Lei1} and \cite{Lei2})  and,
despite their many differences, they share the common feature that the totality
of cells in each dimension has a discrete structure; that is, it is a set. We
refer to this property as the globularity condition.
Recently, Blanc and the first author showed in \cite{bpa} and \cite{bp} 
that, in  modeling homotopy $n$-types, a new type of higher
groupoidal structure arises naturally: this differs from the classical ones
in that the weakening does not occur in relaxing the associativity and unit laws of
the composition of cells, but in relaxing the globularity condition. That is,
in requiring that the cells in each dimension have a structure which is not
discrete, but suitably equivalent to a discrete one. These structures, called
weakly globular $n$-fold groupoids in \cite{bpa}, have been shown to satisfy many interesting homotopical
and categorical properties, and to admit a comparison functor with a classical
model of weak $n$-groupoids, due to Tamsamani \cite{tam}.
Likewise, in the path-connected case, weakly globular cat$^{n-1}$-groups were shown 
by the first author \cite{Pao} to model path-connected $n$-types.

Starting in dimension 2, these results motivate us to search for a more general notion of  ``weakly
globular double category'' and to show that this is suitably equivalent to a
classical notion of weak $2$-category, such as the one by Tamsamani.
This low dimensional case presents a special interest in this context because 
it links with the vast existing literature on double categories, see for instance, 
\cite{Ehr}, \cite{BS},  \cite{Gr}, \cite{Shul} and \cite{FPP},
and with the notion of bicategory of fractions due the second author \cite{P1995}, \cite{P-thesis}, 
as developed in the second part of this paper.

\subsection{Rigidification}
The process of associating to a Tamsamani weak $2$-category an equivalent weakly globular
double category should be thought of as {\em rigidification}, since the
composition and unit laws, which in the Tamsamani model hold only up to
isomorphism, are equalities in the new structure. This type of rigidification,
is, however, quite different from the classical strictification of
bicategories \cite{ben}. In the latter, in fact,
the globularity condition is preserved: a bicategory becomes a strict $2$-category 
with (the same) discrete set of objects. Instead, in passing from a Tamsamani weak
$2$-category to a weakly globular double category we relax the globularity
condition: the set of objects of the Tamsamani weak 2-category is replaced after 
the rigidification by a non-discrete structure consisting of a posetal groupoid whose 
set of connected components is the original set of objects.

A Tamsamani weak 2-category is a functor $X_*\colon\dop\rightarrow\Cat$
such that $X_0$ is discrete and the Segal maps $\eta_k\colon X_k\rightarrow\pro{X_1}{X_0}{k}$
are equivalences of categories.
In \cite{lp}, Lack and the first author described a biequivalence of 2-categories between
the 2-category of bicategories, normal homomorphisms and icons
(identity component pseudo-natural transformations) and the 2-category of Tamsamani weak 2-categories,
with pseudo-natural transformations. The Tamsamani weak 2-categories
were obtained by taking an appropriate $\Cat$-valued nerve (called 2-nerve) of bicategories.

In order to be able to view such a functor as the horizontal nerve of a double category,
we need the Segal maps to be isomorphisms of categories rather than mere equivalences.
So we first functorially associate to a Tamsamani weak 2-category $X_*$ 
a 2-equivalent pseudo-functor $(SX)_*$ with the property that $SX_0=X_0$ and $SX_1=X_1$, but 
$SX_k=\pro{X_1}{X_0}{k}$; this places us in the position to apply a strictification result for pseudo-functors
due to Power \cite{pow} to obtain a strict functor $(RX)_*\colon \dop\rightarrow\Cat$. By working out the details
of Power's construction for this type of functors we find that  $(RX)_0$ is
no longer discrete, but merely equivalent to a discrete category; that is, $(RX)_0\simeq X_0$. Further, 
$(RX)_k$ is not only equivalent to
$\pro{(RX)_1}{(RX)_0}{k}$, but in fact isomorphic to $\pro{(RX)_1}{X_0}{k}$.
So we obtain the horizontal nerve of a strict double category $\bbX$ with a posetal groupoidal category
$\bbX_0$ of vertical arrows such that  the Segal maps induce equivalences of categories, 
$\pro{\bbX_1}{\bbX_0}{n}\simeq\pro{\bbX_1}{\bbX_0^d}{n}$ for $n\ge 2$.
This condition is satisfied in particular when either $d_0\colon\bbX_1\rightarrow\bbX_0$ 
or $d_1\colon\bbX_1\rightarrow\bbX_0$ is an isofibration.
In general, this property implies that for each pair of a horizontal and vertical arrow with a common codomain,
$$
\xymatrix@R=1.5em@C=2em{
&B_1\ar[d]|\bullet^v\\
A_0\ar[r]_f&B_0
}
$$
there is a completion to a double cell of the form
$$
\xymatrix@R=1.8em{A_2\ar@{}[ddr]|\alpha\ar[dd]|\bullet_u\ar[r]^g & B_2\ar[d]|\bullet^w
\\
&B_1\ar[d]|\bullet^v\\
A_0\ar[r]_f&B_0
}
$$
where $\alpha$ is vertically invertible.
We call such double categories weakly globular, and prove in Theorem \ref{s4.pro4} 
below that there is a pseudo-inverse to this construction giving 
us a biequivalence of 2-categories:

\begin{thma}
 There is a biequivalence of 2-categories:
\begin{equation*}
    (\WGDbl)_{\ps}\simeq(\Ta)_{\ps}\;.
\end{equation*}
\end{thma}

Imposing some invertibility conditions, this theorem specializes to a new model of $2$-types, 
called groupoidal weakly globular double categories, which generalizes the weakly 
globular double groupoids of \cite{bp}. 

The arrows between weakly globular double categories that correspond
to normal homomorphisms of bicategories are pseudo-functors  that correspond
to pseudo-natural transformations of simplicial functors, so they are weaker 
than what one might expect for double categories in the sense that horizontal domains and codomains are
only preserved up vertical isomorphisms. We conclude that we can strictify a bicategory to a double category
and in some sense move the weakness of the bicategory into the vertical part of the double category structure.
This is reflected in the morphisms we obtain. With these morphisms the correct notion of 2-cell is
a version of vertical transformation that has also been appropriately weakened to correspond to
modifications between pseudo-natural transformations of simplicial functors into $\Cat$.
So we obtain a biequivalence of 2-categories
$$\xymatrix@C=5em{
(\WGDbl)_{\ps}\ar@/^2ex/[r]^-{\smBic}\ar@{}[r]|-{\simeq_{\mathrm{bi}}}&\NHom\ar@/^2ex/[l]^-{\smDbl}}$$
but these functors do not form a biadjunction.

Another way in which we can make precise how the weak aspect of bicategories has moved to the vertical arrows
in the corresponding double categories is by studying  quasi units and internal equivalences in bicategories. 
Quasi units are
endomorphisms that are isomorphic to identity arrows. We show that quasi units in a bicategory $\calB$
are closely related  to companion pairs in $\Dbl(\calB)$ and companion pairs in a weakly globular double category
$\bbD$ correspond precisely
to quasi units in its fundamental bicategory $\Bic(\bbD)$. 
An internal equivalence in a bicategory is an arrow with a pseudo- inverse arrow in the sense that both 
composites are isomorphic to the respective identity arrows. 
We introduce a notion of {\em pre-companion} in a double category, such that,
similarly to the case of companions and quasi-units, pre-companions in double categories 
correspond to internal equivalences in bicategories.

\subsection{Double categories and universal properties}
As for general double categories, weakly globular double categories can be considered as a 2-category in more than 
one way. When one uses strict functors, one has the choice between horizontal and vertical transformations.
Since we used pseudo-functors above, which are weak in the horizontal direction and preserve
the horizontal domains and codomains only up to a vertical isomorphism, only the vertical 
transformations made sense.
However, if we use strict functors, we can also use horizontal transformations 
to form a 2-category $\WGDbl_h$ of weakly globular double categories.
The universal properties of an object with respect to one structure are distinct from those 
with respect to the other structure,
and it has been shown in the literature that it is important to consider both together.
Grandis and Par\'e noted in \cite{GP} for instance that when one takes the double category of topological groups
with group homomorphisms in the horizontal direction and continuous functions 
in the vertical direction, with commutative squares
as double cells, one needs the property of both being a horizontal and a vertical product in order to obtain 
the product of topological groups.

When we apply this to the relationship between weakly globular double categories and bicategories,
we see that universal properties of bicategories correspond to vertical universal properties for the 
corresponding weakly globular double categories, and these properties determine the weakly globular
double category up to a 2-equivalence in the vertical direction. However, this 2-equivalence class may intersect 
with various 2-equivalence classes in the horizontal direction. 
And some of them may carry  a (horizontal) universal property of their own,
so that both properties together determine the object up to both horizontal and vertical 2-equivalences.
We illustrate this by considering a double categorical 
variant of the bicategory of fractions of a category.

\subsection{Weakly globular double categories of fractions}
For a category $\bfC$ with a class of arrows $W$ satisfying certain axioms,
Gabriel and Zisman introduced the category of fractions ${\bfC}[W^{-1}]$
which can be viewed as a homotopy category of $\bfC$ with respect to the class $W$.
The construction is not only useful in homotopy theory, but anywhere
where one has a class of morphisms which one wants to invert.

For instance, when constructing the category of smooth manifolds, one
may start by considering just the atlases for the manifolds and take smooth maps 
between the associated groupoids (which may be viewed as a smooth version of an equivalence relation
on the disjoint union of the atlas charts).
In this case we obtain the usual category of manifolds by taking the category of fractions with respect
to the class of atlas refinements. A related example is that of \'etendues and geometric morphisms.
Etendues can be represented by \'etale groupoids and to obtain the correct notion of morphisms
one needs to take the category of fractions of the category of groupoids with respect to the
internal weak equivalences of groupoids. This gives us then the category of \'etendues with
isomorphism classes of geometric morphisms as arrows.

To obtain the correct 2-categorical structure, the construction of a category of fractions
needs to be refined to a bicategory of fractions. This was done by the first author in \cite{P1995},
with applications to \'etendues and stacks. A further application of this is given in \cite{MP}, to provide a new notion
of mapping between orbifolds. The resulting maps are the ones that work well for studying orbifold homotopy
theory, and they were called good maps in \cite{Adem} for instance.

In the current paper we study a weakly globular double categorical equivalent of this construction.
We could of course simply take $\Dbl(\bfC(W^{-1}))$.
This weakly globular double category inherits the following universal property from 
$\bfC(W^{-1})$. For any weakly globular double category $\bbD$, there
is an equivalence of categories
$$
\Hom_{\smWGDbl_{\ps},W}(\Dbl(\bfC),\bbD)\simeq\Hom_{\smWGDbl_{\ps}}(\Dbl(\bfC(W^{-1})),\bbD),
$$
where $\Hom_{\smWGDbl_{\ps}}$ is the category of pseudo-functors and vertical transformations.
Furthermore, $\Hom_{\smWGDbl_{\ps},W}(\Dbl(\bfC),\bbD)$ is the subcategory of pseudo-functors 
that send the horizontal arrows in $\Dbl(\bfC)$
that are related to $W$ (as defined in Section \ref{fracns} below) 
to pre-companions, and vertical transformations that respect 
this pre-companion structure as spelled out in Definition \ref{W-trafo}.
This equivalence is obtained by composition with the inclusion 
$\Dbl(\bfC)\rightarrow \Dbl(\bfC(W^{-1}))$ and this is a 
non-strict pseudo-functor.

This universal property is shared by all weakly globular double categories that are (vertically) 2-equivalent to $\Dbl(\bfC(W^{-1}))$,
and $\Dbl(\bfC(W^{-1}))$ does not have an obvious nice universal property 
with respect to strict functors and horizontal transformations.
So we will present a different weakly globular double category $\bfC\{W\}$, 
which is (vertically) 2-equivalent to $\Dbl(\bfC(W^{-1}))$,
but also has an interesting universal property with respect to strict functors 
and horizontal transformations.

The weakly globular double category $\CW$ has the property that 
the composition of pseudo functors 
$\xymatrix@1@C=4em{H\bfC\ar[r]^-\sim &\Dbl(\bfC)\ar[r]^-{\mathbf{\Dbl}(U)}&\Dbl(\bfC(W^{-1})\ar[r]^-\sim& \CW}$ is a strict functor, which we will denote by $\JC$. The universal property of $\Dbl(\bfC(W^{-1})$ 
given above translates into the following (vertical) universal property for $\CW$.

\begin{thmb}
Composition with 
the functor $\JC\colon H\bfC\rightarrow\CW$  gives rise to an equivalence of categories,
$$
\Hom_{\smWGDbl_{\ps}}(\CW,\bbD)\simeq\Hom_{\smWGDbl_{\ps},W}(H\bfC,\bbD).
$$
\end{thmb}

The new universal property of $\CW$, in the horizontal strict direction, is expressed in terms of companions and conjoints.
We write $\Hom_{\smWGDbl_{h},W}(H\bfC,\bbD)$ for the category of $W$-friendly functors 
and $W$-friendly horizontal transformations. 
$W$-friendly functors are functors that up to an invertible  horizontal
transformation send the arrows of $W$ to arrows that have a vertical companion and conjoint,
and the $W$-friendly transformations are those that respect this additional structure.
For further details, see Section \ref{fracns}, where we establish the following result.

\begin{thmc}
 Composition with $\calJ_{\bfC}\colon H\bfC\rightarrow \CW$ induces an equivalence of
 categories, $$\Hom_{\smWGDbl_h,W}(H\bfC,\bbD)\simeq\Hom_{\smWGDbl_h}(\CW,\bbD).$$
\end{thmc}

The connection between companions, conjoints and adjoints was explored in \cite{DPP-spans2}
where a double categorical version of the span-construction was introduced.
The double category $\bbSpan(C)$ freely adds horizontal companions and conjoints (which then become  adjoints) to 
the arrows of a category $\bfC$ (when viewed as a vertical double category).
However,  the resulting double categories in that paper are not weakly globular.
The beauty of the current construction is that it works within a 2-category that models 
weak 2-categories, but has additional structure available.

Another advantage of the double category $\CW$ over the bicategory $\bfC(W^{-1})$
is that it has small hom-sets, in the horizontal and vertical directions, and in the cells.
So we avoid the size-issues that we normally face when dealing with categories or bicategories of fractions
due to the presence of spans in the homs.

\subsection{Outline of the paper} The organization of this paper is as follows: 
in Section \ref{specialcase} we spell out the results by
Power \cite{pow} as they apply to a special class of pseudo-functors from $\dop$ to $\Cat$. 
In Section \ref{rigidification}
we particularize this to a class of pseudo-functors constructed out of Tamsamani weak 2-categories: 
this yields a rigidification result for the latter. In Section \ref{wglobular2} 
we introduce weakly globular double categories and provide
the necessary background and notation for double categories. We also provide a pseudo-inverse 
 $D\colon(\Ta)_{\ps}\rightarrow\NHom$ called dicretization, for the rigidification functor.
Section \ref{bicat} provides the necessary background on the relationship between bicategories and
Tamasamani weak 2-categories from \cite{lp}. We also provide an explicit description of
the fundamental bicategory $\Bic(\bbD)$ of a weakly globular double category $\bbD$, as well as of the associated
double category $\Dbl(\calB)$, as a double category of marked paths in a bicategory $\calB$.

In Section \ref{ccqu} we review the double categorical notion of companions and conjoints and show how they are related to
quasi units in bicategories under the functors $\Dbl$ and $\Bic$. This section also introduces the
category $\Comp(\bbD)$ and the double category $\bbComp(\bbD)$  of companions in a double category $\bbD$.
In Section \ref{CW} we introduce the weakly globular category of fractions $\CW$ and in Section \ref{fracns}
we describe its universal properties.

Finally, in Section \ref{grpdl} we introduce the subcategory of groupoidal weakly globular double categories
and show that they form an algebraic model of 2-types.

\medskip

\textbf{Acknowledgements.} The first author is supported by a Marie Curie
International Reintegration Grant no. 256341. She would also like to thank the
Department of Mathematics and Statistics of Dalhousie University for their
hospitality during a visit in June 2012. The second author is supported by an NSERC Discovery grant.
She would also like to thank the Department of Mathematics of the University of Leicester for their 
hospitality during a visit in August 2012.

\bigskip

\section{A special case of the strictification of pseudo-functors}\label{specialcase}

In this section we first recall a classical result due to Power \cite{pow} on
strictification of pseudo-functors. We then apply this technique to a
particular class of pseudo-functors from $\dop$ to $\Cat$.
\subsection{Strictification of pseudo-functors}\label{strpseudo}

As explained in \cite[4.2]{pow}, the functor $2$-category $[\dop,\Cat]$ is
$2$-monadic over $[\;|\dop|, \Cat]$, where $|\dop|$ is the set of objects of
$\dop$.

Let $U:[\dop,\Cat]\rw [\,|\dop|, \Cat]$ be the forgetful functor,
$(UX)_{n}=X_{n}$ for all $[n]\in\dop, X\in [\dop,\Cat]$. Then its left adjoint
$F$ is given on objects by
\begin{equation*}
(FH)_{n}=\coprod_{[m]\in |\dop|} \dop([m],[n])\times H_{m},
\end{equation*}
for $H\in [\,|\dop|, \Cat], [n]\in |\dop|$. If $T$ is the monad corresponding to
the adjunction $F\dashv U$ then
\begin{equation*}
(TH)_{n}=\coprod_{[m]\in |\dop|} \dop([m],[n])\times H_{m},
\end{equation*}
for $H\in [\,|\dop|, \Cat]$, and $[n]\in |\dop|$.

A pseudo-$T$-algebra is given by $H\in [\,|\dop|, \Cat]$, functors
$$h_{n}:\prod_{[m]\in |\dop|} \dop([m],[n])\times H_{m}\rw H_{n}$$ and additional data
as described in \cite[4.2]{pow}. This amounts precisely to a pseudo-functor
from $\dop$ to $\Cat$ and the $2$-category $\PsTalg$ of pseudo-$T$-algebras corresponds
to the $2$-category $\Ps [\dop,\Cat]$ of pseudo-functors, pseudo-natural
transformations and modifications.

The general strictification result proved in \cite[3.4]{pow}, when applied to
this case, yields that every pseudo-functor from $\dop$ to Cat is equivalent,
in $\Ps [\dop,\Cat]$, to a $2$-functor.

The construction given in \cite{pow} is as follows. Given a pseudo-$T$-algebra
as above, factor $h:TH \rw H$ as $TH \supar{r}L\supar{g}H$ with $r_{n}$
bijective on objects and $g_{n}$ fully faithful for each $[n]\in \dop$. Then it
is shown in \cite{pow}  that it is possible to give a strict $T$-algebra
structure $TL\rw L$ such that $(g,Tg)$ is an equivalence of pseudo-$T$-algebras.

\begin{remark}\rm\label{rem1}
Since $(g,Tg)$ is an equivalence of pseudo-$T$-algebras, $g_{n}$ is an
equivalence of categories for every $[n]\in \dop$. In fact, by definition there
is a map $(g', Tg')$ with invertible $2$-cells $\alpha: (g,
Tg)(g',Tg')\Longrightarrow \mathrm{id}$ and $\beta: (g',
Tg')(g,Tg)\Longrightarrow \mathrm{id}$. A $2$-cell in $\PsTalg$ amounts to a
$2$-cell in $[\,|\dop|, \Cat]$ satisfying the condition of \cite[2.6]{pow}.
Since the $2$-cells in $[\,|\dop|, \Cat]$  are modifications, 
this implies that for each $[n]\in \dop$ there are natural transformations
$\mathrm{id}\cong g_{n}g'_{n}$, and $g'_{n}g_{n}\cong \mathrm{id}$; that is,
$g_{n}$ is an equivalence of categories for each $n$.
\end{remark}
\subsection{A special case of strictification}\label{specialstrict}

We now apply the technique of Section \ref{strpseudo} to a class of
pseudo-functors which, as we will see in Section \ref{rigtam2}, arises from
Tamsamani weak 2-categories.

\begin{lemma}\label{special.2}
    Let $H \in \Ps[\dop,\Cat]$ be a pseudo-functor such that $H_{0}$ is a
    discrete category and, for each $n\geq 2$, $H_{n}\cong\pro{H_1}{H_{0}}{n}$.
    Let $U:\Ps[\dop,\Cat]\rw[|\dop|,\Cat]$ be the forgetful functor,
    $(UH)_{n}=H_{n}$ for all $n\geq 0$, and let $T$ be the monad on $[|\dop|,\Cat]$ as in Section
    \ref{strpseudo}. Then
    \begin{itemize}
      \item [a)] The pseudo-$T$-algebra corresponding to $H$ has structure map
      $h:TUH\rw UH$ given as follows: for each $k\geq 0$,
\begin{equation*}
    (TUH)_{k}=\underset{[n]\in\Delta}{\coprod}\Delta([k],[n])
      \times
      H_{n}=\underset{[n]\in\Delta}{\coprod}\;\;\underset{\Delta([k],[n])}{\coprod}H_n.
\end{equation*}
      For $n\geq 0$ and $f\in\Delta([k],[n])$, let
      $i_n:\underset{\Delta([k],[n])}{\coprod}H_n\rw(TUH)_k$ and $j_f:H_n\rw\underset{\Delta([k],[n])}{\coprod}
      H_n$ be the coproduct injections. Then $h_k i_n j_f=H(f)$.
      \medskip
      \item [b)] There are functors $\pt'_0, \pt'_1:(TUH)_1\rightrightarrows(TUH)_0$ making
      the following diagram commute:
      \begin{equation}\label{special.eq1}
        \xymatrix{
        (TUH)_1  \ar[r]^{h_1}\ar@<1ex>[d]^{\pt'_1}\ar@<-1ex>[d]_{\pt'_0} & H_1\ar@<1ex>[d]^{\pt_1}\ar@<-1ex>[d]_{\pt_0}
         \\
        (TUH)_0  \ar[r]^{h_0} & H_0
        }
      \end{equation}
      that is, $\pt_ih_1=h_0\pt'_i,$ for $i=0,1$.
      \item [c)] For each $k\geq 2$, $(TUH)_k\cong\pro{(TUH)_1}{(TUH)_0}{k}$.
      \medskip
      \item [d)] For each $k\geq 2$, the morphism $h_k:(TUH)_k\rw H_k$ is
      $h_k=(h_1,\ldots,h_1)$.
    \end{itemize}
\end{lemma}
\begin{proof}
\begin{itemize}
  \item [a)] From the general correspondence between pseudo-$T$-algebras and
  pseudo-functors, the pseudo-$T$-algebra corresponding to $H$ has structure
  map $h:TUH\rw UH$ as stated. Recalling that, if $X$ is a set and $\calC$ is a
  category, $X\times \calC\cong\underset{X}{\coprod}\calC$, we have
  \begin{equation*}
    (TUH)_k=\underset{[n]\in\dop}{\coprod}\dop([n],[k])\times
    H_n=\underset{[n]\in\Delta}{\coprod}\Delta([k],[n])\times
    H_n\cong\underset{[n]\in\Delta}{\coprod}\;\;\underset{\Delta([k],[n])}{\coprod}H_n.
  \end{equation*}
  \item [b)] Let $\delta_i:[0]\rw[1]$, $\delta_i(0)=0$, $\delta_i(1)=i$, for $i=0,1$. For
  $f\in\Delta([1],[n])$ let $j_f:H_n\rw\underset{\Delta([1],[n])}{\coprod}H_n$ and 
  $i_n\colon\underset{\Delta([1],[n])}{\coprod}H_n
  \rw \underset{[n]\in\Delta}{\coprod}\;\;\underset{\Delta([1],[n])}{\coprod}H_n$ be the
  coproduct injections. Let $\pt'_i:(TUH)_1\rw(TUH)_0$ be the functors determined by
  \begin{equation}\label{special.eq2}
    \pt'_i i_n j_f=i_n j_{f \delta_i}
  \end{equation}
  From a), we have
  $$h_0\pt'_i i_n j_f=h_0 i_n j_{f \delta_i}=H(f \delta_i)\quad \text{and}\quad \pt_i h_1 i_n j_f=H(\delta_i)
  H(f)\,.$$
  Since $H\in \Ps[\dop,\Cat]$ and $H_0$ is discrete, it follows that $H(f\delta_i)=H(\delta_i)H(f)$, 
  so that, from above, $h_0\pt'_i i_n j_f=\pt_i h_1 i_n j_f$
  for each $[n]\in\Delta$, $f\in\Delta([1],[n])$. We conclude that
  $h_0\pt'_i=\pt_ih_1$.
  \medskip
  \item [c)] For each $k\geq 2$, $[k]$ is the colimit in $\Delta$ of the
  diagram
  \begin{equation*}
    \xymatrix{
    [1] && [1] && && [1]\\
    &[0]\ar[ul]^{0} \ar[ur]_1 && [0]\ar[ul]^0 & \cdots & [0]\ar[ur]_1
    }
  \end{equation*}
    that is,
    $[k]=[1]\underset{[0]}{\coprod}\overset{k}{\cdots}\underset{[0]}{\coprod}[1]$.
In fact, it is easy to check by direct computation that there is a pushout
    in $\Delta$:
    \begin{equation*}
        \xymatrix{
        [0]\ar[r]^0\ar[d]_1 & [k-1]\ar[d]^p\\
        [1]\ar[r]_q &[k]
        }
    \end{equation*}
    where $q(i)=i$, for $i=0,1$, and $p(t)=t+1$, for $t=0,\ldots,k-1$.
In particular, $[2]=[1]\underset{[0]}{\coprod}[1]$. Inductively,
    if $[k-1]=[1]\underset{[0]}{\coprod}\overset{k-1}{\ldots}\underset{[0]}{\coprod}[1]$ then,
    from the above pushout,
    $[k]=[k-1]\underset{[0]}{\coprod}[1]=[1]\underset{[0]}{\coprod}\overset{k}{\ldots}\underset{[0]}{\coprod}[1]$,
    as claimed. It follows that there is a bijection, for $k\geq 2$
    \begin{equation*}
        \Delta([k],[n])\cong\pro{\Delta([1],[n])}{\Delta([0],[n])}{k}.
    \end{equation*}
From the proof of b), the
    functors $\pt_i:(TUH)_1\rw(TUH)_0$ are determined by the functors 
$(\ovl{\delta}_i,\id):\Delta([1],[n])\times H_n\rw\Delta([0],[n])\times H_n$ 
where $\ovl{\delta}_i(f)=f{\delta}_i$ for $f\in\Delta([1],[n])$. Hence, from above, we
    obtain
    \begin{equation*}
        \begin{split}
            &\quad \pro{(TUH)_1}{(TUH)_0}{k}= \\
             & =\pro{(\underset{[n]\in\Delta}{\coprod}\Delta([1],[n])\times
             H_n)}{(\underset{[n]\in\Delta}{\coprod}\Delta([0],[n])\times
             H_n)}{k}\\
             & \cong \underset{[n]\in\Delta}{\coprod}\pro{(\Delta([1],[n])\times
             H_n)}{(\Delta([0],[n])\times
             H_n)}{k}\\
             &=\underset{[n]\in\Delta}{\coprod}(\pro{\Delta([1],[n])}{\Delta([0],[n])}{k})\times(\pro{H_n}{H_n}{k})\\
             &=\underset{[n]\in\Delta}\coprod\Delta([k],[n])\times H_n=(TUH)_k.
         \end{split}
    \end{equation*}

    \bigskip
    \item [d)] From a), $h_k i_n j_f=H(f)$ for $f\in\Delta([k],[n])$, $n>0$. Let $f$
    correspond to $(\delta_1,\ldots,\delta_k)$ in the isomorphism
    $\Delta([k],[n])\cong\pro{\Delta([1,[n])}{\Delta([0,[n])}{k}$. Then
    $j_f=(j_{\delta_1},\ldots,j_{\delta_k})$. Since $H_k\cong\pro{H_1}{H_0}{k}$, $H(f)$
    corresponds to $(H(\delta_1),\ldots,H(\delta_k))$ with $p_i H(f)=H(\delta_i)$. Then, for
    all $f\in\Delta([k],[n])$, $n>0$ we have
    \begin{equation*}
        \begin{split}
            & h_k i_n j_f=H(f)=(H(\delta_1),\ldots,H(\delta_k))=(h_1 i_n j_{\delta_1},\ldots,h_1 i_n j_{\delta_k})= \\
            & =(h_1,\ldots,h_1) i_n(j_{\delta_1},\ldots,j_{\delta_k})=(h_1,\ldots,h_1)
            i_n j_f.
        \end{split}
    \end{equation*}
    We conclude that $h_k=(h_1,\ldots,h_1)$
\end{itemize}
\end{proof}

\begin{proposition}\label{special.3}
    Let $H\in \Ps[\dop,\Cat]$ be such that
    $H_n\cong\pro{H_1}{H_0}{n}$ for each $n\geq 2$ and $H_0$ is discrete. Let $L\in[\dop,\Cat]$ be
    the strictification of $H$ as in Section \ref{strpseudo}. Then

    \begin{itemize}
      \item [a)] There is a morphism $g:L\rw H$ in  $\Ps[\dop,\Cat]$ such that,
      for each $k\geq 0$, $g_k$ is an equivalence of categories.
      \medskip
      \item [b)] $L_k\cong\pro{L_1}{L_0}{k}$ for all $k\geq 2$.
      \medskip
      \item [c)] The functor $g_0:L_0\rw H_0$ induces equivalences of
      categories\\ $\pro{L_1}{L_0}{k}\simeq\pro{L_1}{H_0}{k}$ for all $k\geq 2$.
    \end{itemize}
\end{proposition}
\begin{proof}
\begin{itemize}
  \item [a)] This follows directly from \cite{pow} (see Remark \ref{rem1}).
  \medskip
  \item [b)] Let $h:TUH\rw UH$ be as in Lemma \ref{special.2}. As recalled in section \ref{strpseudo}, 
  factorize $h=gr$
  so that, for each $i\geq 0$, $h_{i}$ factorizes as  $(TUH)_i\supar{r_i} L_i\supar{g_i} H_i$ with $r_i$
  bijective on objects and $g_i$ fully faithful. As explained in
  \cite{pow}, the $g_i$ are in fact equivalences.

  Since the bijective on objects and fully faithful functors form a
  factorization system in $\Cat$, the commutativity of (\ref{special.eq1})
  implies that there are functors $d_0,d_1:L_1\rw L_0$ such that the following
  commutes:
  \begin{equation*}
    \xymatrix{
    (TUH)_1 \ar[r]^(.6){r_1}\ar@<-1ex>[d]_{\pt'_0}\ar@<1ex>[d]^{\pt'_1} & L_1
    \ar[r]^(.5){g_1}\ar@<-1ex>[d]_{d_0}\ar@<1ex>[d]^{d_1} & H_1
    \ar@<-1ex>[d]_{\pt_0}\ar@<1ex>[d]^{\pt_1}\\
    (TUH)_0 \ar[r]_(.6){r_0} & L_0 \ar[r]_(.5){g_0} & H_0
    }
  \end{equation*}
    that is, $\pt_i r_1=r_0\pt'_i$, $\pt_ig_1=g_0d_i$ for $i=0,1$. By Lemma
    \ref{special.2} d) , $h_k$ factorizes as
    \begin{equation}\label{factorization}
    \begin{split}
        & (TUH)_k=\pro{(TUH)_1}{(TUH)_0}{k}\supar{(r_1,\ldots,r_1)} \pro{L_1}{L_0}{k} \\
        & \supar{(g_1,\ldots,g_1)}\pro{H_1}{H_0}{k}\cong H_{k}.
    \end{split}
    \end{equation}
    Since $r_0,r_1$ are bijective on objects, so is $(r_1,\ldots,r_1)$. Since $g_0,g_1$ are fully faithful, so is
    $(g_1,\ldots,g_1)$. Hence (\ref{factorization}) is the factorization of $h_k$ and
    we conclude that $L_k\cong\pro{L_1}{L_0}{k}$.
  \item [c)] Since $H_1\simeq L_1$ and $H_0$ is discrete,
  $H_k\cong\pro{H_1}{H_0}{k}\simeq\pro{L_1}{H_0}{k}$. On the other hand, $H_k\simeq L_k\cong
  \pro{L_1}{L_0}{k}$. In conclusion,
  $\pro{L_1}{H_0}{k}\\ \simeq\pro{L_1}{L_0}{k}$.
\end{itemize}
\end{proof}

\begin{remark}\label{rem2add}\rm
    Recall that the factorization of any functor $F:\bfC\rw\bfD$ as the
    composite $\bfC\supar{S}\bfE\supar{T}\bfD$ with $S$ bijective on objects
    and $T$ fully faithful is done as follows. Consider the pullbacks of sets
\begin{equation*}
  \xymatrix@R=2em@C=4em{
     \bfE_1 \ar[r]^-{(\tilde{d}_0,\tilde{d}_1)} \ar_{T_1}[d] &  \bfC_0\times\bfC_0 \ar^{F_0\times F_0}[d]\\
    \bfD_1 \ar_-{(d_0,d_1)}[r] & \bfD_0\times\bfD_0
  }
\end{equation*}
   where $d_0,d_1$ are the source and target map in the category $\bfD$. It is
   easy to see that there is a category $\bfE$ with objects $\bfE_0=\bfC_0$ and
   $\bfE_1$ and source and target maps $\tilde{d}_0$ and $\tilde{d}_1$ as in the pullback diagram above.
 Further, there are functors $S:\bfC\rw\bfE$, $T:\bfE\rw\bfD$ with $S_0=\id$,
   $T_0=F_0$, and $S_1$ determined by $F_1:\bfC_1\rw\bfD_1$ and
   $(d'_0,d'_1):\bfC_1\rw\bfC_0\times\bfC_0$. Hence, in the notation of
   Proposition \ref{special.3}, we have
\begin{equation*}
    \begin{split}
       L_{00} &= (T U H)_{00}=\underset{[n]\in\Delta}{\coprod}\Delta([0],[n])\times H_{n0} \\
       L_{10}  & = (T U H)_{10}=\underset{[n]\in\Delta}{\coprod}\Delta([1],[n])\times H_{n0}
     \end{split}
\end{equation*}
   while $L_{01}$ and $L_{11}$ are given by the following pullbacks:
\begin{equation*}
  \xymatrix@R=2.1em@C=4em{
     L_{11} \ar[r]\ar_{g_{11}}[d] &  (T U H)_{10}\times (T U H)_{10} \ar^{h_{10}\times h_{10}}[d]\\
    H_{11} \ar_{(d_0,d_1)}[r] & H_{10}\times H_{10}
  }
\end{equation*}
\begin{equation*}
  \xymatrix@R=2.1em@C=4em{
     L_{01} \ar[r]\ar_{g_{01}}[d] &  (T U H)_{00}\times (T U H)_{00} \ar^{h_{00}\times h_{00}}[d]\\
    H_{00} \ar_{(\id,\id)}[r] & H_{00}\times H_{00}
  }
\end{equation*}
\end{remark}

\section{Rigidification of Tamsamani weak 2-categories}\label {rigidification}

In this section we associate to a Tamsamani weak 2-category a pseudo-functor in
$\Ps[\dop,\Cat]$ satisfying the hypotheses of Proposition \ref{special.3}.
Hence the strictification result of Proposition \ref{special.3} gives rise to a
functor in $[\dop,\Cat]$ which is suitably equivalent to the original Tamsamani weak 2-category.

As we will see in Section \ref{wglobular2}, the resulting strict functor is in fact the horizontal
nerve of a special type of double category, which we call weakly globular.

We begin with some background on Tamsamani weak 2-categories (Section
\ref{tam2}) as well as on a general categorical technique known as 'transport
of structure along an adjunction' (Section \ref{transp}). The latter is used in
Section \ref{rigtam2} to associate to a Tamsamani weak 2-category a
pseudo-functor satisfying the hypotheses of Proposition \ref{special.3}.

\subsection{Tamsamani weak 2-categories}\label{tam2}

We first recall the notion of Segal map. Let $\bfC$ be a category with
pullbacks and let $X\in[\dop,\bfC]$. For each $k\geq 2$ we denote by
\begin{equation*}
    \pro{X_1}{X_0}{k}
\end{equation*}
the limit of the diagram
\begin{equation*}
    X_1\supar{\pt_1}X_0\suparle{\pt_0}\cdots\supar{\pt_1}X_0\suparle{\pt_0}X_1\;.
\end{equation*}
For each $1\leq j\leq k$ let $\nu_j:X_k\rw X_1$ be induced by the map  $[1]\rw[k]$ in
$\Delta$ sending $0$ to $j-1$ and $1$ to $j$. Then the following
diagram commutes:
\begin{equation*}
    \xy
    0;/r.80pc/:
    (0,0)*+{X_k}="1";
    (-7,-5)*+{X_1}="2";
    (-2,-5)*+{X_1}="3";
    (9,-5)*+{X_1}="4";
    (-10,-9)*+{X_0}="5";
    (-5,-9)*+{X_0}="6";
    (0,-9)*+{X_0}="7";
    (6,-9)*+{X_0}="8";
    (12,-9)*+{X_0}="9";
    (3,-5)*+{\cdots}="10";
    (3,-9)*+{\cdots}="11";
    {\ar_{\nu_1}"1";"2"};
    {\ar^{\nu_2}"1";"3"};
    {\ar^{\nu_k}"1";"4"};
    {\ar_{\pt_1}"2";"5"};
    {\ar^{\pt_0}"2";"6"};
    {\ar^{\pt_1}"3";"6"};
    {\ar^{\pt_0}"3";"7"};
    {\ar_{\pt_1}"4";"8"};
    {\ar^{\pt_0}"4";"9"};
    \endxy
\end{equation*}
By definition of limit, there is a unique map
\begin{equation*}
    \eta_k:X_k\rw\pro{X_1}{X_0}{k}
\end{equation*}
such that $\pr_j\,\eta_k=\nu_j$, where $\pr_j$ is the $j^{th}$ projection.

The maps $\eta_k$ are called Segal maps and they play an important role in
Tamsamani's model of weak 2-categories. They can also be used to characterize nerves of internal
categories: a simplicial  object in $\calC$ is the nerve of an internal category
in $\calC$ if and only if the Segal maps are isomorphisms for all $k\geq 2$.

\begin{definition}\rm{[8]\;}\label{tam1}
The category $\Ta$ of {\em Tamsamani weak $2$-categories} is the full subcategory
of $[\dop,\Cat]$ whose objects $X$ are such that $X_0$ is discrete and, for all
$k\geq 2$, the Segal map $\eta_k:X_k\rw \pro {X_1}{X_0}{k}$ is a categorical
equivalence.
\end{definition}

Let $\pi_0:\Cat\rw\Set$ associate to a category the set of isomorphism classes
of its objects. The functor  $\pi_0$ induces a functor
${\pi}_0^*:[\dop,\Cat]\rw[\dop,\Set]$, $(\pi_0^* X)_n=\pi_0 X_n$. If $X\in\Ta$,
then ${\pi}_0^* X$ is the nerve of a category. In fact, since $\pi_0$ sends
categorical equivalences to isomorphisms and preserves fiber products over
discrete objects, for all $k\geq 2$ we have

\begin{equation*}
    \pi_0X_{k}\cong\pi_0(\pro{X_{1}}{X_{0}}{k})\cong\pro{\pi_0 X_{1}}{\pi_0
    X_{0}}{k}.
    \end{equation*}
\medskip

We write $\Pi_0 X$ for the category whose nerve is ${\pi}_0^* X$. This defines a
functor
\begin{equation*}
    \Pi_0: \Ta \rw \Cat\,.
\end{equation*}
Given $X\in\Ta$ and $a,b\in X_0$ let $X_{(a,b)}$ be the full subcategory of
$X_1$ whose objects $z$ are such that $d_0 z=a$ and $d_1 z=b$. By considering
the functor $(d_0, d_1):X_1 \rw X_0 \times X_0$, since $X_0$ is discrete, we
obtain a coproduct decomposition $X_1=\coprod_{a,b\in X_0} X_{(a,b)}$.

\begin{definition}\label{tam.2}\rm A morphisms $F:X\rw Y$ in $\Ta$ is a
{\em 2-equivalence} if, for all $a, b \in X_0$, $F_{(a,b)}: X_{(a,b)} \rw Y_{(Fa,
Fb)}$ and $\Pi_0 F$ are categorical equivalences.
\end{definition}

\begin{remark}\label{s2.rem2}\rm
    Notice that if a morphism in $\Ta$ is a levelwise equivalence of
    categories, it is in particular a 2-equivalence.
\end{remark}
\subsection{Transport of structure along an adjunction}\label{transp}

We now recall a general categorical property, known as transport of structure
along an adjunction, with one of its applications.

\begin{theorem}\rm{\cite[Theorem 6.1]{lk}}\em\label{s2.the1}
    Given an equivalence $\;\eta,\;\vep : f \dashv f^* : A\rw B$ in the complete and locally small
    2-category $\calA$, and an algebra $(A,a)$ for the monad $T=(T,i,m)$ on $\calA$, the
    equivalence enriches to an equivalence
\begin{equation*}
  \eta,\vep:(f,\ovll{f})\vdash (f^*,\ovll{f^*}):(A,a)\rw(B,b,\hat{b},\ovl{b})
\end{equation*}
in $\PsTalg$, where $\hat{b}=\eta$, $\;\ovl{b}=f^* a\cdot T\vep \cdot
Ta\cdot T^2 f$, $\;\ovll{f}=\vep^{-1} a\cdot Tf$, $\;\ovll{f^*}=f^* a\cdot
T\vep$.
\end{theorem}
Let $\eta',\vep':f'\vdash f'^{*}:A'\rw B'$ be another equivalence in $\calA$ and
let $(B',b',\hat{b'},\ovl{b'})$ be the corresponding pseudo-$T$-algebra as in
Theorem \ref{s2.the1}. Suppose $g:(A,a)\rw(A',a')$ is a morphism in $\calA$ and
$\gamma$ is an invertible 2-cell in $\calA$
\begin{equation*}
  \xy
    0;/r.10pc/:
    (-20,20)*+{B}="1";
    (-20,-20)*+{B'}="2";
    (20,20)*+{A}="3";
    (20,-20)*+{A'}="4";
    {\ar_{f^*}"3";"1"};
    {\ar_{h}"1";"2"};
    {\ar^{f'^*}"4";"2"};
    {\ar^{g}"3";"4"};
    {\ar@{=>}^{\gamma}(0,3);(0,-3)};
\endxy
\end{equation*}
Let $\ovl{\gamma}$ be the invertible 2-cell given by the following pasting:
\begin{equation*}
    \xy
    0;/r.15pc/:
    (-40,40)*+{TB}="1";
    (40,40)*+{TB'}="2";
    (-40,-40)*+{B}="3";
    (40,-40)*+{B'}="4";
    (-20,20)*+{TA}="5";
    (20,20)*+{TA'}="6";
    (-20,-20)*+{A}="7";
    (20,-20)*+{A'}="8";
    {\ar^{Th}"1";"2"};
    {\ar_{b}"1";"3"};
    {\ar^{b'}"2";"4"};
    {\ar_{h}"3";"4"};
    {\ar_{Tg}"5";"6"};
    {\ar^{}"5";"7"};
    {\ar_{}"6";"8"};
    {\ar^{g}"7";"8"};
    {\ar_{Tf^*}"5";"1"};
    {\ar^{f^*}"7";"3"};
    {\ar^{Tf'^*}"6";"2"};
    {\ar^{f'^*}"8";"4"};
    {\ar@{=>}^{(T\gamma)^{-1}}(0,33);(0,27)};
    {\ar@{=>}^{\gamma}(0,-27);(0,-33)};
    {\ar@{=>}^{\ovll{f'^*}}(30,3);(30,-3)};
    {\ar@{=>}^{\ovll{f^*}}(-30,3);(-30,-3)};
\endxy
\end{equation*}
Then it is not difficult to show that
$(h,\ovl{\gamma}):(B,b,\hat{b},\ovl{b})\rw(B',b',\hat{b'},\ovl{b'})$ is a
pseudo-$T$-algebra morphism.

The following fact is essentially known and, as sketched in the proof below, it
is an instance of Theorem \ref{s2.the1}

\begin{lemma}\label{s4.lem1}
    Let $\calC$ be a small 2-category, $F,F':\calC\rw\Cat$ be 2-functors, $\alpha:F\rw F'$
    a 2-natural transformation. Suppose that, for all objects $C$ of $\calC$, the
    following conditions hold:
\begin{itemize}
  \item [i)] $G(C),\;G'(C)$ are objects of $\Cat$ and there are adjoint equivalences of
  categories $\mu_C\vdash\eta_C$, $\mu'_C\vdash\eta'_C$,
\begin{equation*}
  \mu_C:G(C)\;\rightleftarrows\;F(C):\eta_C\qquad\qquad
  \mu'_C:G'(C)\;\rightleftarrows\;F'(C):\eta'_C,
\end{equation*}
  \item [ii)] there are functors $\beta_C:G(C)\rw G'(C),$
  \item [iii)] there is an invertible 2-cell
\begin{equation*}
  \gamma_C:\beta_C\eta_C\Rar\eta'_C\alpha_C.
\end{equation*}
\end{itemize}
Then
\begin{itemize}
  \item [a)] There exists a pseudo-functor $G:\calC\rw\Cat$ given on objects by $G(C)$,
  and pseudo-natural transformations $\eta:F\rw G$, $\mu:G\rw F$ with
  $\eta(C)=\eta_C$, $\mu(C)=\mu_C$; these are part of an adjoint equivalence
  $\mu\vdash\eta$ in the 2-category $\Ps[\calC,\Cat]$.
  \item [b)] There is a pseudo-natural transformation $\beta:G\rw G'$ with
  $\beta(C)=\beta_C$ and an invertible 2-cell in $\Ps[\calC,\Cat]$,
  $\gamma:\beta\eta\Rar\eta\alpha$ with $\gamma(C)=\gamma_C$.
\end{itemize}
\end{lemma}
\begin{proof}
Recall \cite{pow} that the functor 2-category $[\calC,\Cat]$ is 2-monadic over
$[|\calC|,\Cat]$, where $|C|$ is the set of objects in $\calC$. Let $T$ be the
2-monad; then the pseudo-$T$-algebras are precisely the pseudo-functors from
$\calC$ to $\Cat$. Let
\begin{equation*}
  \calU:\PsTalg\equiv \Ps[\calC,\Cat]\rw[|\calC|,\Cat]
\end{equation*}
be the forgetful functor.

Then the adjoint equivalences $\mu_C\vdash\eta_C$ amount precisely to an
adjoint equivalence in $[|\calC|,\Cat]$, $\;\mu_0\vdash\eta_0$,
$\;\mu_0:G_0\;\;\rightleftarrows\;\;\calU F:\eta_0$ where $\;G_0(C)=G(C)$ for
all $C\in |\calC|$. By Theorem \ref{s2.the1}, this equivalence enriches to an
adjoint equivalence $\mu\vdash\eta$ in $\Ps[\calC,\Cat]$
\begin{equation*}
  \mu:G\;\rightleftarrows\; F:\eta
\end{equation*}
between $F$ and a pseudo-functor $G$; it is $\calU G=G_0$, $\;\calU\eta=\eta_0$,
$\;\calU\mu=\mu_0$; hence on objects $G$ is given by $G(C)$, and
$\eta(C)=\calU\eta(C)=\eta_C$, $\;\mu(C)=\calU\mu(C)=\mu_C$.

Let $\nu_C:\id_{G(C)}\Rar\eta_C\mu_C$ and $\vep_C:\mu_C\eta_C\Rar\id_{F(C)}$ be
the unit and counit of the adjunction $\mu_C\vdash\eta_C$. From Theorem
\ref{s2.the1}, given a morphism $f:C\rw D$ in $\calC$, it is
\begin{equation*}
  G(f)=\eta_D F(f)\mu_C
\end{equation*}
and we have natural isomorphisms:
\begin{align*}
   & \eta_f : G(f)\eta_C=\eta_D F(f)\mu_C\eta_C\overset{\eta_D F(f)\vep_C}{=\!=\!=\!=\!\Rar} \eta_D F(f)\\
   & \mu_f : F(f)\mu_C\overset{\nu_{F(f)}\mu_C}{=\!=\!=\!\Rar}\mu_D\eta_D F(f)\mu_C=\mu_D
   G(f).
\end{align*}
Also, the natural isomorphism
\begin{equation*}
  \beta_f: G'(f)\beta_C\Rar\beta_D G(f)
\end{equation*}
is the result of the following pasting
\begin{equation*}
    \xy
    0;/r.15pc/:
    (-40,40)*+{G(C)}="1";
    (40,40)*+{G'(C)}="2";
    (-40,-40)*+{G(D)}="3";
    (40,-40)*+{G'(D)}="4";
    (-20,20)*+{F(C)}="5";
    (20,20)*+{F'(C)}="6";
    (-20,-20)*+{F(D)}="7";
    (20,-20)*+{F'(D)}="8";
    {\ar^{\beta_C}"1";"2"};
    {\ar_{G(f)}"1";"3"};
    {\ar^{G'(f)}"2";"4"};
    {\ar_{\beta_D}"3";"4"};
    {\ar^{\alpha_C}"5";"6"};
    {\ar^{F(f)}"5";"7"};
    {\ar_{F'(f)}"6";"8"};
    {\ar_{\alpha'_D}"7";"8"};
    {\ar^{}"5";"1"};
    {\ar^{}"7";"3"};
    {\ar^{}"6";"2"};
    {\ar^{}"8";"4"};
    {\ar@{=>}^{\gamma_C}(0,33);(0,27)};
    {\ar@{=>}^{\gamma_D^{-1}}(0,-27);(0,-33)};
    {\ar@{=>}^{\eta_f'}(30,3);(30,-3)};
    {\ar@{=>}^{\eta_f}(-30,3);(-30,-3)};
\endxy
\end{equation*}
\end{proof}
\subsection{Rigidifying Tamsamani weak 2-categories}\label{rigtam2}

We now use the results of the previous sections to associate to a Tamsamani
weak 2-category a pseudo-functor satisfying the hypotheses of Proposition
\ref{special.3} and then obtain a rigidification result.

\begin{proposition}\label{s4.pro1}
    There is a functor
\begin{equation*}
  S:\Ta\rw \Ps[\Delta^{op},\Cat].
\end{equation*}
which associates to an object $ X$ of $\Ta$ a pseudo-functor $S
X\in \Ps[\dop,\Cat]$ with
\begin{equation*}
  (S X)_n=
  \begin{cases}
     X_1\tiund{ X_0}\overset{n}{\cdots}\tiund{ X_0} X_1 & n\geq 2, \\
     X_1 & n=1, \\
     X_0 & n=0.
  \end{cases}
\end{equation*}
To a morphism $F: X\rw X'$, $S$ associates a pseudo-natural transformation
$\beta(F):S X \rw S X'$ with
\begin{equation*}
  \beta(F)_n=
  \begin{cases}
    (F_1,\ldots,F_1) & n\geq 2, \\
    F_1 & n=1, \\
    F_0 & n=0.
  \end{cases}
\end{equation*}
Further, there is a pseudo-natural transformation $\alpha:S X \rw  X$ which is a
levelwise categorical equivalence.
\end{proposition}

\begin{proof}
We apply Lemma \ref{s4.lem1} to the case where $\calC=\dop$, considered as a
2-category with identity 2-cells; let $ X:\dop\rw\Cat$ be an object of $\Ta$.
By definition, for each $n\geq 2$ there is an equivalence of categories $
X_n\simeq X_1\tiund{ X_0}\overset{n}{\cdots}\tiund{ X_0} X_1$. We can always
choose this equivalence to be an adjoint equivalence; thus let $\eta_n: X_n\rw
X_1\tiund{ X_0}\overset{n}{\cdots}\tiund{ X_0} X_1$ be the Segal map and
$\mu_n$ its left adjoint. By Lemma \ref{s4.lem1}, we deduce that there is a
pseudo-functor $S X\in \Ps[\dop,\Cat]$ with
\begin{equation*}
  (S X)_n=
  \begin{cases}
     X_1\tiund{ X_0}\overset{n}{\cdots}\tiund{ X_0} X_1 & n\geq 2, \\
     X_1 & n=1, \\
     X_0 & n=0.
  \end{cases}
\end{equation*}
Suppose $F: X\rw X'$ is a morphism in $\Ta$ and let $\beta_n:(S X)_n \rw (S
X')_n$ be
\begin{equation*}
  \beta_n=
  \begin{cases}
    (F_1,\ldots,F_1) & n\geq 2, \\
    F_1 & n=1, \\
    F_0 & n=0.
  \end{cases}
\end{equation*}
It is immediate to check from the definition of Segal map that the following
diagram commutes for all $n\geq 0$,
\begin{equation*}
  \xymatrix@+20pt{
     X_n \ar[r]^{F_n}\ar_{\eta_n}[d] &  X'_n \ar^{\eta'_n}[d]\\
    S X_n \ar_{\beta_n}[r] & S X'_n.
  }
\end{equation*}
Thus the condition in the hypothesis of Lemma \ref{s4.lem1} is satisfied, with
$\gamma_n$ the identity 2-cell. It follows from Lemma \ref{s4.lem1} that there
is a pseudo-natural transformation $\beta:S X\rw S X'$ with
$\beta(F)_n=\beta_n$.

Suppose that $ X\supar{F} X'\supar{F'} X''$ is a pair of composable morphisms
in $\Ta$; then, for each $n\geq 2$,
\begin{equation*}
  \beta(F'F)_n=((F'F)_1,\ldots,(F'F)_1)=(F'_1,\ldots,F'_1)(F_1,\ldots,F_1)=
  \beta(F')_n\beta(F)_n.
\end{equation*}
Therefore, $\beta(F'F)=\beta(F')\beta(F)$.

Finally, the existence of a pseudo-natural transformation $\alpha:S X \rw  X$
follows immediately by Lemma \ref{s4.lem1}, taking $\alpha_{i}=\mathrm{id}$ for
$i=0,1$ and $\alpha_{n}=\mu_n$ for $n\geq2$.
\end{proof}

In the next theorem we apply Propositions \ref{s4.pro1} and \ref{special.3} to
obtain a rigidification functor for Tamsamani weak $2$-categories.

\begin{theorem}\label{s4.the1}
There is a functor
\begin{equation*}
    R:\Ta\rw[\dop,\Cat],
\end{equation*}
such that, for every $ X\in \Ta$,
\begin{itemize}
  \item [a)] $(R X)_n\cong\pro{(R X)_1}{(R X)_0}{}$ for all $n\geq 2$.
  \item [b)] There is a pseudo-natural transformation
  $r_X:R X\rw X$ which is a levelwise equivalence of categories.
  \item [c)] The map $(r_X)_0:(RX)_0\rw X_0$ induces equivalences of categories, for all $n\geq
  2$,
  \begin{equation*}
    \pro{(R X)_1}{(R X)_0}{n}\simeq\pro{(R X)_1}{ X_0}{n}\rlap{\,.}
  \end{equation*}
\end{itemize}
\end{theorem}
\begin{proof}
\begin{itemize}
  \item [a)] Let $\St:\Ps[\dop,\Cat]\rw[\dop,\Cat]$ be the strictification
  functor as in \cite{pow}, and let $S:\Ta\rw \Ps[\dop,\Cat]$ be as in
  Proposition \ref{s4.pro1}. Let $R$ be the composite
  \begin{equation*}
    \Ta\supar{S}\Ps[\dop,\Cat]\supar{\St}[\dop,\Cat].
  \end{equation*}
  Since, by Proposition \ref{s4.pro1} $(SX)_n=\pro{(SX)_1}{(SX)_0}{n}$ and $(SX)_0=X_0$
  is discrete, by Proposition \ref{special.3} we deduce, for all $n\geq 2$,
  \begin{eqnarray*}
   (R X)_n&=&(\mathit{St}\,S X)_n\\&\cong& \pro{(\mathit{St}\,S X)_1}{(\mathit{St}\,SX)_0}{n}\\
      &=&\pro{(R X)_1}{(R X)_0}{n}.
  \end{eqnarray*}

  \item [b)] By Proposition \ref{special.3}, there is a pseudo-morphism
  $r'_X:R X=\St\,S X\rw S X$ which is a levelwise categorical equivalence,
  and by Proposition \ref{s4.pro1} there is a pseudo-morphism $\alpha:SX\rw X$
  which is a levelwise categorical equivalence. Hence $r_X=\alpha r'_X:RX\rw X$
  satisfies b).
  \smallskip
  \item [c)]By Proposition \ref{special.3} c), since
  $(S X)_0= X_0$, we have
  \begin{eqnarray*}
        \pro{(R X)_1}{(R X)_0}{n}&=& \pro{(\St\,S X)_1}{(\St\,S X)_0}{n}\\
        &\simeq &\pro{(\St\,S X)_1}{(S X)_0}{n}\\
&=& \pro{(R X)_1}{ X_0}{n}.
  \end{eqnarray*}
\end{itemize}
\end{proof}

\section{Weakly globular double categories}\label{wglobular2}

In this section we show that the functors from $\dop$ to $\Cat$ 
in the image of $R$ can be viewed as a particular kind of double category.
We call these double categories weakly globular.
We will start with some background and notation for double categories. Then we will introduce 
the weakly globular double categories and some of their basic properties.
Finally we prove a discretization result: there is a functor $D$ from weakly globular 
double categories to $\Ta$ which is pseudo-inverse to $R$, giving us a biequivalence of 2-categories,
$$
\WGDbl_{\ps}\simeq_{\mbox{\scriptsize bi}}(\Ta)_{\ps}.
$$

\subsection{Double categories}
A double category $\bbX$ is an internal category in $\Cat$, i.e. 
a diagram of the form 
\begin{equation}\label{dblcat}
\bbX=\left(\xymatrix{\bbX_1\times_{\bbX_0}\bbX_1\ar[r]|-m
	&\bbX_1\ar@<.7ex>[r]^{d_0}\ar@<-.7ex>[r]_{d_1}&\bbX_0\ar[l]|s}\right).
\end{equation}
The elements of $\bbX_{00}$, i.e., the objects of the category $\bbX_0$, are the {\em objects} of the double category.
The elements of $\bbX_{01}$, i.e., the arrows of the category $\bbX_0$, are the {\em vertical arrows} of the double category.
Their domains, codomains, identities, and composition are as in $\bbX_0$. For objects $A,B\in \bbX_{00}$, 
we write $$\Hom_{\bbX,v}(A,B)=\Hom_{\bbX_0}(A,B).$$
We denote a vertical identity arrow by $1_A$ and write $\cdot$ for vertical composition.
The elements of $\bbX_{10}$ are the {\em horizontal arrows} of the double category, and their domain, codomain, identities
and composition are determined by  $d_0$, $d_1$, $s$, and $m$ in (\ref{dblcat}).
For objects $A,B\in\bbX_{00}$, we write 
$$\Hom_{\bbX,h}(A,B)=\{f\in\bbX_{10}| \,d_0(f)=A\mbox{ and }d_1(f)=B\}.$$
In order to make a notational distinction between horizontal and vertical arrows, we denote the vertical arrows
by $\xymatrix@1{\ar[r]|\bullet&}$ and the horizontal arrows by $\xymatrix@1{\ar[r]&}$.
The elements of $\bbX_{11}$ are the double cells of the double category.
An element $\alpha\in \bbX_{11}$ has a vertical domain and codomain in $\bbX_{10}$ (since $\bbX_1$ 
is a category), which are horizontal arrows, say $h$ and $k$ respectively.  
The cell $\alpha$ also has a horizontal domain, $d_0(\alpha)$, and a horizontal codomain, $d_1(\alpha)$.
The arrows $d_0(\alpha)$ and $d_1(\alpha)$ are vertical arrows. 
Furthermore, the horizontal and vertical domains and codomains of these arrows
match up in such a way that all this data fits together in a diagram
$$
\xymatrix{
\ar[r]^{h} \ar[d]|\bullet_{d_0(\alpha)} \ar@{}[dr]|\alpha
	& \ar[d]|\bullet^{d_1(\alpha)}
\\
\ar[r]_{k}& \rlap{\quad.}
}
$$
These double cells can be composed vertically by composition in $\bbX_1$  (again written as $\cdot$) and horizontally
by using $m$, and written as $\alpha_1\circ\alpha_2=m(\alpha_1,\alpha_2)$.
The identities in $\bbX_1$ give us vertical identity cells, denoted by
$$
\xymatrix{
A\ar[d]|\bullet_{1_A}\ar@{}[dr]|{1_f}\ar[r]^f &B\ar[d]|\bullet^{1_B}\ar@{}[drr]|{\mbox{or}} 
  && A\ar@{=}[d] \ar@{}[dr]|{1_f}\ar[r]^f &B\ar@{=}[d]
\\
A\ar[r]_f &B && A\ar[r]_f &B\rlap{\,.} 
}
$$
The image of $s$ gives us horizontal identity cells, denoted by
$$
\xymatrix{
A\ar[d]|\bullet_v\ar@{}[dr]|{\mbox{\scriptsize id}_{v}} \ar[r]^{\mbox{\scriptsize Id}_A} 
  &A\ar[d]|\bullet^v \ar@{}[drr]|{\mbox{or}} 
&& A\ar[d]|\bullet_v\ar@{}[dr]|{\mbox{\scriptsize id}_{v}} \ar@{=}[r]
  &A\ar[d]|\bullet^v 
\\
B\ar[r]_{\mbox{\scriptsize Id}_B} & B && B\ar@{=}[r] &B\rlap{\,,}
}$$
where $\mbox{Id}_A=s(A)$ and $\mbox{id}_v=s(v)$.
Composition of squares satisfies horizontal and vertical associativity laws and the middle four interchange law.
Further, $\mbox{id}_{1_A}=1_{\mbox{\scriptsize Id}_A}$ and we will denote this cell by $\iota_A$,
$$
\xymatrix{
A\ar[d]|\bullet_{1_A}\ar[r]^{\mbox{\scriptsize Id}_A} \ar@{}[dr]|{\iota_A}& A\ar[d]|\bullet^{1_A}
\\
A\ar[r]_{\mbox{\scriptsize Id}_A} & A\rlap{\,.}
}$$

For any double category $\bbX$, the {\em horizontal  nerve} $N_h\bbX$ is defined to be the functor
$N_h\bbX\colon \dop\rightarrow\Cat$ such that $(N_h\bbX)_0=\bbX_0$, $(N_h\bbX)_1=\bbX_1$
and $(N_h\bbX)_k=\bbX_1\times_{\bbX_0}\stackrel{k}{\cdots}\times_{\bbX_0}\bbX_1$ for $k\ge 2$.
So $N_h\bbX$ is given by the diagram
$$
\xymatrix{\ar@{}[r]|-\cdots&\bbX_1\times_{\bbX_0}\bbX_1\ar[r]|-m\ar@<.7ex>[r]^-{\pi_1}\ar@<-.7ex>[r]_-{\pi_2}
	&\bbX_1\ar@<.7ex>[r]^{d_0}\ar@<-.7ex>[r]_{d_1}&\bbX_0\ar[l]|s}
$$

\subsection{Pseudo-functors and strict functors}
As maps between double categories we consider those functors that correspond to 
natural transformations between their horizontal nerves.

\begin{dfn}\rm
\begin{enumerate}
 \item  A {\em strict functor} $F\colon \bbX\rightarrow \bbY$ between double categories 
is given by a natural transformation 
$$F\colon N_h\bbX\Rightarrow N_h\bbY\colon\Delta^{\mbox{\scriptsize op}}\rightarrow\Cat.$$
\item 
A {\em pseudo-functor} $F\colon \bbX\rightarrow \bbY$ between double categories 
is given by a pseudo-natural transformation 
$F\colon N_h\bbX\Rightarrow N_h\bbY\colon\Delta^{\mbox{\scriptsize op}}\rightarrow\Cat.$ 
\end{enumerate}
\end{dfn}

So strict functors send objects to objects, horizontal arrows to horizontal arrows, vertical arrows to vertical arrows, 
and double cells to double cells,
and preserve domains, codomains, identities and horizontal and vertical composition strictly.
However, pseudo-functors between double categories are  weak in a different way 
from what is described in \cite{DPP-spans2}, for instance.
A pseudo-functor $(F,\varphi,\sigma,\mu)\colon\bbX\rightarrow\bbY$ consists of functors 
$F_0\colon \bbX_0\rightarrow\bbY_0$, $F_1\colon \bbX_1\rightarrow \bbY_1$
and 
$F_k\colon \bbX_1\times_{\bbX_0}\stackrel{k}{\cdots}\times_{\bbX_0}\bbX_1
\rightarrow \bbY_1\times_{\bbY_0}\stackrel{k}{\cdots}\times_{\bbY_0}\bbY_1$ for $k\ge 2$,
together with invertible natural transformations,
$$
\xymatrix{
\bbX_1\ar[d]_{d_i}\ar[r]^{F_1}\ar@{}[dr]|{\varphi_i}
  &\bbY_1\ar[d]^{d_i} & \bbX_0\ar[d]_s\ar[r]^{F_0}\ar@{}[dr]|\sigma
  &\bbY_0\ar[d]^s 
\\
\bbX_0\ar[r]_{F_0}&\bbY_0 & \bbX_1\ar[r]_{F_1}&\bbY_1 
\\
\bbX_1\times_{\bbX_0}\bbX_1\ar[d]_m \ar[r]^{F_2}\ar@{}[dr]|\mu 
  &\bbY_1\times_{\bbY_0}\bbY_1\ar[d]^m 
  &  \bbX_1\times_{\bbX_0}\bbX_1\ar[d]_{\pi_i}\ar[r]^{F_2} \ar@{}[dr]|{\theta_i }
  &\bbY_1\times_{\bbY_0}\bbY_1\ar[d]^{\pi_i}
\\
\bbX_1\ar[r]_{F_1} & \bbY_1 
  &\bbX_1\ar[r]_{F_1} & \bbY_1\rlap{\,,}
}
$$
(where $i=1,2$), and analogously for $F_k$ with $k\ge 2$.
These satisfy the usual naturality and coherence conditions, that can be derived from their simplicial description. 
We will list the ones that we will use in the remainder of this paper.
For instance,
\begin{equation}\label{identityeqn}
\xymatrix{
\bbX_0\ar[d]_s\ar[r]^{F_0}\ar@{}[dr]|\sigma
  &\bbY_0\ar[d]^s  & \bbX_0\ar@{=}[dd]\ar[r]^{F_0}\ar@{}[ddr]|{1_{F_0}} & \bbY_0\ar@{=}[dd]
\\
\bbX_1\ar[d]_{d_i}\ar[r]^{F_1}\ar@{}[dr]|{\varphi_i}
  &\bbY_1\ar[d]^{d_i} \ar@{}[r]|{\textstyle =} & &&\mbox{for }i=1,2.
\\
\bbX_0\ar[r]_{F_0}&\bbY_0& \bbX_0\ar[r]_{F_0}&\bbY_0
}\end{equation}
This means that vertical composition and domains and codomains are preserved strictly,
but horizontal domains and codomains are only preserved up to a vertical isomorphism.
(Compare this to the notion of pseudo-morphism in \cite{DPP-spans2} which requires that domains and codomains
are preserved strictly, but horizontal composition and units are only preserved up to coherent isomorphisms.)
For the pseudo-morphisms we consider in this paper the typical image of a horizontal arrow 
$\xymatrix@1{A\ar[r]^f&B}$ under $F$ corresponds to the following diagram:
$$
\xymatrix@R=2em{d_0F_1f\ar[r]^{F_1f}\ar[d]|\bullet_{(\varphi_1)_A} ^\sim& d_1F_1f\ar[d]|\bullet_\sim^{(\varphi_2)_B}
\\
F_0A&F_0B
}$$
and the equation (\ref{identityeqn}) means that 
the components of $\sigma$ have the following shape:
$$
\xymatrix{
F_0A\ar[r]^{\mbox{\scriptsize Id}_{F_0A}}\ar[d]|\bullet_{(\varphi_1)_A^{-1}}\ar@{}[dr]|{\sigma_A} 
  & F_0A\ar[d]|\bullet^{(\varphi_2)_A^{-1}}
\\
d_0F_1(\mathrm{Id}_A) \ar[r]_{F_1(\mbox{\scriptsize Id}_A)} & d_1F_1(\mathrm{Id}_A)
}$$

Naturality of $\mu$ means that for  any pair of horizontally composable double cells in $\bbX$,
$$
\xymatrix{
A_1\ar[d]|\bullet_u\ar[r]^{f_1}\ar@{}[dr]|\alpha &B_1\ar[d]|\bullet_{v_1}\ar[r]^{g_1}\ar@{}[dr]|\beta 
	& C_1\ar[d]|\bullet^w
\\
A_2\ar[r]_{f_2} & B_2\ar[r]_{g_2} & C_2\rlap{\,,}
}$$
the following equation holds:
\begin{equation}\label{comp-nat}
 \xymatrix@C=4em@R=1.8em{
\ar[rrr]^{F_1(g_1\circ f_1)}\ar[d]|\bullet_\sim \ar@{}[drrr]|{\mu_{g_1,f_1}} &&& \ar[d]|\bullet^\sim
\\
\ar[r]^{\pi_2F_2(g_1,f_1)} \ar[d]|\bullet_\sim \ar@{}[dr]|{(\theta_2)_{g_1,f_1}} 
	&\ar[d]|\bullet^\sim \ar@{=}[r] \ar@{}[dddr]|=
	&\ar[r]^{\pi_1F_2(g_1,f_1)} \ar[d]|\bullet_\sim \ar@{}[dr]|{(\theta_1)_{g_1,f_1}} 
	&\ar[d]|\bullet^\sim
\\
\ar[r]|{F_1f_1}\ar[d]|\bullet_\sim \ar@{}[dr]|{F_1\alpha} & \ar[d]|\bullet^\sim 
	& \ar[d]|\bullet_\sim\ar@{}[dr]|{F_1\beta}\ar[r]|{F_1g_1} & \ar[d]|\bullet^\sim \ar@{}[drr]|{\textstyle =}
	&& \ar[r]|{F_1(g_1\circ f_1)}\ar@{}[dr]|{F_1(\beta\circ\alpha)} \ar[d]|\bullet_\sim &\ar[d]|\bullet^\sim
\\
\ar[d]|\bullet_\sim \ar[r]_{F_1f_2} \ar@{}[dr]|{(\theta_2)_{g_2,f_2}^{-1}} &\ar[d]|\bullet^\sim 
	& \ar[d]|\bullet_\sim \ar[r]|{F_1g_2} \ar@{}[dr]|{(\theta_1)_{g_2,f_2}^{-1}} &\ar[d]|\bullet^\sim 
	&& \ar[r]|{F_1(g_2\circ f_2)} &\rlap{\qquad.}
\\
\ar[r]_{\pi_2F_2(g_2,f_2)} \ar[d]|\bullet_\sim \ar@{}[drrr]|{\mu_{g_2,f_2}^{-1}}
	& \ar@{=}[r] & \ar[r]_{\pi_1F_2(g_2,f_2)} &\ar[d]|\bullet^\sim
\\
\ar[rrr]_{F_1(g_2\circ f_2)} &&&
}
\end{equation}
Furthermore, the identity coherence axioms for $m$ and $s$ state that
for a horizontal arrow  $\xymatrix@1{A\ar[r]^f&B}$,
\begin{equation}\label{ass-id1}
\xymatrix@C=4.5em@R=2em{
 \ar[rrr]^{F_1(\mbox{\scriptsize Id}_B\circ f)} \ar@{}[drrr]|{\mu_{\mathrm{Id}_B, f}} \ar[d]|\bullet 
	&&& \ar[d]|\bullet \ar@{}[dr]|{\textstyle =}
	& \ar[rr]^{F_1(f)}\ar[d]|\bullet\ar@{}[drr]|{\mathrm{id}_{f}}
	&& \ar[d]|\bullet
\\
\ar[ddd]|\bullet\ar[r]^{\pi_2(F_2(\mathrm{Id}_B,f))} \ar@{}[dddr]|{(\theta_2)_{\mathrm{Id}_B,f}}
	&\ar[ddd]|\bullet\ar@{=}[r] 
	& \ar[r]^{\pi_1(F_2(\mathrm{Id}_B,f))} \ar@{}[dr]|{(\theta_1)_{\mathrm{Id}_B,f}}\ar[d]|\bullet 
	&\ar[d]|\bullet & \ar[rr]_{F_1(f)} &&
\\
&  \ar@{}[dr]|{=}
      &\ar[r]|{F(\mbox{\scriptsize Id}_B)}\ar[d]|\bullet \ar@{}[dr]|{\sigma_B^{-1}} 
      & \ar[d]|\bullet 
\\
&&\ar[d]|\bullet \ar[r]|{\mbox{\scriptsize Id}_{F_0B}} \ar@{}[dr]|{\mbox{\scriptsize id}} &\ar[d]|\bullet 
\\
\ar[r]_{F_1(f)} & \ar@{=}[r] & \ar[r]_{\mbox{\scriptsize Id}_{d_1F_1(f)}} &}
\end{equation}
and
\begin{equation}\label{ass-id2}
\xymatrix@C=4.5em@R=2em{
 \ar[rrr]^{F_1(f\circ \mbox{\scriptsize Id}_A)} \ar@{}[drrr]|{\mu_{f,\mathrm{Id}_A}} \ar[d]|\bullet 
	&&& \ar[d]|\bullet \ar@{}[dr]|{\textstyle =}
	& \ar[rr]^{F_1(f)}\ar[d]|\bullet\ar@{}[drr]|{\mathrm{id}_{f}} 
	&& \ar[d]|\bullet
\\
\ar[d]|\bullet\ar[r]^{\pi_2(F_2(f,\mathrm{Id}))}\ar@{}[dr]|{(\theta_2)_{f,\mathrm{Id}_A}}
	&\ar[d]|\bullet\ar@{=}[r] \ar@{}[dddr]|{=}
	& \ar[r]^{\pi_1(F_2(f,\mathrm{Id}_A))} \ar@{}[dddr]|{(\theta_1)_{f,\mathrm{Id}_A}}\ar[ddd]|\bullet 
	&\ar[ddd]|\bullet & \ar[rr]_{F_1(f)} &&
\\
\ar[r]|{F_1(\mbox{\scriptsize Id}_A)} \ar[d]|\bullet \ar@{}[dr]|{\sigma_A^{-1}}&\ar[d]|\bullet  
\\
\ar[d]|\bullet \ar[r]|{\mbox{\scriptsize Id}_{F_0A}} \ar@{}[dr]|{\mbox{\scriptsize id}} &\ar[d]|\bullet &&
\\
\ar[r]_{\mbox{\scriptsize Id}_{d_0F_1(f)}}  & \ar@{=}[r] &\ar[r]_{F_1(f)}  &\rlap{\,.}}
\end{equation}

\subsection{Vertical and horizontal transformations}
Since double categories have two types of arrows, there are two possible choices
for types of transformations between maps of double functors: vertical and horizontal transformations.
Vertical transformations correspond to modifications between natural transformations of 
functors from $\dop$ into $\Cat$, so these are the ones that are relevant to our study of the functor $R$.

\begin{dfn}\rm
A {\em vertical transformation} $\gamma\colon F\Rightarrow G\colon \bbX\rightrightarrows \bbY$ between 
strict double functors has components
vertical arrows $\xymatrix@1{\gamma_A\colon FA\ar[r]|-\bullet &GA}$
indexed by the objects of $\bbX$ and for each horizontal arrow $\xymatrix@1{A\ar[r]^h&B}$
in $\bbX$, a double cell
$$
\xymatrix{
FA\ar[r]^{Fh}\ar[d]|\bullet_{\gamma_A} \ar@{}[dr]|{\gamma_h}& FB\ar[d]|\bullet^{\gamma_B}
\\
GA\ar[r]_{Gh} & GB,
}$$
such that $\gamma$ is strictly functorial in the horizontal direction,
i.e., $\gamma_{h_2\circ h_1}=\gamma_{h_2}\circ\gamma_{h_1}$,
and natural in the vertical direction, i.e., 
\begin{equation}\label{vertnat}
\xymatrix{
FA\ar[r]^{Fh}\ar[d]|\bullet_{Fv}\ar@{}[dr]|{F\zeta} & FB\ar[d]|\bullet^{Fw} 
		&& FA\ar[d]|\bullet_{\gamma_A}\ar[r]^{Fh}\ar@{}[dr]|{\gamma_h} & FB\ar[d]|\bullet^{Gw}
\\
FC\ar[d]|\bullet_{\gamma_C}\ar@{}[dr]|{\gamma_k} \ar[r]_{Fk} & FD\ar[d]|\bullet^{\gamma_D}
		& \equiv & GA\ar[d]|\bullet_{Gv}\ar@{}[dr]|{G\zeta}\ar[r]_{Gh} &GB\ar[d]|\bullet^{Gw}
\\
GC\ar[r]_{Gk} & GD && GC\ar[r]_{Gk} & GD\rlap{\,,}
}
\end{equation}
for any double cell $\zeta$ in $\bbX$.
\end{dfn}

To give a vertical transformation between pseudo-functors of double categories,
we need to require that the data above fits together with the structure cells of the
pseudo-transformations. We spell out part of the details and will leave the rest for the reader.

\begin{dfn}\rm
A {\em vertical transformation} $\gamma\colon F\Rightarrow G\colon \bbX\rightrightarrows \bbY$ between 
pseudo-double functors has components
vertical arrows $\xymatrix@1{\gamma_A\colon FA\ar[r]|\bullet &GA}$
indexed by the objects of $\bbX$ and for each horizontal arrow $\xymatrix@1{A\ar[r]^h&B}$
in $\bbX$, a double cell
$$
\xymatrix{
d_0Fh\ar[r]^{Fh}\ar[d]|\bullet_{d_0\gamma_h} \ar@{}[dr]|{\gamma_h}& d_1Fh\ar[d]|\bullet^{d_1\gamma_h}
\\
d_0Gh\ar[r]_{Gh} & d_1Gh,
}$$
such that the following squares of vertical arrows commute:
$$
\xymatrix{
F_0A\ar[r]|\bullet\ar[d]|\bullet_{\gamma_A} & d_0Fh\ar[d]|\bullet^{d_0\gamma_h}
		&F_0B\ar[r]|\bullet\ar[d]|\bullet_{\gamma_B} & d_1F_1h\ar[d]|\bullet^{d_1\gamma_h}
\\
G_0A\ar[r]|\bullet&d_0G_1h & G_0B\ar[r]|\bullet &d_1G_1h
}
$$
where the unlabeled arrows are the structure isomorphisms corresponding to $F$ and $G$.

We require that $\gamma$ is natural in the vertical direction, in the sense that 
the following square of vertical arrows commutes for a 
vertical arrow $\xymatrix@1{A\ar[r]|\bullet ^v&B}$
in $\bbX$,
$$
\xymatrix{
F_0A\ar[r]|\bullet^{F_1v} \ar[d]|\bullet_{\gamma_A}& F_0B \ar[d]|\bullet^{\gamma_B}
\\
G_0A \ar[r]|\bullet^{G_1v}& G_0B
}$$
and furthermore, for any double cell $\zeta$ in $\bbX$,
$$
\xymatrix{
d_0F_1h\ar[r]^{Fh}\ar[d]|\bullet_{d_0F_1\zeta}\ar@{}[dr]|{F_1\zeta} & d_1F_1h\ar[d]|\bullet^{d_1F_1\zeta} 
		&& d_0F_1h\ar[d]|\bullet_{d_0\gamma_h}\ar[r]^{Fh}\ar@{}[dr]|{\gamma_h} 
		& d_1F_1h\ar[d]|\bullet^{d_1\gamma_h}
\\
d_0F_1k\ar[d]|\bullet_{d_0\gamma_k}\ar@{}[dr]|{\gamma_k} \ar[r]_{Fk} 
		& d_1F_1k\ar[d]|\bullet^{d_1\gamma_k}
		& \equiv & d_0G_1h\ar[d]|\bullet_{d_0G_1\zeta}\ar@{}[dr]|{G_1\zeta}\ar[r]_{Gh} 
		&d_1G_1h\ar[d]|\bullet^{d_1G_1\zeta}
\\
d_0G_1k\ar[r]_{Gk} &d_1G_1k && d_0G_1k\ar[r]_{Gk} & d_1G_1k\rlap{\,.}
}
$$

In the horizontal direction we require pseudo-functoriality, which means that
$$
\xymatrix@C=4em@R=1.8em{
\ar[rrr]^{F_1(gf)} \ar[d]|\bullet\ar@{}[drrr]|{\mu^F_{g,f}} &&& \ar[d]|\bullet
\\
\ar[d]|\bullet \ar[r]^{\pi_2F_2(g,f)} \ar@{}[dr]|{(\theta_2)_{g,f}} & \ar[d]|\bullet \ar@{=}[r] 
	&\ar[d]|\bullet \ar[r]^{\pi_1F_2(g,f)} \ar@{}[dr]|{(\theta_1)_{g,f}} &\ar[d]|\bullet 
\\
\ar[r]|{F_1f} \ar[d]|\bullet_{d_0\gamma_f} \ar@{}[dr]|{\gamma_f} & \ar[d]|\bullet 
	&\ar[d]|\bullet \ar[r]|{F_1g}\ar@{}[dr]|{\gamma_g} 
	& \ar[d]|\bullet^{d_1\gamma_g} \ar@{}[drr]|{\textstyle =} 
	&& \ar[r]^{F_1(gf)} \ar[d]|\bullet\ar@{}[dr]|{\gamma_{gf}} &\ar[d]|\bullet
\\
\ar[d]|\bullet\ar@{}[dr]|{(\theta_2)_{g,f}^{-1}} \ar[r]|{G_1f} &\ar[d]|\bullet 
	&\ar[d]|\bullet \ar[r]|{G_1g}\ar@{}[dr]|{(\theta_1)_{g,f}^{-1}} 
	&\ar[d]|\bullet &&\ar[r]_{G_1(gf)}& \rlap{\quad.}
\\
\ar[d]|\bullet\ar[r]_{\pi_2G_2(g,f)}\ar@{}[drrr]|{(\mu_{g,f}^G)^{-1}} &\ar@{=}[r] & \ar[r]_{\pi_2G_2(g,f)} 
	&\ar[d]|\bullet 
\\
\ar[rrr]_{G_1(gf)}&&&
}$$
\end{dfn}

In order to state the universal property  in Section \ref{fracns} below, 
we need to consider horizontal 
transformations between strict functors of double categories.
The definition of a horizontal transformation is dual to that of a vertical transformation 
in that all mentions of vertical and horizontal
have been exchanged.

\begin{dfn} \rm A {\em horizontal transformation} $a\colon G\Rightarrow K\colon \bbD\rightrightarrows\bbE$
between functors of double categories has components horizontal arrows 
$\xymatrix@1{GX\ar[r]^{a_X}&KX}$ indexed by the objects of $\bbD$
and for each vertical arrow $\xymatrix@1{X\ar[r]|\bullet^v&Y}$ in $\bbD$, a double cell
$$
\xymatrix@R=1.9em{
GX\ar[r]^{a_X}\ar[d]|\bullet_{Gv} \ar@{}[dr]|{a_v}& KX\ar[d]|\bullet^{Kv}
\\
GY\ar[r]_{a_Y} & KY,
}$$
such that $a$ is strictly functorial in the vertical direction, i.e., $a_{v_1\cdot v_2}=a_{v_1}\cdot a_{v_2}$
and natural in the horizontal direction,
i.e., the composition of
$$
\xymatrix{
GX\ar[d]|\bullet_{Gv}\ar@{}[dr]|{G\zeta}\ar[r]^{Gf} 
    & GX'\ar[d]|\bullet^{Gv'}\ar[r]^{a_{X'}} \ar@{}[dr]|{a_{v'}} 
    & KX'\ar[d]|\bullet^{Kv'}
\\
GY\ar[r]_{Gg} & GY'\ar[r]_{a_{Y'}} &KY'
}$$
is equal to the composition of
$$
\xymatrix{
GX\ar[d]|\bullet_{Gv}\ar@{}[dr]|{a_{v}}\ar[r]^{a_X} 
    & KX\ar[d]|\bullet^{Kv}\ar[r]^{Kf} \ar@{}[dr]|{K\zeta} 
    & KX'\ar[d]|\bullet^{Kv'}
\\
GY\ar[r]_{a_Y} & KY\ar[r]_{Kg} &KY'
}$$
for any double cell $\zeta$ in $\bbD$.
\end{dfn}

It is possible to consider $\DblCat$ as a double category with strict functors, 
horizontal transformations, vertical transformations, and modifications. However, we won't need 
this result in this paper. We will either use just the vertical transformations, 
or just the horizontal transformations.

We write $\DblCat_v$, respectively $\DblCat_h$, for the 2-categories of double categories, strict functors,
and vertical transformations, respectively horizontal transformations.
We will write $\DblCat_{\ps}$ for the 2-category of double categories, pseudo-functors, and vertical transformations
of pseudo-functors.

There is an inclusion functor $H\colon \mbox{\bf Cat}\rightarrow \mbox{\bf DblCat}_h$
that sends a category $\bfC$ to a double category
with the category $\bfC$  as horizontal arrows and only identity arrows (on objects of $\bfC$) as vertical arrows 
and vertical identity cells as squares.
Adjoint to this functor there is $h\colon \mbox{\bf DblCat}_h\rightarrow\mbox{\bf Cat}$ sending 
a double category to its category of horizontal arrows.
Completely analogously, there is a functor $V\colon\mbox{\bf Cat}\rightarrow \mbox{\bf DblCat}_v$
that sends a category $\bfC$ to a double category with the arrows of $\bfC$ in the vertical direction
(and a discrete horizontal category) with adjoint $v\colon \mbox{\bf DblCat}_v\rightarrow\mbox{\bf Cat}$.

\subsection{The definition of weakly globular double categories}
\begin{definition}\label{s4.def1}\rm
    The 2-category $\WGDbl$ of weakly globular double categories is the full
    sub 2-category of double categories whose objects are double categories $\bbX$
    such that there is an equivalence of categories $\gamma:\bbX_0\rw \bbX_0^d$,
    natural in $\bbX_0$, where $\bbX_0^d$ is the discrete category of the path components of $\bbX_{0}$; 
    further, $\gamma$ induces
    an equivalence of categories, for all $n\geq 2$,
    \begin{equation}\label{s4.eq01}
        \pro{\bbX_1}{\bbX_0}{n}\simeq\pro{\bbX_1}{\bbX_0^d}{n}\rlap{\,.}
    \end{equation}
    Analogously to the notation introduced above we write $\WGDbl_{\ps}$ 
for the 2-category of weakly globular double categories, pseudo-functors, and vertical transformations 
between them. We also write $\WGDbl_v$ and $\WGDbl_h$ for the 2-categories with strict functors and
vertical, respectively horizontal, transformations.
\end{definition}

\begin{remark}\label{s4.rem1}\rm
    Notice that the condition \eqref{s4.eq01} is in particular satisfied in the
    case where at least one of the maps $\xymatrix@1{\bbX_1\ar@<.5ex>[r]^{d_0}\ar@<-.5ex>[r]_{d_1}&\bbX_0}$ is an
    isofibration.
\end{remark}

Note that the domain map $d_0\colon\bbX_1\rightarrow\bbX_0$ is an isofibration
if and only if each pair of arrows 
$$
\xymatrix{
C\ar[d]|\bullet_x
\\
A\ar[r]_f &B
}
$$
(with $x$ an isomorphism)
can be completed to a vertically invertible double cell
$$
\xymatrix{
C\ar[d]|\bullet_x\ar[r]^g \ar@{}[dr]|\alpha &D\ar[d]|\bullet^y
\\
A\ar[r]_f&B
}
$$
and the codomain map $d_1\colon\bbX_1\rightarrow\bbX_0$ is an isofibration
if and only if each pair of arrows 
$$
\xymatrix{
&B\ar[d]|\bullet^y
\\
C\ar[r]_f&D
}
$$
(with $y$ an isomorphism)
can be completed to a vertically invertible double cell
$$
\xymatrix{
A\ar[d]|\bullet_x\ar[r]^g \ar@{}[dr]|\alpha &B\ar[d]|\bullet^y
\\
C\ar[r]_f&D.
}
$$

For arbitrary weakly globular double categories this is not
always possible, but the condition that
\begin{equation}\label{cond2}\bbX_1\times_{\bbX_0^d}\bbX_1\simeq\bbX_1\times_{\bbX_0}\bbX_1 \end{equation}
(given by the Segal map $\eta_2\colon \bbX_1\times_{\bbX_0}\bbX_1\rightarrow\bbX_1\times_{\bbX_0^d}\bbX_1$
and its pseudo-inverse $\mu_2\colon \bbX_1\times_{\bbX_0^d}\bbX_1\rightarrow\bbX_1\times_{\bbX_0}\bbX_1$)
together with the condition that $\bbX_0^d\simeq \bbX_0$
does imply the following weaker results.

\begin{lma}\label{squarecompln} Let $\bbX$ be a weakly globular double category.
\begin{enumerate}
\item
For each  pair of arrows 
$$
\xymatrix{
A'\ar[d]|\bullet_x
\\
A\ar[r]_f &B
}
$$
in $\bbX$, there is a vertically invertible double cell of the form
$$
\xymatrix{
A''\ar[d]|\bullet_{y_1}\ar@{}[ddr]|\alpha\ar[r]^g & B'\ar[dd]|\bullet^{y_2}
\\
A'\ar[d]|\bullet_x
\\
A\ar[r]_f & B\rlap{\,.}
}
$$
\item
For each pair of arrows 
$$
\xymatrix{
&B'\ar[d]|\bullet^y
\\
A\ar[r]_f&B
}
$$
in $\bbX$,
there is a vertically invertible  double cell of the form
$$
\xymatrix{
A'\ar[dd]|\bullet_{x_1}\ar[r]^g \ar@{}[ddr]|\alpha & B''\ar[d]|\bullet^{x_2}
\\
& B'\ar[d]|\bullet^y
\\
A\ar[r]_f&B.
}
$$
\end{enumerate}
\end{lma}

\begin{proof}
\begin{enumerate}
\item
Consider the pair of horizontal arrows $\mbox{Id}_{A'}\colon A'\rightarrow A'$ and $f\colon A\rightarrow B$.
Then $(\mbox{Id}_{A'},f)$ is an object of $\bbX_1\times_{\bbX_0^d}\bbX_1$, 
so let $\mu_2(\mbox{Id}_{A'},f)=(\mbox{Id}_{A''},g)$
and the equivalence (\ref{cond2})
gives rise to vertically invertible double cells as in the following diagram:
$$
\xymatrix{
A''\ar@{}[dr]|{\mbox{\scriptsize id}_{y_1}}\ar[d]|\bullet_{y_1}
	  \ar[r]^{\mbox{\scriptsize Id}_{A''}}
      &A''\ar[r]^g\ar[d]|\bullet^{y_1} \ar@{}[ddr]|\alpha & B'\ar[dd]|\bullet^{y_2}
\\
A'\ar[r]_{\mbox{\scriptsize Id}_{A'}} & A'\ar[d]|\bullet_x
\\
& A\ar[r]_f&B
}
$$
\item
Consider the pair of horizontal arrows $g\colon A\rightarrow B$ and $\mbox{Id}_{B'}\colon B'\rightarrow B'$.
Analogous to the argument given in the first part of this proof, there
are vertically invertible double cells as in the following diagram:
$$
\xymatrix{
A\ar[dd]|\bullet_{x_1}\ar@{}[ddr]|{\alpha} \ar[r]^g 
	& B''\ar[d]|\bullet_{x_2}\ar[r]^{\mbox{\scriptsize Id}_{B''}} \ar@{}[dr]|{\mbox{\scriptsize id}_{x_2}}
	&B''\ar[d]|\bullet^{x_2}
\\
	& B'\ar[r]_{\mbox{\scriptsize Id}_{B'}} \ar[d]|\bullet^{y} & B'
\\
A\ar[r]_f &B	
}
$$
\end{enumerate}
\end{proof}

Finally, we have the following corollary to Theorem \ref{s4.the1}.

\begin{corollary}\label{s4.cor1}
    There is a functor
    $$Q:\Ta\rw \WGDbl_v$$ such that, for all $ X\in \Ta$, there is a
    pseudo-morphism $N_{h}Q X\rw X$ which is a levelwise categorical
    equivalence.
\end{corollary}
\begin{proof}
Recall that the category of double categories is isomorphic to the subcategory
of $[\dop,\Cat]$ whose objects $X$ satisfy $X_{n*}\cong\pro{X_{1*}}{X_{0*}}{n}$
for all $n\geq 2$. It is immediate from the definition that $RX$ is the
horizontal nerve of a weakly globular double category. 
If $P$ is the left adjoint to the horizontal nerve $N_{h}$, then $PN_h=\id$. 
Define $Q=PR$. 
Then $QX\in\WGDbl$ and has the desired property by Theorem \ref{s4.the1}.
\end{proof}

\subsection{Discretization}
In the next proposition we construct a functor in the opposite direction, from
weakly globular double categories to Tamsamani weak 2-categories.  The idea
is to replace the category of objects and vertical morphisms in a weakly
globular double category by its equivalent discrete category. This recovers the
globularity condition, but at the expense of the strictness of  the Segal maps:  from being
isomorphisms they become equivalences.

\begin{proposition}\label{s4.pro2}
    There is a functor
    \begin{equation*}
        D:\WGDbl_v\rw \Ta
    \end{equation*}
    such that, for every $\bbX\in\WGDbl$ there is a pseudo-morphism $\eta_\bbX:D\bbX \rw N_h \bbX$,
    natural in $\bbX$, which is a levelwise categorical equivalence.
\end{proposition}

\begin{proof}
If $\bbX \in\WGDbl$ by definition there is a categorical equivalence
$\gamma:\bbX_0\rw \bbX_0^d$ with $\bbX_0^d$ discrete, hence there is
$\gamma':\bbX_0^d\rw \bbX_0$ with $\gamma\gamma'=\id$. Let $(D\bbX)_0=\bbX_0^d$,
$(D\bbX)_1=\bbX_1$, and 
$(D\bbX)_k=\bbX_1\times_{\bbX_0}\stackrel{k}{\cdots}\times_{\bbX_0}\bbX_1$ for $k\geq 2$. 
Let $\pt_i,\sigma_i$ be the face and
degeneracy operators of $N_h \bbX$. Define $d_i:(D\bbX)_1\rw(D\bbX)_0$ and
$s_{0}:(D\bbX)_0\rw(D\bbX)_1$ by $d_i=\gamma\pt_i$, for $i=0,1$ and
$s_0=\sigma_0\gamma'$. All other face and degeneracy operators in $D\bbX$ are
as in $N_h \bbX$. Notice that, since $\gamma\gamma'=\id$,
$D\bbX\in[\dop,\Cat]$.

 By construction, $(D\bbX)_0$ is discrete. To show that
$D\bbX\in \Ta$ we need to show that all Segal maps are categorical
equivalences. Since $\bbX$ is weakly globular, by definition we have for
$n\geq 2$,
\begin{eqnarray*}
    (D\bbX)_n&=&\pro{\bbX_1}{\bbX_0}{n}\simeq \pro{\bbX_1}{\bbX_0^d}{n}\\
    & =&\pro{(D\bbX)_1}{(D\bbX)_0}{n}.
\end{eqnarray*}
This shows that $D\bbX\in \Ta$.

Let $\eta_0=\gamma'$, $\eta_k=\id$ for $k>0$. This defines a pseudo-natural
transformation $\eta:D\bbX \rw N_h \bbX$ which is a levelwise categorical
equivalence.
\end{proof}

\begin{remark}\label{s4.rem01}\rm
    Using Lemma \ref{s4.lem1} it is straightforward to see that for each $\bbX \in
    \WGDbl$ the pseudo-morphism $\eta_\bbX: D \bbX \rw N_h \bbX$ of Proposition \ref{s4.pro2}
    has a pseudo-inverse $\mu_\bbX:N_h \bbX \rw D\bbX$ in $\Ps[\dop,\Cat]$. Likewise, for
    each $Y \in \Ta$ the pseudo-morphism $N_h Q Y\rw Y$ of Corollary \ref{s4.cor1}
    has a pseudo-inverse $Y \rw N_h Q Y$.
\end{remark}

Notice that the functors $Q:\Ta\rw\WGDbl$ and $D:\WGDbl\rw \Ta$ of Corollary
\ref{s4.cor1} and Proposition \ref{s4.pro2} extend to functors
$Q:(\Ta)_{\ps}\rw(\WGDbl)_{\ps}$ and $D:(\WGDbl)_{\ps}\rw(\Ta)_{\ps}$ (we shall
denote them with the same letters). 

\begin{thm}\label{s4.pro4}
There is a biequivalence of 2-categories:
\begin{equation*}
    (\WGDbl)_{\ps}\simeq(\Ta)_{\ps}\;.
\end{equation*}
\end{thm} 

\begin{proof}
Since the horizontal nerve functor $N_h:(\WGDbl)_{\ps}\rw \Ps[\dop,\Cat]$ is
fully faithful, there is an isomorphism
\begin{equation}\label{s4.eq5}
    \mathrm{Hom}_{(\smWGDbl)_{\ps}}(\bbX,\bbY)\cong \mathrm{Hom}_{\sf{Ps}[\dop,\sf{Cat}]}(N_h
    \bbX,N_h \bbY)\;.
\end{equation}
We claim that there is an equivalence of categories
\begin{equation}\label{s4.eq6}
    F: \mathrm{Hom}_{\sf{Ps}[\dop,\sf{Cat}]}(N_h \bbX,N_h \bbY) \simeq
    \mathrm{Hom}_{(\sf{Ta}_2)_{\it ps}}(D\bbX,D\bbY)\;:G
\end{equation}
This is constructed as follows. Let $\eta_\bbX:D\bbX\rw N_h \bbX$ and
$\mu_\bbX:N_h\bbX\rw D\bbX$ be as in Remark \ref{s4.rem01}.

Define
\begin{equation*}
    Ff=\mu_\bbY f \eta_\bbX, \qquad\qquad Gg=\eta_\bbY g \mu_\bbX\;.
\end{equation*}
Then  $FGg=\mu_\bbY(\eta_\bbY g \mu_\bbX)\eta_\bbX\cong g$ and
$GFf=\eta_\bbY(\mu_\bbY f \eta_\bbX)\mu_\bbX \cong f$, showing that
\eqref{s4.eq6} is an equivalence of categories as claimed.

From \eqref{s4.eq5} and \eqref{s4.eq6} we deduce
\begin{equation*}
    \mathrm{Hom}_{(\smWGDbl)_{\mathrm{ps}}}(\bbX,\bbY) \simeq 
		\mathrm{Hom}_{(\sf{Ta_2})_{\mathrm{ps}}}(D\bbX,D\bbY)\;,
\end{equation*}
that is, the functor $D$ is locally an equivalence of categories.

On the other hand, $D$ is also biessentially surjective on objects. In fact by
Corollary \ref{s4.cor1} and by Proposition \ref{s4.pro2}, for every
$X\in(\Ta)_{\ps}$ there is a composite morphism $DQ X\rw N_h QX\rw X$
in $\Ps[\dop,\Cat]$ which is levelwise a categorical equivalence, hence an
equivalence in $(\Ta)_{\ps}$. In conclusion, $D$ is a
biequivalence.
\end{proof}

\begin{rmk}\label{notadj}\rm
The biequivalence of the previous theorem is not an adjoint equivalence.
The functor $D$ cannot be a right adjoint, since it doesn't preserve general limits, as it is $\pi_0$ at level 0.
On the other hand, $Q$ clearly does not preserve products, so it cannot be a right adjoint either.
\end{rmk}

We are going to introduce a notion of equivalence in $(\WGDbl)_{\ps}$ modeled
over the comparison with Tamsamani weak 2-categories. We need the following
preliminary lemma. Let $\pi_0:\Cat\rw\Set$ associate to a category the set of
isomorphism classes of its objects.

\bigskip

\bigskip

\begin{lemma}\label{s4.lem2}
\
    \begin{itemize}
      \item [i)] There is a functor
      \begin{equation*}
        \Pi_0:(\WGDbl)_{\ps}\rw\Cat
      \end{equation*}
    such that, for all $\bbX\in(\WGDbl)_{\ps}$ and $n\geq 0$, $(N\Pi_0 \bbX)_n=(\pi_0^* N_h
    \bbX)_n$ where $N:\Cat\rw[\dop,\Set]$ is the nerve functor.
      \item [ii)] If $\bbX\in(\WGDbl)_{\ps}$ and $a,b\in \bbX_0^d$, let $\bbX_{(a,b)}$ be
      the full subcategory of $\bbX_1$, whose objects $z$ are such that
      $\gamma\pt_0(z)=a$, $\gamma\pt_1(z)=b$, where $\pt_i:\bbX_1\rw \bbX_0$
      are the face operators and $\gamma:\bbX_0\rw \bbX_0^d$. Then
      $\bbX_1\cong\underset{a,b\in \bbX_0^d}{\coprod}\bbX_{(a,b)}$.
      \item [iii)] A morphism $F:\bbX\rw \bbY$ in $(\WGDbl)_{\ps}$ induces functors $F_{(a,b)}\colon\bbX_{(a,b)}\rw
      \bbY_{(Fa,Fb)}$ for all $a,b\in \bbX_0^d$.
    \end{itemize}
\end{lemma}
\begin{proof}

\begin{itemize}
  \item [i)]
    By definition, since $\bbX$ is weakly globular, there is a categorical
    equivalence
    $$\pro{\bbX_{1}}{\bbX_{0}}{n}\simeq\pro{\bbX_{1}}{\bbX^d_{0}}{n}$$
    for all $n\geq 2$. Since $\pi_0$ sends categorical equivalences to
    isomorphisms and preserves fiber products over discrete objects,
    this implies
    \begin{equation*}
   \pi_0(\pro{\bbX_{1}}{\bbX^d_{0}}{n})\cong\pro{\pi_0 \bbX_{1}}{\pi_0
    \bbX^d_{0}}{n}.
    \end{equation*}
    This shows that $\pi_0 N_h \bbX$ is the nerve of a category, which we
    denote by $\Pi_0 \bbX$.
The functor  $\pi_0$ induces a functor
    ${\pi}_0^*:\Ps[\dop,\Cat]\rw[\dop,\Set]$. In particular, if $F$ is
    a morphism in $(\WGDbl)_{\ps}$, ${\pi}_0^* F$ is a morphism between
    nerves of categories. This defines $\Pi_0$ on morphisms.

    \medskip

  \item [ii)] This follows immediately by considering the functor $\gamma(\pt_0,\pt_1):\bbX_{1}\rw
  \bbX^d_{0}\times \bbX^d_{0}$, since $\bbX^d_{0}$ is discrete.

  \medskip

  \item [iii)] Since $F$ is a pseudo-natural transformation, there
  is a pseudo-commutative diagram
  \begin{equation*}
    \xymatrix@R=1.5em{
    \bbX_1\ar^(0.4){(\pt_0,\pt_1)}[rr] \ar^{}[d] &&
    \bbX_0\times \bbX_0\ar^{}[d]\\
    \bbY_1\ar_(0.4){(\pt'_0,\pt'_1)}[rr] && \bbY_0\times \bbY_0
    }
  \end{equation*}
and therefore, since $\bbX_0^d$ and $\bbY_0^d$ are discrete, a commutative
diagram
  \begin{equation*}
    \xymatrix@R=1.5em{
    \bbX_1\ar^(0.4){\gamma_{\bbX}(\pt_0,\pt_1)}[rr] \ar^{}[d] &&
    \bbX_0^d\times \bbX_0^d\ar^{}[d]\\
    \bbY_1\ar_(0.4){\gamma_{\bbY}(\pt'_0,\pt'_1)}[rr] && \bbY_0^d\times \bbY_0^d
    }
  \end{equation*}
This determines the functor $F_{(a,b)}\colon\bbX_{(a,b)}\rw \bbY_{(Fa,Fb)}$ for all $a,b\in
\bbX^d_{0}$.
\end{itemize}
\end{proof}

\begin{definition}\label{s4.def2}\rm
    We say that a morphism $F:\bbX\rw \bbY$ in $(\WGDbl)_{\ps}$ is a
    2-equivalence if
    \begin{itemize}
      \item [i)] For all $a,b\in \bbX^d_{0}$, $F_{(a,b)}:\bbX_{(a,b)}\rw
      \bbY_{(Fa,Fb)}$ is an equivalence of categories.
      \item [ii)] $\Pi_0 F$ is an equivalence of categories, where
      $\Pi_0$ is as in Lemma \ref{s4.lem2}
    \end{itemize}
\end{definition}

\begin{proposition}\label{s4.pro3}
    The functors $Q$ and $D$  preserve
    2-equivalences and induce an equivalence of categories between
    the localizations with respect to the 2-equivalences:
    \begin{equation*}
        (\WGDbl)_{\ps}\bsim^2\,\;\;\simeq\;\;(\Ta)_{\ps}\bsim^2
    \end{equation*}
\end{proposition}

\begin{proof}
The fact that $D$ preserves 2-equivalences is immediate from the definitions.
Let $G:X\rw Y$ be a 2-equivalence in $(\Ta)_{\ps}$. By Corollary \ref{s4.cor1},
there is a pseudo-commutative diagram in $\Ps[\dop,\Cat]$
\begin{equation}\label{s4.eq2}
    \xymatrix@R=6pt{
    N_h QX\ar^{\alpha_X}[rr] \ar_{N_h QF}[dd] && X \ar^F[dd]\\
    & \sim &\\
    N_h QY \ar_{\alpha_Y}[rr] && Y
    }
\end{equation}
in which $\alpha_X$ and $\alpha_Y$ are levelwise categorical equivalences. Applying
the functor ${\pi}_0^* :\Ps[\dop,\Cat]\rw[\dop,\Set]$ we obtain a commutative
diagram in $[\dop,\Set]$
\begin{equation*}
    \xymatrix@R=1.8em{
    {\pi}_0^* N_h QX\ar^{{\pi}_0^* \alpha_X}[rr] \ar_{{\pi}_0^* N_h QF}[d] &&
    {\pi}_0^* X \ar^{{\pi}_0^* F}[d]\\
    {\pi}_0^* N_h QY \ar_{{\pi}_0^* \alpha_Y}[rr] && {\pi}_0^* Y.
    }
\end{equation*}
Recalling that ${\pi}_0^* N_hQX=N\Pi_0 QX$, ${\pi}_0^* X=N\Pi_0 X$ and
similarly for the other terms, and applying the functor $P:[\dop,\Set]\rw\Cat$
which is left adjoint to the nerve, we obtain the commutative diagram in $\Cat$
\begin{equation}\label{s4.eq3}
    \xymatrix@R=1.6em{
    \Pi_0 Q X \ar^{P{\pi}_0^* \alpha_X}[rr] \ar_{\Pi_0 Q F}[d] &&
    \Pi_0 X \ar^{\Pi_0 F}[d]\\
    \Pi_0 Q Y \ar_{P{\pi}_0^* \alpha_Y}[rr] && \Pi_0 Y.
    }
\end{equation}
Since $\alpha_X$ and $\alpha_Y$ are levelwise categorical equivalences, ${\pi}_0^*
\alpha_X,$ and ${\pi}_0^* \alpha_Y$ are isomorphisms, hence $P{\pi}_0^* \alpha_X$
and $P{\pi}_0^* \alpha_Y$ are isomorphisms. Since $F$ is a 2-equivalence in
$(\Ta)_{\ps}$, by definition $\Pi_0 F$ is an equivalence of categories. Hence
the commutativity of (\ref{s4.eq3}) implies that $\Pi_0 Q F$ is an equivalence
of categories. Also, for each $a,b\in(N_h QX)^d_{0*}\cong X_{0*}$, by
(\ref{s4.eq2}) we obtain a pseudo-commutative diagram in $\Cat$
\begin{equation}\label{s4.eq4}
    \xymatrix@R=6pt{
    (QX)_{(a,b)} \ar^{(\alpha_X)_{(a,b)}}[rr] \ar_{(QF)_{(a,b)}}[dd] &&
    X_{(a,b)}\ar^{F_{(a,b)}}[dd]\\
    & \sim &\\
    (QY)_{(a,b)} \ar_{(\alpha_Y)_{(a,b)}}[rr] && Y_{(a,b)}.
    }
\end{equation}
But $(\alpha_X)_{(a,b)}$ and $(\alpha_Y)_{(a,b)}$ are categorical equivalences
because $(\alpha_X)_1$ and $(\alpha_Y)_1$ are categorical equivalences. Also
$F_{(a,b)}$ is a categorical equivalence because $F$ is a 2-equivalence in
$(\Ta)_{\ps}$. Hence by (\ref{s4.eq4}), $(QF)_{(a,b)}$ is a categorical
equivalence. In conclusion, $QF$ is a 2-equivalence in $(\WGDbl)_{\ps}$.

Since $D$ and $Q$ preserve 2-equivalences, they induce functors between
localizations
\begin{equation*}
    \ovl{D}:(\WGDbl)_{\ps}\bsim \,\rightleftarrows\;(\Ta)_{\ps}\bsim\;
    :\ovl{Q}
\end{equation*}
\smallskip

 By Corollary \ref{s4.cor1} and by Proposition \ref{s4.pro2}, for every
$X\in(\Ta)_{\ps}$ there is a composite morphism $DQX\rw N_h QX\rw X$ in
$\Ps[\dop,\Cat]$ which is a levelwise categorical equivalence, hence a
2-equivalence in $(\Ta)_{\ps}$.

It follows that $\ovl{D}\,\ovl{Q}X\cong X$. Similarly, for each $\bbY\in
(\WGDbl)_{\ps}$ there is a composite morphism $N_h QD \bbY\rw D\bbY\rw N_h \bbY$
in $\Ps[\dop,\Cat]$ which is a levelwise categorical equivalence and hence a 2-equivalence. It follows that
$\ovl{Q}\,\ovl{D}\bbY\cong \bbY$. This shows that $(\ovl{D}, \;\ovl{Q})$ is an
equivalence of categories.
\end{proof}

\section{Bicategories and Double Categories}\label{bicat}

\subsection{Bicategories and Tamsamani weak 2-categories}

In this section we recall some results from \cite{lp}. We consider the
2-category $\NHom$ whose objects are bicategories, whose morphisms are normal
homomorphisms and whose 2-cells are icons; the latter are oplax natural
transformations with identity components. The fully faithful inclusion $J:
\Delta\rw\NHom$  gives rise to a 2-nerve functor
\begin{equation*}
    N:\NHom\rw[\dop,\Cat]\;,
\end{equation*}
\begin{equation*}
    (N\calB)_n=\NHom([n], \calB)\;.
\end{equation*}
It is shown in \cite{lp} that $N$ is fully faithful and that the 2-nerve of a
bicategory is in fact a Tamsamani weak 2-category. Given a bicategory $\calB$,
$(N\calB)_0$ is the discrete category with objects the objects of $\calB$. An
object of $(N\calB)_1$ is a morphism of $\calB$ while a morphism in $(N\calB)_1$
is a 2-cell in $\calB$. A complete characterization of 2-functors
$X:\dop\rw\Cat$ which are 2-nerves of bicategories is given in \cite[Theorem
7.1]{lp}.

The 2-nerve functor $N$ has a left 2-adjoint $G$, which was defined in
\cite{tam}. Given a Tamsamani weak 2-category $X$, the objects of $GX$ are the
element of $X_0$, the 1 and 2-cells are the objects and morphisms of $X_1$ and
the vertical composition of 2-cells is the composition in $X_1$.

Since the Segal map $\eta_2: X_2\rw\tens{X_1}{X_0}$ is an equivalence, we can
choose a functor $M:\tens{X_1}{X_0}\rw X_1$ and an isomorphism $\sigma:d_1\cong
M\eta_2$ as follows:

$$
\xymatrix@!=1.5pc{
X_2\ar[dd]_{d_1}\ar[rr]^(0.4){\eta_2}\drtwocell<\omit>{<0>\sigma}&&
\tens{X_1}{X_0}\ar[ddll]^{M}
\\
&\
\\
X_1 && }
$$
This gives the composition of 1-cells and the horizontal composition of
2-cells.

The identity isomorphisms are $\sigma s_0$, $\sigma s_1$ (where $s_0,s_1:X_0\rw
X_1$ are the degeneracy maps). For the associativity isomorphisms, one needs to
consider the following pasting diagrams, where we denote
$X_1^k=\pro{X_1}{X_0}{k}$, for $k=2,3$.
$$
\xymatrix@!=1.0pc{
 X_3\ar[rr]^(0.4){\big(\stackrel{d_0}{\scriptscriptstyle{d_2d_2}}\big)}\ar[dd]_{d_2}
    && \tens{X_2}{X_0}\ar[dd]_{d_1\times 1}\ar[rr]^(0.6){\eta_2\times 1}
					\drtwocell<\omit>{<0>{\quad\sigma\times 1}}
    && X_1^3 \ar[ddll]^{M\times 1}
 \\
 &&&\
 \\
 X_2 \ar[rr]^{\eta_2}\ar[dd]_{d_1}\drtwocell<\omit>{<0>{\sigma}} && X_1^2 \ar[ddll]^M &&
 \\
 &\
 \\
 X_1
 }
 \qquad
\xymatrix@!=1.0pc{
 X_3\ar[rr]^(0.4){\big(\stackrel{d_0d_1}{\scriptscriptstyle{d_3}}\big)}\ar[dd]_{d_1}&&
 \tens{X_2}{X_0}\ar[dd]_{1\times d_1 }\ar[rr]^(0.6){1\times \eta_2 }\drtwocell<\omit>{<0>{\quad 1\times\sigma }}&& X_1^3 \ar[ddll]^{1\times M }
 \\
 &&&\
 \\
 X_2 \ar[rr]^{\eta_2}\ar[dd]_{d_1}\drtwocell<\omit>{<0>{\sigma}} && X_1^2 \ar[ddll]^M &&
 \\
 &\
 \\
 X_1
 }
$$
Since the left hand composites of the diagrams are equal and the top composites are both equal to
the equivalence $\eta_3:X_3\rw X^3_1$, there is a unique invertible cell
$M(M\times 1)\cong M(1\times M)$ which pasted onto the left diagram gives the
right diagram. The proof of the coherence laws uses the fact that $\eta_4$ is
an equivalence, see \cite{lp} for details.

The relation between the functors $N$ and $G$ is summarized in the following
theorem

\begin{theorem}\rm[6, Theorem 7.2].\ \label{the7.2}\it
    The 2-nerve 2-functor $N : \NHom \rw \Ta$, seen as landing in the 2-category
$\Ta$ of Tamsamani weak 2-categories, has a left 2-adjoint given by $G$. Since $N$
is fully faithful, the counit $GN\rw 1$ is invertible. Each component $u : X
\rw NGX$ of the unit is a pointwise equivalence, and $u_0$ and $u_1$ are
identities.
\end{theorem}

A morphism $f:X\rw Y$ in $\Ta$ is a 2-equivalence if and only if $Gf$ is a
biequivalence of bicategories. It is not hard to see that inverting these maps
in $\NHom$ and in $\Ta$ gives equivalent categories. Another approach consists
in enlarging the morphisms in $\Ta$ to include pseudo-natural transformations.
One then obtains:

\begin{theorem}\rm[6, Theorem 7.3].\ \label{the7.3}\it
    The 2-nerve 2-functor $N:\NHom\rw(\Ta)_{\ps}$ is a biequivalence of
    2-categories.
\end{theorem}

\subsection{The fundamental bicategory}\label{fundabicat}
The functor $G\colon (\mbox{Ta}_2)_{\ps}\rightarrow \NHom$ from \cite{lp} 
can be composed with the discretization functor $D\colon \mbox{\sf WGDbl}\rightarrow \mbox{\sf Ta}_2$
to obtain a functor $GD\colon \mbox{\sf WGDbl}_{\ps}\rightarrow \NHom$
which associates a bicategory $\mathcal X$ to a weakly globular double category $\mathbb X$.
This bicategory $\mathcal X$ is obtained by taking the connected components of the vertical arrow category 
of the weakly globular double category $\mathbb X$ as the objects, and is also called the
{\em fundamental bicategory} of $\mathbb X$. We 
write $\Bic\colon  \mbox{\sf WGDbl}\rightarrow \NHom$ for the composite $GD$.

Since the relationship between a weakly globular double category $\bbX$ and its fundamental bicategory $\Bic\bbX$
plays an important role in this paper, we spell out how the arrows and 2-cells of $\Bic\bbX$ 
are obtained from the horizontal and vertical arrows and double cells in $\bbX$.

As noted above, the {\em objects} of $\Bic\bbX$ are obtained as the connected components $\pi_0\bbX_0$
of 
the vertical arrow category $\bbX_0$.
When $A$ is an object of $\bbX$, i.e., an element of $\bbX_{00}$,
we write $\bar{A}$ for the corresponding object in $\Bic\bbX$.
Note that  $\bar{A}=\bar{B}$ if and only if there is a (unique) vertical arrow 
$v\colon \xymatrix@1{{A}\ar[r]|\bullet&B}$ in $\bbX$ 
(since the vertical arrow category $\bbX_0$ is posetal and groupoidal).

For any two objects $\bar{A}$ and $\bar{B}$
in $\bbX$, the {\em set of arrows}, $\Hom_{\smBic\bbX}(\bar{A},\bar{B})$
is obtained as a disjoint union of horizontal hom-sets in $\bbX$,
$$
\Hom_{\smBic\bbX}(\bar{A},\bar{B})
=\coprod_{\scriptstyle \begin{array}{c}\bar{A'}=\bar{A}\\\bar{B'}=\bar{B}\end{array}}\Hom_{\bbX,h}(A',B').
$$
Note that we do not put an equivalence relation on the horizontal arrows
of $\bbX$ to obtain the arrows of the fundamental bicategory; we will use the same symbol to denote 
a horizontal arrow in $\bbX$ and the corresponding arrow in $\Bic(\bbX)$.

For any two arrows $\xymatrix@1{\bar{A}\ar@<.5ex>[r]^f\ar@<-.5ex>[r]_g&\bar{B}}$  
in $\Bic\bbX$ represented by horizontal arrows $\xymatrix@1{A_1\ar[r]^f&B_1}$
and $\xymatrix@1{A_2\ar[r]^g&B_2}$ in $\bbX$, the {\em 2-cells} from $f$ to $g$ 
correspond to double cells of the form 
$$
\xymatrix{
A_1\ar[d]|\bullet_{v}\ar[r]^f\ar@{}[dr]|\alpha &B_1\ar[d]|\bullet^{w}
\\
A_2\ar[r]_g & B_2\rlap{\quad.}
}
$$
Since $v$ and $w$ are unique, we will denote the corresponding 2-cell in $\Bic\bbX$ by
$\alpha\colon f\Rightarrow g$.

Let  $f\colon A_1\rightarrow B_1$ and $g\colon B_2\rightarrow C_2$  be horizontal arrows
such that there is an invertible vertical arrow $\xymatrix@1@C=1.5em{v\colon B_2\ar[r]|-\bullet &B_1}$.
By an argument as given in the proof of Lemma \ref{squarecompln}
there are horizontal arrows $f_3\colon A_3\rightarrow B_3$ and $g_3\colon B_3\rightarrow C_3$
with vertically invertible double cells $\varphi_{f_3,f}$ and $\varphi_{g_3,g}$
as in the following diagram
$$
\xymatrix@R=2em{
A_3\ar[dd]|\bullet_{x}\ar@{}[ddr]|{\varphi_{f_3,f}}\ar[r]^{f_3} 
	& B_3\ar[d]|\bullet^y\ar@{}[dr]|{\varphi_{g_3,g}}\ar[r]^{g_3} & C_3\ar[d]|\bullet^z
\\
& B_2\ar[d]|\bullet^{v}\ar[r]_g& C_2 
\\
A_1\ar[r]_f & B_1
}
$$
The composition of $f\colon \bar{A}_1\rightarrow\bar{B}_1$ and $g\colon\bar{B}_2\rightarrow\bar{C}_2$ 
(where $\bar{B}_1=\bar{B}_2$) in $\Bic\bbX$ is defined to be the horizontal composite
$g_3\circ f_3\colon\bar{A}_3\rightarrow\bar{C}_3$.

The horizontal composition of 2-cells is defined as follows:
Let 
$$\xymatrix@R=2em{\bar{A}\ar@/^2ex/[r]^{f}\ar@{}[r]|{\Downarrow\alpha}\ar@/_2ex/[r]_g&
	\bar{B} \ar@/^2ex/[r]^{h}\ar@{}[r]|{\Downarrow\beta}\ar@/_2ex/[r]_k& \bar{C}}$$
be a diagram of arrows and cells in the fundamental bicategory, represented by 
double cells in $\bbX$,
$$
\xymatrix@R=2em{
A_2 \ar[d]|\bullet_{u_{21}} \ar[r]^f \ar@{}[dr]|\alpha &
	B_2 \ar[d]|\bullet^{v_{21}} \ar@{}[drr]|{\mbox{and}} &&
	B_4 \ar[d]|\bullet_{v_{43}}\ar[r]^h\ar@{}[dr]|\beta & C_4\ar[d]|\bullet^{z}
\\
A_1\ar[r]_g & B_1 && B_3\ar[r]_k & C_3.}
$$
Let the composite of $g$ and $k$ in $\Bic\bbX$ be the arrow $k_5\circ g_5$ 
as in the diagram
$$
\xymatrix@R=2em{
A_5\ar[dd]|\bullet_{u_{51}}\ar@{}[ddr]|{\varphi_{g_5,g}}\ar[r]^{g_5} 
	& B_5\ar[d]|\bullet_{v_{53}}\ar@{}[dr]|{\varphi_{k_5,k}}\ar[r]^{k_5} 
	& C_5\ar[d]|\bullet^{w_{53}}
\\
& B_3\ar[d]|\bullet^{v_{31}}\ar[r]_k & C_3
\\
A_1\ar[r]_g & B_1
}
$$
and let the composite of $f$ and $h$ be the arrow $h_6\circ f_6$ as in the diagram
$$
\xymatrix@R=2em{
A_6\ar[dd]|\bullet_{u_{62}}\ar@{}[ddr]|{\varphi_{f_6,f}}\ar[r]^{f_6} 
	& B_6\ar[d]|\bullet_{v_{64}}\ar@{}[dr]|{\varphi_{h_6,h}}\ar[r]^{h_6} 
	& C_6\ar[d]|\bullet^{w_{64}}
\\
& B_4\ar[d]|\bullet^{v_{42}}\ar[r]_h & C_4
\\
A_2\ar[r]_f & B_2
}
$$
Then the composition of $\alpha$ and $\beta$ is represented by the following 
pasting of double cells:
$$
\xymatrix@R=2em{
A_6 \ar[ddd]|\bullet_{u_{62}} \ar@{}[dddr]|{\varphi_{f_6,f}} \ar[r]^{f_6} 
	& B_6\ar[d]|\bullet_{v_{64}} \ar@{}[dr]|{\varphi_{h_6,h}} \ar[r]^{h_6} 
	& C_6\ar[d]|\bullet^{w_{64}}
\\
&B_4\ar[d]|\bullet^{v_{43}}\ar@{}[dr]|{\beta} \ar[r]_h & C_4\ar[d]|\bullet^{w_{43}}
\\
&B_3\ar[r]_k\ar[d]|\bullet^{v_{32}} \ar@{}[dddr]|{\varphi^{-1}_{k_5,k}} & C_3\ar[ddd]|\bullet^{v_{35}}
\\
A_2\ar[d]|\bullet_{u_{21}}\ar@{}[dr]|{\alpha} \ar[r]^f & B_2\ar[d]|\bullet_{v_{21}}
\\
A_1\ar[d]|\bullet_{u_{15}}\ar@{}[dr]|{\varphi_{g_5,g}^{-1}} \ar[r]_g & B_1\ar[d]|\bullet^{v_{15}}
\\
A_5\ar[r]_{g_5}&B_5\ar[r]_{k_5} & C_5.
}
$$
(Here, $u_{ij}=u_{ji}^{-1}$ and $u_{jk}\cdot u_{ij}=u_{ik}$, and analogous for $v$ and $w$, 
since the vertical category is groupoidal posetal.
Furthermore, the same holds for the cells, because they are components of a vertical transformation.)

The units for the composition are obtained from the functor $\mu_0\colon\bbX_0^d\rightarrow\bbX_0$ 
which is part of the equivalence of
categories, $\bbX_0^d\simeq \bbX_0$. For an object $\bar{A}$ in $\Bic\bbX$, $1_{\bar{A}}$ 
is the horizontal arrow 
$\mbox{Id}_{\mu_0(\bar{A})}$. 

There are associativity and isomorphisms for this composition that satisfy the usual coherence conditions 
by the results in \cite{tam} and \cite {lp}.
This can also be found in \cite{Simpson}.
However, we won't explicitly use this, so we won't spell out the details.
 
\subsection{The associated double category}

The composition of the functors $$\xymatrix{\NHom\ar[r]^N&(\Ta)_{\ps}\ar[r]^Q&\WGDbl_{\ps}}$$
gives us a pseudo-inverse $\NHom\rightarrow \WGDbl_{\ps}$ 
to the fundamental bicategory functor $\Bic$. We will call this pseudo-inverse $\mbox{\bf Dbl}$.
Note that this is a pseudo-inverse in the sense that $\mbox{\bf Dbl}\circ\Bic\simeq \mbox{Id}_{\smWGDbl}$
and $\Bic\circ\mbox{\bf Dbl}\simeq \mbox{Id}_{\smBicat}$, but these functors are not adjoint, since
$D$ and $Q$ are not adjoint as shown in Remark \ref{notadj}, but $G$ and $N$ form a biadjoint biequivalence  
as proved in \cite{lp}.

Both $\Bic$ and $\Dbl$ are useful in translating results from weakly globular double categories to bicategories 
and back, so we also give an explicit description of the arrows and cells of $\Dbl(\calB)$ for a bicategory $\calB$.
This construction shows also in detail how one builds a strict structure out of a weak one by adding 
a second class of arrows.
Before we begin the construction of $\Dbl(\calB)$, we first choose a composite $\varphi_{f_{1},\ldots,f_{n}}$ for each finite path 
$\xymatrix@1@C=1.3em{A_0\ar[r]^{f_{1}}&A_1\ar[r]^{f_{2}}&\cdots\ar[r]^{f_{n}}&A_n}$ of 
(composable) arrows in
$\calB$. If the path is empty, we take $\varphi_{A_0}=1_{A_0}$.
For each path of such paths, 
$$
\xymatrix@C=3.85em{(\ar[r]^{f_{i_1}}&\cdots\ar[r]^{f_{i_{n_1}}}&)(\ar[r]^{f_{i_{n_1+1}}}
    &\cdots\ar[r]^{f_{i_{n_2}}}&)
    \quad\cdots\quad (\ar[r]^-{f_{i_{n_{m-1}+1}}} &\cdots\ar[r]^{f_{i_{n_m}}}&)}
$$
the associativity and unit cells give rise to unique invertible comparison 2-cells, which we denote by
$$
\xymatrix@C=20em{
\ar@/^5ex/[r]^{\varphi_{\varphi_{f_{i_1},\ldots,f_{i_{n_1}}},\varphi_{f_{i_{n_1+1}},\ldots,f_{i_{n_2}}},\ldots,
    \varphi_{f_{i_{n_{m-1}+1}},\ldots,f_{i_{n_m}}}}} \ar@/_5ex/[r]_{\varphi_{f_{i_1},\ldots,f_{i_{n_m}}}}
    \ar@{}[r]|{\Phi_{f_{i_1}\cdots f_{i_{n_1}},f_{i_{n_1+1}}\cdots f_{i_{n_2}},\ldots,f_{i_{n_{m-1}+1}}\cdots f_{i_{n_m}}}} &\rlap{\quad.}
}
$$
(The uniqueness of these cells follows from the associativity and unit coherence axioms.)
With these chosen composites and cells, we will walk through the constructions corresponding to the functors 
in the composition $\Dbl=QN=PRN=P\St \,SN$. The 2-nerve 
$N\calB\colon\dop\rightarrow\Cat$ has the following components:
\begin{itemize}
\item $(N\calB)_0$ is the discrete category with objects $\calB_0$.
\item $(N\calB)_1$ has objects $\calB_1$, i.e., the arrows of $\calB$, and arrows the 2-cells of $\calB$.
\item $(N\calB)_2$ has objects diagrams of the form
$$\xymatrix{\ar[r]^f\ar@<-1.5ex>@{}[rr]|{\alpha\cong}\ar@/_3ex/[rr]_h&\ar[r]^g&}$$
in $\calB$, and arrows cylinders between such diagrams, i.e., an arrow from $(f,g,h,\alpha)$ to
$(f',g',h',\alpha')$ is a triple $(\varphi,\gamma,\theta)$
of 2-cells, $\varphi\colon f\Rightarrow f'$, $\gamma\colon g\Rightarrow g'$, and $\theta\colon h\Rightarrow h'$,
such that $\theta\cdot\alpha=\alpha'\cdot(\gamma\circ\varphi)$,
$$
\xymatrix@R=3em{
\ar[r]^f \ar@<-1.5ex>@{}[rr]|{\alpha\cong}\ar@/_3ex/[rr]_h \ar@{=}[d]  &\ar[r]^g\ar@{}[d]|(1){\theta\Downarrow}
    &\ar@{=}[d] \ar@{}[dr]|{\textstyle =}& \ar@{=}[d] \ar[r]^f \ar@{}[dr]|{\varphi \Downarrow}
    & \ar@{=}[d]\ar@{}[dr]|{\gamma\Downarrow} \ar[r]^g &\ar@{=}[d]
\\
\ar@/_3ex/[rr]_{h'}&&&\ar[r]^{f'}\ar@<-1.5ex>@{}[rr]|{\alpha'\cong}\ar@/_3ex/[rr]_{h'} 
    & \ar[r]^{g'} &
}
$$
\end{itemize}
The pseudo-functor $SN\calB\dop\rightarrow\Cat$ has then
\begin{itemize}
 \item $(SN\calB)_0$ is the discrete category on the objects of $\calB$;
\item $(SN\calB)_1$ has as objects the arrows of $\calB$ and as arrows the 2-cells of $\calB$;
\item $(SN\calB)_2$ has as objects paths of length 2 in $\calB$, $\xymatrix@1{\ar[r]^{f_1}&\ar[r]^{f_2}&}$
and as arrows horizontal paths of 2-cells of length 2 in $\calB$, 
$\xymatrix@1{\ar@/^2ex/[r]^{f_1}\ar@/_2ex/[r]_{g_1}\ar@{}[r]|{\alpha_1}
    & \ar@/^2ex/[r]^{f_2}\ar@/_2ex/[r]_{g_2}\ar@{}[r]|{\alpha_2}&}
$.
\item $(SN\calB)_n$ has as objects paths of length $n$ in $\calB$,
 $\xymatrix@1{\ar[r]^{f_1}&\ar[r]^{f_2}&\ar@{..}[r]&\ar[r]^{f_n}&}$
and as arrows horizontal paths of 2-cells of length $n$ in $\calB$, 
$\xymatrix@1{\ar@/^2ex/[r]^{f_1}\ar@/_2ex/[r]_{g_1}\ar@{}[r]|{\alpha_1}
    & \ar@/^2ex/[r]^{f_2}\ar@/_2ex/[r]_{g_2}\ar@{}[r]|{\alpha_2}&\ar@{..}[r]& 
	\ar@/^2ex/[r]^{f_n}\ar@/_2ex/[r]_{g_n}\ar@{}[r]|{\alpha_n}&}
$.
\end{itemize}
By the construction described in Remark \ref{rem2add} and taking the pseudo-inverse of the horizontal nerve,
we obtain the following description of the double category $\Dbl(\calB)$.

The {\em objects} of $\Dbl(\calB)$ are given as pairs of an arrow $\psi\colon[0]\rightarrow [n]$ in 
$\Delta$ with a path, $\xymatrix@1{A_0\ar[r]^{f_1} & A_2\ar[r]^{f_2}& \cdots\ar[r]^{f_n}&A_n}$, 
of length $n$ in $\calB$, for all $n$.
Since the arrow $\psi$ is determined by its image $i_0=\psi(0)\in[n]$, we will denote this object in
$\Dbl(\calB)$ by $$(\xymatrix{A_0\ar[r]^{f_1} & A_2\ar[r]^{f_2}& \cdots\ar[r]^{f_n}&A_n};i_0)$$
and think of $A_{i_0}$ as a marked object along the path. So we will also use the notation
$$\xymatrix{A_0\ar[r]^{f_1} & A_2\ar[r]^{f_2}&\cdots\ar[r]^{f_{i_0}}
  &[A_{i_0}]\ar[r]^{f_{i_0+1}}& \cdots\ar[r]^{f_n}&A_n\rlap{\,.}}$$

There is a unique {\em vertical arrow} from 
$\xymatrix@1@C=1.8em{A_0\ar[r]^{f_1} & A_2\ar[r]^{f_2}&\cdots\ar[r]^{f_{i_0}}
  &[A_{i_0}]\ar[r]^{f_{i_0+1}}& \cdots\ar[r]^{f_n}&A_n}$ to 
$\xymatrix@1@C=2em{B_0\ar[r]^{g_1} & B_2\ar[r]^{g_2}&\cdots\ar[r]^{g_{j_0}}
  &[B_{j_0}]\ar[r]^{g_{j_0+1}}& \cdots\ar[r]^{g_m}&A_m}$ if and only if $A_{i_0}=B_{j_0}$.
In diagrams we will include this vertical arrow in the following way
$$\xymatrix{
A_0\ar[r]^{f_1} & A_2\ar[r]^{f_2}&\cdots\ar[r]^{f_{i_0}}
  &[A_{i_0}]\ar@{=}[d]\ar[r]^{f_{i_0+1}}& \cdots\ar[r]^{f_n}&A_n
\\
B_0\ar[r]^{g_1} & B_2\ar[r]^{g_2}&\cdots\ar[r]^{g_{j_0}}
  &[B_{j_0}]\ar[r]^{g_{j_0+1}}& \cdots\ar[r]^{g_m}&A_n
}
$$

{\em Horizontal arrows} in $\Dbl(\calB)$ are given as  pairs of an arrow 
$\psi\colon[1]\rightarrow [n]$ in $\Delta$ with a path of length $n$
in $\calB$, for all $n$. Analogous to what we did for objects we denote horizontal arrows by
$$(\xymatrix{A_0\ar[r]^{f_1} & A_2\ar[r]^{f_2}& \cdots\ar[r]^{f_n}&A_n};i_0,i_1)\quad\mbox{ with }\quad i_0\le i_1,$$
or by
$$\xymatrix{A_0\ar[r]^{f_1} & A_2\ar[r]^{f_2}&\cdots\ar[r]^{f_{i_0}}&[A_{i_0}]\ar[r]^{f_{i_0+1}}& \cdots
\ar[r]^{f_{i_1}}&[A_{i_1}]\ar[r]^{f_{i_1+1}}& \cdots\ar[r]^{f_n}&A_n\rlap{\,.}}$$
The domain of $(\xymatrix@C=1.2em@1{A_0\ar[r]^{f_1} & A_2\ar[r]^{f_2}& \cdots\ar[r]^{f_n}&A_n};i_0,i_1)$
is $(\xymatrix@1@C=1.2em{A_0\ar[r]^{f_1} & A_2\ar[r]^{f_2}& \cdots\ar[r]^{f_n}&A_n};i_0)$
and the codomain is $(\xymatrix@C=1.2em{A_0\ar[r]^{f_1} & A_2\ar[r]^{f_2}& \cdots\ar[r]^{f_n}&A_n};i_1)$.
For a horizontal identity arrow,
$$
(\xymatrix{A_0\ar[r]^{f_1} & A_2\ar[r]^{f_2}& \cdots\ar[r]^{f_n}&A_n};i_0,i_0)
$$
we will also use the notation
$$
\xymatrix{A_0\ar[r]^{f_1} & A_2\ar[r]^{f_2}& \ar[r]^{f_{i_0}}&[A_{i_0}]\ar@{=}[r] &[A_{i_0}]
\ar[r]^{f_{i_0+1}}& \cdots\ar[r]^{f_n}&A_n}
$$
when this makes it easier to fit such an arrow into a diagram representing a double cell as shown below.

A {\em double cell} consists of two horizontal arrows 
$$(\xymatrix@C=1.7em{A_0\ar[r]^{f_1} & A_2\ar[r]^{f_2}& \cdots\ar[r]^{f_n}&A_n};i_0,i_1)\mbox{ and }
(\xymatrix@C=1.7em{B_0\ar[r]^{g_1} & B_2\ar[r]^{g_2}& \cdots\ar[r]^{g_m}&B_m};j_0,j_1)$$
(for the vertical domain and codomain respectively), such that 
$A_{i_0}=B_{j_0}$ and $A_{i_1}=B_{j_1}$ (such that there are unique vertical arrows between the domains
of these arrows and between the codomains of these arrows),
together with a 2-cell in $\calB$ between the chosen composites,
$$
\xymatrix@C=5em{\ar@/^4ex/[r]^{\varphi_{f_{i_0+1},\ldots,f_{i_1}}} 
      \ar@/_4ex/[r]_{\varphi_{g_{j_0+1},\ldots,g_{j_1}}}
      \ar@{}[r]|{\Downarrow\alpha}&\rlap{\quad.}}
$$
We combine all this information together in the following diagram
$$
\xymatrix@R=4em{
A_0\ar[r]^{f_1}&\cdots\ar[r]^{f_{i_0}}
  &[A_{i_0}]\ar@/_2ex/[rr]_{\varphi_{f_{i_0+1},\ldots,f_{i_1}}}\ar@{}[drr]|{\alpha}\ar@{=}[d]\ar[r]^{f_{i_0+1}}
  &\cdots\ar[r]^{f_{i_1}}
  &[A_{i_1}]\ar@{=}[d]\ar[r]^{f_{i_1+1}}
  &\cdots\ar[r]^{f_n}&A_n
\\
B_0\ar[r]_{g_1}&\cdots\ar[r]_{g_{j_0}}
  &[B_{j_0}] \ar@/^2ex/[rr]^{\varphi_{g_{j_0+1},\ldots,g_{j_1}}}\ar[r]_{g_{j_0+1}}
  &\cdots\ar[r]_{g_{j_1}}
  &[B_{j_1}]\ar[r]_{g_{j_1+1}}
  &\cdots\ar[r]_{g_m}&B_m
}
$$

Two horizontal arrows, $$(\xymatrix@1{A_0\ar[r]^{f_1} & A_2\ar[r]^{f_2}& \cdots\ar[r]^{f_n}&A_n};i_0,i_1)
\mbox{ and }(\xymatrix@1{B_0\ar[r]^{g_1} & B_2\ar[r]^{g_2}& \cdots\ar[r]^{g_n}&A_n};j_0,j_1),$$ are composable if
and only if the two paths are the same, i.e., $m=n$, $A_i=B_i$ and $f_i=g_i$ for all $i=1,\ldots,n$, 
and furthermore, $i_1=j_0$. In that case, the {\em horizontal composition} of these arrows is given by
$(\xymatrix@1{A_0\ar[r]^{f_1} & A_2\ar[r]^{f_2}& \cdots\ar[r]^{f_n}&A_n};i_0,j_1)$.

The {\em horizontal composition of double cells}
$$
\xymatrix@R=4em{
A_0\ar[r]^{f_1}&\cdots\ar[r]^{f_{i_0}}
  &[A_{i_0}]\ar@/_2ex/[rr]_{\varphi_{f_{i_0+1},\ldots,f_{i_1}}}\ar@{}[drr]|{\alpha}\ar@{=}[d]\ar[r]^{f_{i_0+1}}
  &\cdots\ar[r]^{f_{i_1}}
  &[A_{i_1}]\ar@{=}[d]\ar[r]^{f_{i_1+1}}
  &\cdots\ar[r]^{f_n}&A_n 
\\
B_0\ar[r]_{g_1}&\cdots\ar[r]_{g_{j_0}}
  &[B_{j_0}] \ar@/^2ex/[rr]^{\varphi_{g_{j_0+1},\ldots,g_{j_1}}}\ar[r]_{g_{j_0+1}}
  &\cdots\ar[r]_{g_{j_1}}
  &[B_{j_1}]\ar[r]_{g_{j_1+1}}
  &\cdots\ar[r]_{g_m}&B_m
}
$$
and
$$
\xymatrix@R=4em{
A_0\ar[r]^{f_1}&\cdots\ar[r]^{f_{i_1}}
  &[A_{i_1}]\ar@/_2ex/[rr]_{\varphi_{f_{i_1+1},\ldots,f_{i_2}}}\ar@{}[drr]|{\beta}\ar@{=}[d]\ar[r]^{f_{i_1+1}}
  &\cdots\ar[r]^{f_{i_2}}
  &[A_{i_2}]\ar@{=}[d]\ar[r]^{f_{i_2+1}}
  &\cdots\ar[r]^{f_n}&A_n 
 \\
 B_0\ar[r]_{g_1}&\cdots\ar[r]_{g_{j_1}}
  &[B_{j_1}] \ar@/^2ex/[rr]^{\varphi_{g_{j_1+1},\ldots,g_{j_2}}}\ar[r]_{g_{j_1+1}}
  &\cdots\ar[r]_{g_{j_2}}
  &[B_{j_2}]\ar[r]_{g_{j_2+1}}
  &\cdots\ar[r]_{g_m}&B_m
 }
 $$
 is defined to be
 $$
\xymatrix@R=4em{
A_0\ar[r]^{f_1}&\cdots\ar[r]^{f_{i_0}}
  &[A_{i_0}]\ar@/_2ex/[rr]_{\varphi_{f_{i_0+1},\ldots,f_{i_2}}}
	\ar@{}[drr]|{\alpha\otimes\beta}\ar@{=}[d]\ar[r]^{f_{i_0+1}}
  &\cdots\ar[r]^{f_{i_2}}
  &[A_{i_2}]\ar@{=}[d]\ar[r]^{f_{i_2+1}}
  &\cdots\ar[r]^{f_n}&A_n 
 \\
  B_0\ar[r]_{g_1}&\cdots\ar[r]_{g_{j_0}}
  &[B_{j_0}] \ar@/^2ex/[rr]^{\varphi_{g_{j_0+1},\ldots,g_{j_2}}}\ar[r]_{g_{j_0+1}}
  &\cdots\ar[r]_{g_{j_2}}
  &[B_{j_2}]\ar[r]_{g_{j_2+1}}
  &\cdots\ar[r]_{g_m}&B_m
 }
 $$
 where $\alpha\otimes\beta$ is the 2-cell in $\calB$ given by the following pasting diagram
 $$
\xymatrix@C=15em@R=8em{
 \ar@/^8ex/[rr]^{\varphi_{f_{i_0+1},\ldots,f_{i_2}}}
 \ar@/_8ex/[rr]_{\varphi_{g_{j_0+1},\ldots,g_{j_2}}}
 \ar@{}@<3ex>[rr]^{\Phi_{f_{i_0+1}\cdots f_{i_1},f_{i_1+1}\cdots f_{i_2}}}
\ar@{}@<-3ex>[rr]_{\Phi_{g_{j_0+1}\cdots g_{j_1},g_{j_1+1}\cdots g_{j_2}}}
\ar@/^4ex/[r]_{\varphi_{f_{i_0+1},\ldots,f_{i_1}}}
\ar@/_4ex/[r]^{\varphi_{g_{j_0+1},\ldots,g_{j_1}}}
\ar@{}[r]|\alpha
&\ar@/^4ex/[r]_{\varphi_{f_{i_1+1},\ldots,f_{i_2}}}
\ar@/_4ex/[r]^{\varphi_{g_{j_1+1},\ldots,g_{j_2}}}
\ar@{}[r]|\beta
&
}
$$

\begin{rmks}\rm
\begin{enumerate}
 \item Note that both the category of horizontal arrows and the category of vertical arrows of 
$\Dbl(\calB)$ are posetal.
\item We will also call $\Dbl(\calB)$ the {\em double category of marked paths in $\calB$}.
\item 
For a 2-category $\calC$ there is a double category $H\calC$ with the arrows of $\calC$ in the horizontal direction and only identity 
arrows in the vertical direction, and the double cells correspond to the 2-cells in $\calC$. 
The double category $\Dbl(\calC)$ is not isomorphic to this double category,  $H\calC$, but it is 2-equivalent to it.
 And the same statement applies to a category $\bfC$: $H\bfC\not\cong\Dbl(\bfC)$, but $H\bfC\simeq_2\Dbl(\bfC)$.
\end{enumerate}
\end{rmks}

\section{Companions, conjoints, equivalences and quasi units}\label{ccqu}

In general double categories, the notions of companion and conjoint have been recognized
as important concepts related to the notion of adjoint.
While adjoint arrows have to be of the same kind, i.e., both horizontal or both vertical, 
the relations of companionship and conjointship are for arrows of different types.
The notions of companion and conjoint were first introduced by Ehresmann in \cite{GP},
but companions in symmetric double categories (where the horizontal and vertical arrow categories are the same)
were studied by Brown and Mosa \cite{BM} under the name `connections'.
Connection pairs were first introduced by Spencer in \cite{spen}.
The existence of companions and conjoints for the vertical arrows in a double category is
related to Shulman's notion of an anchored bicategory \cite{Shul}, 
also called a gregarious double category in \cite{DPP-spans2}.
We will show that for weakly globular double categories horizontal companions are related to
so-called quasi units, i.e., arrows with an invertible 2-cell to a unit arrow, in their associated bicategories.
Then we will introduce a slightly weaker notion, that of a {\em pre-companion}. 
We will show that pre-companions in weakly globular double categories correspond to equivalences in bicategories.
In Section \ref{fracns} we will show that both companions and pre-companions  play an important role in 
the description of the universal properties 
of a weakly globular double
category of fractions. 

We begin this section by repeating the definitions of companions and conjoints 
from \cite{GP} to set our notation, and then we will discuss their relationship to quasi units.
Then we will introduce both a category and a double category of companions. 
And in the last part we will study pre-companions and their properties.

\subsection{Companions and conjoints}

\begin{dfn}\label{friends}\rm
Let $\bbD$ be a double category and consider horizontal morphisms $f\colon A\rightarrow B$ 
and $u\colon B\rightarrow A$ and a vertical morphism 
$\xymatrix@1@C=1.5em{v\colon A\ar[r]|-{\scriptscriptstyle\bullet}&B}$.
We say that $f$ and $v$ are {\em companions} if there exist {\em binding cells}
$$
\xymatrix{
A\ar@{=}[d]\ar@{=}[r] \ar@{}[dr]|\psi & A\ar[d]|{\scriptscriptstyle\bullet}^v \ar@{}[drr]|{\mbox{and}} 
	&& A\ar[d]|{\scriptscriptstyle\bullet}_v \ar[r]^f \ar@{}[dr]|\chi & B\ar@{=}[d]
\\
A\ar[r]_f & B && B\ar@{=}[r] &B,
}
$$
such that
\begin{equation}\label{compcondns}
\xymatrix@R=.5em@C=2em{
	&&&&&&A\ar@{=}[dd]\ar@{=}[r]\ar@{}[ddr]|\psi & A\ar[dd]|{\scriptscriptstyle\bullet}^v 
\\
A \ar@{=}[dd]\ar@{=}[r] \ar@{}[ddr]|\psi
	& A\ar[dd]|{\scriptscriptstyle\bullet}^v \ar[r]^f \ar@{}[ddr]|\chi
	& B\ar@{=}[dd] \ar@{}[ddr]|{\textstyle =} & A\ar[r]^f\ar@{=}[dd] \ar@{}[ddr]|{\mbox{\scriptsize id}_f} 
	& B\ar@{=}[dd]
	&&&& A \ar[dd]|{\scriptscriptstyle\bullet}_v \ar@{}[ddr]|{1_v} \ar@{=}[r] 
	& A \ar[dd]|{\scriptscriptstyle\bullet}^v
\\
	&&&& \ar@{}[rr]|{\mbox{and}}&& A\ar[r]_f \ar[dd]|{\scriptscriptstyle\bullet}_v \ar@{}[ddr]|\chi 
	& B\ar@{=}[dd] \ar@{}[r]|{\textstyle{=}}&
\\
A\ar[r]_f & B\ar@{=}[r] & B & A\ar[r]_f & B&&	&&B\ar@{=}[r] & B
\\
	&&&&&&B\ar@{=}[r] & B&&.
}
\end{equation}
Dually, $u$ and $v$ are {\em conjoints} if there exist {\em binding cells}
$$
\xymatrix{
A\ar[d]|{\scriptscriptstyle\bullet}_v \ar@{}[dr]|\alpha \ar@{=}[r] & A\ar@{=}[d] \ar@{}[drr]|{\mbox{and}} 
	&& B\ar@{=}[d] \ar[r]^u\ar@{}[dr]|\beta & A \ar[d]|{\scriptscriptstyle\bullet}^v 
\\
B\ar[r]_u & A  && B\ar@{=}[r] & B,	
}
$$
such that
\begin{equation}\label{conjcondns}
\xymatrix@R=.5em{
	&&&&&& A\ar[dd]|{\scriptscriptstyle\bullet}_v\ar@{}[ddr]|\alpha \ar@{=}[r] & A\ar@{=}[dd]
\\
B\ar@{=}[dd]\ar@{}[ddr]|\beta \ar[r]^u 
	& A\ar[dd]|{\scriptscriptstyle\bullet}^v \ar@{}[ddr]|\alpha\ar@{=}[r] 
	& A\ar@{=}[dd] \ar@{}[ddr]|{\textstyle =} 
	& B\ar@{=}[dd]\ar[r]^u\ar@{}[ddr]|{\mbox{\scriptsize id}_u} & A\ar@{=}[dd]
	&&&&A\ar@{=}[r]\ar[dd]|{\scriptscriptstyle\bullet}_v \ar@{}[ddr]|{1_v} & A \ar[dd]|{\scriptscriptstyle\bullet}^v
\\
	&&&&\ar@{}[rr]|{\mbox{and}} 
	&&B\ar@{=}[dd]\ar@{}[ddr]|\beta \ar[r]_u & A\ar[dd]|{\scriptscriptstyle\bullet}^v \ar@{}[r]|{\textstyle =}&
\\
B\ar@{=}[r] & B\ar[r]_u & A &B\ar[r]_u &A, &&&&B\ar@{=}[r] & B
\\
&&&&&&B\ar@{=}[r] & B&&.
}
\end{equation}
\end{dfn}

\begin{rmk}\label{compconj}\rm
Since the vertical arrow category in a weakly globular double category is a posetal groupoid, it follows 
from  Proposition 3.5 in \cite{DPP-spans2} that 
the inverse of any companion is a conjoint, and the binding cells for the conjoint pair can be taken to be the vertical inverses of the
binding cells for the companion pair. So a horizontal arrow in a weakly globular double category has a vertical conjoint if and 
only if it has a vertical companion. So it will often be sufficient to focus on the companion pairs.
\end{rmk}

\subsection{Units}
\label{compquasi}

The units in a bicategory are weak units in the sense that we only have that 
there are invertible 2-cells 
\begin{equation}\label{weakunits}
 \lambda_f\colon 1_B\circ f \stackrel{\sim}{\Rightarrow}f
\mbox{ and }\rho_f\colon f\circ 1_A\stackrel{\sim}{\Rightarrow} f
\end{equation}
for any $f\colon A\rightarrow B$
(and these need to satisfy the coherence conditions). However, any arrow $g\colon A\rightarrow A$ with an invertible 2-cell 
$g\stackrel{\sim}{\Rightarrow} 1_A$ would satisfy (\ref{weakunits}).
Since we don't want to consider the coherence conditions that one normally requires of weak units, 
we introduce the notion of quasi units:

\begin{dfn}\rm
An endomorphism $f\colon A\rightarrow A$ in a bicategory is a {\em quasi unit}
if it satisfies the following equivalent conditions:
\begin{enumerate}
\item
$f\cong 1_A$;
\item
$f\circ g\cong g$ for all $g\colon C\rightarrow A$ and $h\circ f \cong h$ for all $h\colon B\rightarrow D$.
\end{enumerate}
\end{dfn} 

We want to characterize the quasi units in the fundamental bicategory $\Bic\bbX$.

\begin{lma}
Every arrow of the form $\mbox{\rm Id}_{A}\colon \bar{A}\rightarrow\bar{A}$
is a quasi unit in $\Bic\bbX$.
\end{lma}

\begin{proof}
There is an invertible 2-cell $\mbox{Id}_{A}\Rightarrow \mbox{Id}_{\theta(\bar{A})}$
given by the vertically invertible double cell
$$
\xymatrix{
A\ar[d]|\bullet_{x} \ar@{}[dr]|{\mbox{\scriptsize id}_{x}} \ar[r]^{\mbox{\scriptsize Id}_A} & A\ar[d]|\bullet^x
\\
\theta(\bar{A}) \ar[r]_{\mbox{\scriptsize Id}_{\theta(\bar{A})}} & \theta(\bar{A})
}
$$
(with inverse $\mbox{id}_{x^{-1}}$).
\end{proof}

We can now characterize the quasi units in $\Bic\bbX$ as those horizontal arrows in $\bbX$ which have a companion.

\begin{prop}\label{quasi-comp}
Let $w\colon A\rightarrow B$ be a horizontal arrow in a weakly globular double category $\bbX$.
Then $w\colon \bar{A}\rightarrow\bar{B}$ is a quasi unit in $\Bic\bbX$ if and only if 
$w\colon  A\rightarrow B$ has a companion $w_* \colon \xymatrix@1{A\ar[r]|\bullet&B}$ in $\bbX$.
\end{prop}

\begin{proof}
If $w$ is a quasi unit, then $\bar{A}=\bar{B}$ and there is an invertible 2-cell 
$\zeta\colon w\Rightarrow 1_{\bar{A}}$ in $\Bic\bbX$, given by a vertically 
invertible double cell in $\bbX$,
$$
\xymatrix{
A\ar[r]^w\ar[d]|\bullet_x\ar@{}[dr]|\zeta &B\ar[d]|\bullet^y
\\
\theta(\bar{A})\ar[r]_{\mbox{\scriptsize Id}_{\theta(\bar{A})}} & \theta(\bar{A})
}
$$ 
We can compose this with the horizontal identity cell on $y^{-1}$ to obtain a double cell
$$
\xymatrix{
A\ar[r]^w\ar[d]|\bullet_x\ar@{}[dr]|\zeta &B\ar[d]|\bullet^y 
	&& A\ar[dd]|\bullet_{y^{-1}\cdot x} \ar[rr]^w \ar@{}[ddrr]|{\mbox{\scriptsize id}_{y^{-1}}\cdot\zeta} 
	&& B\ar[dd]|\bullet^{1_B}
\\
\theta(\bar{A})\ar[r]_{\mbox{\scriptsize Id}_{\theta(\bar{A})}} \ar[d]|\bullet_{y^{-1}} 
    \ar@{}[dr]|{\mbox{\scriptsize id}_{y^{-1}}} 
	& \theta(\bar{A})\ar[d]|\bullet^{y^{-1}} &=& 
\\
B\ar[r]_{\mbox{\scriptsize Id}_B} & B && B\ar[rr]_{\mbox{\scriptsize Id}_B} && B
}
$$ 
Furthermore, precomposing $\zeta^{-1}$ (the vertical inverse of $\zeta$) with 
the horizontal identity cell on $x$ gives
$$
\xymatrix{
A\ar[r]^{\mbox{\scriptsize Id}_A} \ar[d]|\bullet_{x} \ar@{}[dr]|{\mbox{\scriptsize id}_x} &A \ar[d]|\bullet^x 
	&& A\ar[dd]|\bullet_{1_A} \ar[rr]^{\mbox{\scriptsize Id}_A} \ar@{}[ddrr]|{\zeta^{-1}\cdot \mbox{\scriptsize id}_{x}} 
	&& A\ar[dd]|\bullet^{y^{-1}\cdot x}
\\
\theta(\bar{A})\ar[r]_{\mbox{\scriptsize Id}_{\theta(\bar{A})}} \ar[d]|\bullet_{x^{-1}} \ar@{}[dr]|{\zeta^{-1}} 
	& \theta(\bar{A})\ar[d]|\bullet^{y^{-1}} &=& 
\\
A\ar[r]_{w} & B && A\ar[rr]_{w} && B
}
$$
So we claim that we may take $w_*=y^{-1}\cdot x$ with binding cells $\chi_w= \mbox{id}_{y^{-1}}\cdot\zeta$
and $\psi_w=\zeta^{-1}\cdot \mbox{id}_{x}$. (It is straightforward to verify that these double cells satisfy the 
binding cell equations (\ref{compcondns}) for companions.)

Conversely, let $w\colon A\rightarrow B$ be a horizontal arrow with a companion 
$w_*\colon\xymatrix@1{A\ar[r]|\bullet &B}$
and binding cells 
$$
\xymatrix{
A\ar[d]|\bullet_{w_*}\ar@{}[dr]|{\chi_w}\ar[r]^w &B\ar[d]|\bullet^{1_B} &\ar@{}[d]|{\mbox{and}} 
	& A\ar[d]|\bullet_{1_A}\ar@{}[dr]|{\psi_w}\ar[r]^{\mbox{\scriptsize Id}_A} & A\ar[d]|\bullet^{w_*}
\\
B\ar[r]_{\mbox{\scriptsize Id}_B} & B && A\ar[r]_w & B.
}
$$
By the previous lemma it is sufficient to show that there is an invertible 2-cell
$1_A \Rightarrow w$ in $\Bic\bbX$. Such a 2-cell is provided by the binding cell $\psi_w$.
\end{proof}

Now we start with a bicategory $\calB$ and consider its associated double category $\Dbl(\calB)$.
The following proposition describes the relationship between quasi units in $\calB$ and  companions in
$\Dbl(\calB)$.

\begin{prop}\label{comp-quasi}
A horizontal arrow $$\xymatrix{A_0\ar[r]^{f_1}&\cdots\ar[r]^{f_{i_0}}
  &[A_{i_0}]\ar[r]^{f_{i_0+1}}&\cdots\ar[r]^{f_{i_1}}&[A_{i_1}]\ar[r]^{f_{i_1}+1}&\cdots\ar[r]^{f_{n}}&A_n}$$
in $\Dbl(\calB)$ has a companion if and only if $A_{i_0}=A_{i_1}$ and the chosen composition 
$\varphi_{f_{i_0+1}\cdots f_{i_1}}$ is a quasi unit. 
\end{prop}

\begin{proof}
Suppose that $$\xymatrix{A_0\ar[r]^{f_1}&\cdots\ar[r]^{f_{i_0}}
  &[A_{i_0}]\ar[r]^{f_{i_0+1}}&\cdots\ar[r]^{f_{i_1}}&[A_{i_1}]\ar[r]^{f_{i_1}+1}&\cdots\ar[r]^{f_{n}}&A_n}$$
has a companion. Then there are double cells of the form
$$
\xymatrix@R=2.8em{
A_0\ar[r]^{f_1}&\cdots\ar[r]^{f_{i_0}}
  &[A_{i_0}]\ar@{=}[d] \ar@{}[1,4]|(.25)\psi\ar@{=}[rr] && [A_{i_0}]\ar@{=}[d]\ar[r]\ar[r]^{f_{i_0+1}}&\cdots\ar[r]^{f_n}&A_n
\\
A_0\ar[r]_{f_1}&\cdots\ar[r]_{f_{i_0}}
  &[A_{i_0}]\ar@/^3ex/[rr]^{\varphi_{f_{i_0+1}\cdots f_{i_1}}}\ar[r]_{f_{i_0+1}}
  &\cdots\ar[r]_{f_{i_1}}&[A_{i_1}]\ar[r]_{f_{i_1}+1}&\cdots\ar[r]_{f_{n}}&A_n
}
$$
$$
\xymatrix@R=2.8em{
A_0\ar[r]^{f_1}&\cdots\ar[r]^{f_{i_0}}
  &[A_{i_0}] \ar@/_3ex/[rr]_{\varphi_{f_{i_0+1}\cdots f_{i_1}}}\ar@{=}[d]\ar[r]^{f_{i_0+1}}&\cdots\ar[r]^{f_{i_1}}
  &[A_{i_1}]\ar@{=}[d]\ar[r]^{f_{i_1}+1}&\cdots\ar[r]^{f_{n}}&A_n
\\
A_0\ar[r]_{f_1}&\cdots\ar[r]_{f_{i_1}}
  &[A_{i_1}]\ar@{=}[rr]   \ar@{}[-1,4]|(.25)\chi &&[A_{i_1}]\ar[r]_{f_{i_1}+1}&\cdots\ar[r]_{f_{n}}&A_n\rlap{\,.}
}
$$
So $A_{i_0}=A_{i_1}$ and there are 2-cells $\psi\colon \mbox{Id}_{A_{i_0}}\Rightarrow \varphi_{f_{i_0+1}\cdots f_{i_1}}$
and $\chi\colon \varphi_{f_{i_0+1}\cdots f_{i_1}}\Rightarrow \mbox{Id}_{A_{i_1}}$
in $\calB$. The binding cell equations are equivalent to stating that these two 2-cells are inverse to each other, so 
$\varphi_{f_{i_0+1}\cdots f_{i_1}}$ is a quasi unit.

Conversely, if $A_{i_0}=A_{i_1}$ and  $\varphi_{f_{i_0+1}\cdots f_{i_1}}$ is a quasi unit with an invertible 2-cell
$\theta\colon \mbox{Id}_{A_{i_0}}\Rightarrow \varphi_{f_{i_0+1}\cdots f_{i_1}}$, then
$$\xymatrix{A_0\ar[r]^{f_1}&\cdots\ar[r]^{f_{i_0}}
  &[A_{i_0}]\ar[r]^{f_{i_0+1}}&\cdots\ar[r]^{f_{i_1}}&[A_{i_1}]\ar[r]^{f_{i_1}+1}&\cdots\ar[r]^{f_{n}}&A_n}$$
has a companion
$$
\xymatrix@R=1.5em{
A_0\ar[r]^{f_1}&\cdots\ar[r]^{f_{i_0}}
  &[A_{i_0}]\ar@{=}[d]\ar[r]^{f_{i_0+1}}&\cdots\ar[r]^{f_{n}}&A_n
\\
A_0\ar[r]^{f_1}&\cdots\ar[r]^{f_{i_1}}&[A_{i_1}]\ar[r]^{f_{i_1}+1}&\cdots\ar[r]^{f_{n}}&A_n
}
$$
with binding cells
$$
\xymatrix@R=2.8em{
A_0\ar[r]^{f_1}&\cdots\ar[r]^{f_{i_0}}
	&[A_{i_0}]\ar@{=}[d] \ar@{}[1,4]|(.25)\theta\ar@{=}[rr] 
	&& [A_{i_0}]\ar@{=}[d]\ar[r]\ar[r]^{f_{i_0+1}}&\cdots\ar[r]^{f_n}&A_n
\\
A_0\ar[r]_{f_1}&\cdots\ar[r]_{f_{i_0}}
	&[A_{i_0}]\ar@/^3ex/[rr]^{\varphi_{f_{i_0+1}\cdots f_{i_1}}}\ar[r]_{f_{i_0+1}}
	&\cdots\ar[r]_{f_{i_1}}&[A_{i_1}]\ar[r]_{f_{i_1}+1}&\cdots\ar[r]_{f_{n}}
	&A_n
}
$$
$$
\xymatrix@R=2.8em{
A_0\ar[r]^{f_1}&\cdots\ar[r]^{f_{i_0}}
  &[A_{i_0}] \ar@/_3ex/[rr]_{\varphi_{f_{i_0+1}\cdots f_{i_1}}}\ar@{=}[d]\ar[r]^{f_{i_0+1}}&\cdots\ar[r]^{f_{i_1}}
  &[A_{i_1}]\ar@{=}[d]\ar[r]^{f_{i_1}+1}&\cdots\ar[r]^{f_{n}}&A_n
\\
A_0\ar[r]_{f_1}&\cdots\ar[r]_{f_{i_1}}
  &[A_{i_1}]\ar@{=}[rr]   \ar@{}[-1,4]|(.25){\theta^{-1}} &&[A_{i_1}]\ar[r]_{f_{i_1}+1}&\cdots\ar[r]_{f_{n}}&A_n
}
$$
\end{proof}

\subsection{Preservation of companions}
It is clear that strict functors between double categories preserve companions and conjoints.
It is not immediately obvious that the same is true for pseudo-functors, given the fact that horizontal identities
and domains and codomains are not preserved strictly.

However, since there is such a close relationship between companions in weakly globular 
double categories and quasi units in bicategories, and we know that homomorphisms of bicategories 
preserve quasi units, pseudo-functors between weakly globular double categories should preserve companions.
Here we give a direct proof of this result
in terms of companions and binding cells.

\begin{prop}\label{comp-pres}
Let $F\colon \bbX\rightarrow\bbY$ be a pseudo-functor of weakly globular double categories.
If a horizontal arrow $f $ in $\bbX$ has a companion then so does $F(f)$ in $\bbY$.
\end{prop}

\begin{proof}
 Let $\xymatrix@1{A\ar[r]^f&B}$ be a horizontal arrow with vertical 
companion $\xymatrix@1{A\ar[r]|\bullet^{v}&B}$ and binding cells
$$
\xymatrix{
A\ar@{=}[d]\ar@{=}[r]\ar@{}[dr]|\psi & A\ar[d]|\bullet^v 
	& A\ar[d]|\bullet_v \ar[r]^f  \ar@{}[dr]|\chi & B\ar@{=}[d]
\\
A\ar[r]_f & B&B\ar@{=}[r] & B\rlap{\quad.}
}$$
Consider the following pastings of double cells in $\bbY$:
$$
\xymatrix@C=4em{
& d_0F_1f \ar[r]^{\mathrm{Id}_{d_0F_1f}} \ar[d]|\bullet_{(\varphi_1)_A\cdot(d_0F_1\psi)^{-1}}
	    \ar@{}[dr]|{\mathrm{id}} 
& d_0F_1f  \ar[d]|\bullet^{(\varphi_1)_A\cdot(d_0F_1\psi)^{-1}}
\\
\ar@{}[d]|{\textstyle\psi_{F_1f}:=}&FA\ar[r]_{\mathrm{Id}_{F_0A}}\ar[d]|\bullet_{(\varphi_1)_A^{-1}}\ar@{}[dr]|{\sigma_A}
&FA\ar[d]|\bullet^{(\varphi_2)_A^{-1}}
\\
&d_0F_1(\mathrm{Id}_A) \ar[d]|\bullet_{d_0F_1\psi} \ar@{}[dr]|{F_1\psi} \ar[r]_{F_1(\mathrm{Id}_A)}
& d_1F_1(\mathrm{Id}_A) \ar[d]|\bullet^{d_1F_1\psi} 
\\
&d_0F_1f\ar[r]_{F_1f} & d_1F_1f
}
$$
and
$$
\xymatrix@C=4em{
& d_0F_1f \ar[r]^{F_1f} \ar[d]|\bullet_{d_0F_1\chi}
	    \ar@{}[dr]|{F_1\chi} 
& d_0F_1f  \ar[d]|\bullet^{d_1F_1\chi}
\\
\ar@{}[d]|{\textstyle\chi_{F_1f}:=}
&d_0F_1(\mathrm{Id}_B)\ar[r]_{F_1(\mathrm{Id}_{B})}\ar[d]|\bullet_{(\varphi_1)_B}\ar@{}[dr]|{\sigma_B^{-1}}
&d_1F_1(\mathrm{Id}_B)\ar[d]|\bullet^{(\varphi_2)_B}
\\
&FB\ar[d]|\bullet_{(d_1F_1\chi)^{-1}\cdot(\varphi_2)^{-1}_B} \ar@{}[dr]|{\mathrm{id}} 
	    \ar[r]_{\mathrm{Id}_{F_0B}}
&FB\ar[d]|\bullet^{(d_1F_1\chi)^{-1}\cdot(\varphi_2)_B^{-1}} 
\\
&d_1F_1f\ar[r]_{\mathrm{Id}_{d_1F_1f}} & d_1F_1f\rlap{\quad.}
}
$$
Note that $d_1F_1\psi\cdot(\varphi_2)^{-1}_A\cdot(\varphi_1)_A\cdot(d_0F_1\psi)^{-1}=
(d_1F_1\chi)^{-1}\cdot(\varphi_2)^{-1}_B\cdot(\varphi_1)_B\cdot d_0F_1\chi$
since the vertical arrow category is posetal, and we claim that this arrow is the vertical 
companion of $F_1f$ with binding cells $\psi_{F_1f}$ and $\chi_{F_1f}$.
So we need to check that the equations in (\ref{compcondns}) are satisfied.
It is relatively easy to check that the second equation, involving the vertical composition of these cells, 
is satisfied. For the first equation we need to refer to the coherence conditions in 
(\ref{comp-nat}), (\ref{ass-id1}), and (\ref{ass-id2}).
The horizontal composition considered in the binding cell equation is
$$
\xymatrix@C=4em{
d_0F_1f \ar[r]^{\mathrm{Id}_{d_0F_1f}} \ar[d]|\bullet_{(\varphi_1)_A\cdot(d_0F_1\psi)^{-1}}
	    \ar@{}[dr]|{\mathrm{id}} 
& d_0F_1f  \ar[d]|\bullet^(.35){(\varphi_1)_A\cdot(d_0F_1\psi)^{-1}}\ar@{=}[r] \ar@{}[dddr]|{=} 
& d_0F_1f \ar[r]^{F_1f} \ar[d]|\bullet_(.6){d_0F_1\chi}
	    \ar@{}[dr]|{F_1\chi} 
& d_0F_1f  \ar[d]|\bullet^{d_1F_1\chi}
\\
FA\ar[r]|{\mathrm{Id}_{F_0A}}\ar[d]|\bullet_{(\varphi_1)_A^{-1}}\ar@{}[dr]|{\sigma_A}
&FA\ar[d]|\bullet^{(\varphi_2)_A^{-1}} 
& d_0F_1(\mathrm{Id}_B)\ar[r]|{F_1(\mathrm{Id}_{B})}\ar[d]|\bullet_{(\varphi_1)_B}\ar@{}[dr]|{\sigma_B^{-1}}
&d_1F_1(\mathrm{Id}_B)\ar[d]|\bullet^{(\varphi_2)_B}
\\
d_0F_1(\mathrm{Id}_A) \ar[d]|\bullet_{d_0F_1\psi} \ar@{}[dr]|{F_1\psi} \ar[r]|{F_1(\mathrm{Id}_A)}
& d_1F_1(\mathrm{Id}_A) \ar[d]|\bullet^(.35){d_1F_1\psi} 
& FB\ar[d]|\bullet_(.6){(d_1F_1\chi)^{-1}\cdot(\varphi_2)^{-1}_B} \ar@{}[dr]|{\mathrm{id}} 
	    \ar[r]|{\mathrm{Id}_{F_0B}}
&FB\ar[d]|\bullet^{(d_1F_1\chi)^{-1}\cdot(\varphi_2)_B^{-1}} 
\\
d_0F_1f\ar[r]_{F_1f} & d_1F_1f\ar@{=}[r]&d_1F_1f\ar[r]_{\mathrm{Id}_{d_1F_1f}} & d_1F_1f\rlap{\quad.}
}$$
By (\ref{ass-id1}), and (\ref{ass-id2})
this is equal to
$$
\xymatrix@C=6em@R=1.8em{
 \ar[rrr]^{F_1(f\circ \mbox{\scriptsize Id}_A)} \ar@{}[drrr]|{\mu_{f,\mathrm{Id}_A}} \ar[d]|\bullet 
	&&& \ar[d]|\bullet 
\\
\ar[d]|\bullet\ar[r]^{\pi_2(F_2(f,\mathrm{Id}))}\ar@{}[dr]|{(\theta_2)_{f,\mathrm{Id}_A}}
	&\ar[d]|\bullet\ar@{=}[r] \ar@{}[dddr]|{=}
	& \ar[r]^{\pi_1(F_2(f,\mathrm{Id}_A))} \ar@{}[dddr]|{(\theta_1)_{f,\mathrm{Id}_A}}\ar[ddd]|\bullet 
	&\ar[ddd]|\bullet 
\\
\ar[r]|{F_1(\mbox{\scriptsize Id}_A)} \ar[d]|\bullet \ar@{}[dr]|{\sigma_A^{-1}}&\ar[d]|\bullet  &&
\\
\ar[d]|\bullet \ar[r]|{\mbox{\scriptsize Id}_{F_0A}} \ar@{}[dr]|{\mbox{\scriptsize id}} &\ar[d]|\bullet &&
\\
\ar[r]|{\mathrm{Id}_{d_0F_1f}} \ar[d]|\bullet
	    \ar@{}[dr]|{\mathrm{id}} 
&  \ar[d]|\bullet\ar@{=}[r] \ar@{}[dddr]|{=} 
&  \ar[r]|{F_1f} \ar[d]|\bullet
	    \ar@{}[dr]|{F_1\chi} 
&  \ar[d]|\bullet
\\
\ar[r]|{\mathrm{Id}_{F_0A}}\ar[d]|\bullet\ar@{}[dr]|{\sigma_A}
&\ar[d]|\bullet
& \ar[r]|{F_1(\mathrm{Id}_{B})}\ar[d]|\bullet\ar@{}[dr]|{\sigma_B}
&\ar[d]|\bullet
\\
 \ar[d]|\bullet \ar@{}[dr]|{F_1\psi} \ar[r]|{F_1(\mathrm{Id}_A)}
&  \ar[d]|\bullet
& \ar[d]|\bullet \ar@{}[dr]|{\mathrm{id}} 
	    \ar[r]|{\mathrm{Id}_{F_0B}}
&\ar[d]|\bullet
\\
\ar[r]|{F_1f} \ar[ddd]|\bullet\ar@{}[dddr]|{(\theta_2)^{-1}_{\mathrm{Id}_B,f}}
& \ar[ddd]|\bullet \ar@{=}[r]\ar@{}[dddr]|= 
&\ar[r]|{\mathrm{Id}_{d_1F_1f}}\ar[d]|\bullet \ar@{}[dr]|{\mathrm{id}} 
& \ar[d]|\bullet
\\
&&\ar[r]|{\mathrm{Id}_{F_0B}}\ar[d]|\bullet \ar@{}[dr]|{\sigma_B^{-1}} &\ar[d]|\bullet
\\
&&\ar[r]|{F_1(\mathrm{Id}_{B})}\ar[d]|\bullet\ar@{}[dr]|{(\theta_1)^{-1}_{\mathrm{Id}_B,f}}
&\ar[d]|\bullet 
\\
\ar[d]|\bullet \ar[r]_{\pi_2F_2(\mathrm{Id}_B,f)} \ar@{}[drrr]|{\mu^{-1}_{\mathrm{Id}_B,f}}
&\ar@{=}[r]& \ar[r]_{\pi_1F_2(\mathrm{Id}_B,f)} & \ar[d]|\bullet
\\
\ar[rrr]_{F_1(f)} &&&
\rlap{\quad.}
}$$
Cancellation of vertical inverses gives that this is equal to
$$
 \xymatrix@C=4em{
\ar[rrr]^{F_1(f\circ \mathrm{Id}_A)}\ar[d]|\bullet \ar@{}[drrr]|{\mu_{f,\mathrm{Id}_A}} 
	&&& \ar[d]|\bullet
\\
\ar[r]^{\pi_2F_2(f,\mathrm{Id}_A)} \ar[d]|\bullet \ar@{}[dr]|{(\theta_2)_{f,\mathrm{Id}_A}} 
	&\ar[d]|\bullet \ar@{=}[r] \ar@{}[dddr]|=
	&\ar[r]^{\pi_1F_2(f,\mathrm{Id}_A)} \ar[d]|\bullet_\sim \ar@{}[dr]|{(\theta_1)_{f,\mathrm{Id}_A}} 
	&\ar[d]|\bullet
\\
\ar[r]|{F_1\mathrm{Id}_A}\ar[d]|\bullet \ar@{}[dr]|{F_1\psi} & \ar[d]|\bullet 
	& \ar[d]|\bullet\ar@{}[dr]|{F_1\chi}\ar[r]|{F_1f} & \ar[d]|\bullet \ar@{}[drr]|{\textstyle =}
	&& \ar[r]^{F_1(f)}\ar@{}[dr]|{F_1(\chi\circ\psi)} \ar[d]|\bullet &\ar[d]|\bullet
\\
\ar[d]|\bullet \ar[r]|{F_1f} \ar@{}[dr]|{(\theta_2)_{\mathrm{Id}_B,f}^{-1}} &\ar[d]|\bullet
	& \ar[d]|\bullet \ar[r]|{F_1\mathrm{Id}_B} \ar@{}[dr]|{(\theta_1)_{\mathrm{Id}_B,f}^{-1}} 
	&\ar[d]|\bullet
	&& \ar[r]_{F_1(\mathrm{Id}_B\circ f)} &
\\
\ar[r]_{\pi_2F_2(\mathrm{Id}_B,f)} \ar[d]|\bullet \ar@{}[drrr]|{\mu_{\mathrm{Id}_B,f}^{-1}}
	& \ar@{=}[r] & \ar[r]_{\pi_1F_2(\mathrm{Id}_B,f)} &\ar[d]|\bullet
\\
\ar[rrr]_{F_1(\mathrm{Id}_B\circ f)} &&&
}
$$
where the last equality is by (\ref{comp-nat}).
This gives us the required identity double cell.
\end{proof}

\subsection{Companion categories}
In the study of companions and conjoints and functors that send arrows to companions or conjoints
the following construction on (arbitrary) double categories is very useful.

\begin{dfn}\rm
Let $\bbD$ be an arbitrary double category. Then $\bbComp(\bbD)$, the {\em double category 
of companions in $\bbD$}, is defined as follows. It has the 
same objects and horizontal arrows as $\bbD$, but its vertical arrows are companion pairs 
(with their binding cells)
in $\bbD$. So a vertical arrow $\theta\colon\xymatrix@1{A\ar[r]|\bullet & B}$ in $\bbComp(\bbD)$ is given 
by a quadruple $\theta=(h_\theta,v_\theta,\psi_\theta,\chi_\theta)$, where $v_\theta$ and $h_\theta$ 
are companions with binding cells
$$
\xymatrix{
A\ar[r]^{\mathrm{Id}_A} \ar@{}[dr]|{\psi_\theta} \ar[d]|\bullet_{1_A} & A\ar[d]|\bullet^{v_\theta}\ar@{}[drr]|{\mbox{and}}
    &&  A\ar[d]|\bullet_{v_\theta} \ar[r]^{h_\theta} \ar@{}[dr]|{\chi_\theta}& B\ar[d]|\bullet^{1_A}
\\
A\ar[r]_{h_\theta} & B && B\ar[r]_{\mathrm{Id}_A}  &B.
}$$
The vertical identity arrow $1_A$ is given by $(\mathrm{Id}_A,1_A, \iota_A,\iota_A)$.
Vertical composition is defined by $\theta'\cdot\theta=(h_{\theta'}h_{\theta},v_{\theta'}\cdot v_\theta, 
(\psi_{\theta'}\circ 1_{h_\theta})\cdot(\mbox{id}_{v_{\theta'}}\circ\psi_\theta),
(\chi_{\theta'}\circ\mbox{id}_{v_\theta})\cdot(1_{h_{\theta'}}\circ\chi_\theta))$.
A double cell 
$$
\xymatrix{
A\ar[r]^f\ar[d]|\bullet_\theta\ar@{}[dr]|\Theta & A'\ar[d]|\bullet^{\theta'}
\\
B\ar[r]_g&B'
}
$$
in $\bbComp(\bbD)$ consists of a double cell
$$
\xymatrix@R=2em{
A\ar[r]^f\ar[d]|\bullet_{v_\theta}\ar@{}[dr]|\Theta & A'\ar[d]|\bullet^{v_{\theta'}}
\\
B\ar[r]_g&B'
}
$$ 
in $\bbD$ which has the following properties:
\begin{enumerate}
 \item the square of horizontal arrows 
$$
\xymatrix@R=2em{
A\ar[d]_{h_\theta} \ar[r]^f& A'\ar[d]^{h_{\theta'}}
\\
B\ar[r]_g & B'
}$$
commutes in $\bbD$;
\item 
$$
\xymatrix@C=3.5em@R=2em{
A\ar[r]^f\ar[d]|\bullet_{v_\theta}\ar@{}[dr]|\Theta 
    & A'\ar[d]|\bullet^{v_{\theta'}}\ar@{}[dr]|{\chi_{\theta'}} \ar[r]^{h_{\theta'}} & B' \ar@{=}[d]
    \ar@{}[dr]|{\textstyle =} & A\ar[d]|\bullet_{v_{\theta}}\ar[r]^{h_\theta}\ar@{}[dr]|{\chi_\theta}
    & B\ar@{=}[d] \ar[r]^g\ar@{}[dr]|{1_g} & B'\ar@{=}[d]
\\
B\ar[r]_g&B' \ar@{=}[r] & B' & B\ar@{=}[r] &B\ar[r]_g & B'
}
$$ 
\item
$$
\xymatrix@C=3.5em@R=2em{
A\ar@{=}[r]\ar@{=}[d] \ar@{}[dr]|{\psi_\theta}
    & A\ar[d]|\bullet^{v_{\theta}}\ar@{}[dr]|{\Theta} \ar[r]^{f} & A' \ar[d]|\bullet^{v_{\theta'}}
    \ar@{}[dr]|{\textstyle =} & A\ar@{=}[d] \ar[r]^{f}\ar@{}[dr]|{1_f}
    & A'\ar@{=}[d] \ar@{=}[r]\ar@{}[dr]|{\psi_\theta} & B'\ar[d]|\bullet^{v_{\theta'}}
\\
A\ar[r]_{h_{\theta}}&B \ar[r]_g & B' & A\ar[r]_f &A'\ar[r]_{h_{\theta'}} & B'
}
$$ 
\end{enumerate}
\end{dfn}

\begin{rmks}\rm
\begin{enumerate}
 \item We write $\Comp(\bbD)$ for the category of vertical arrows, $v\bbComp(\bbD)$, i.e., 
the category of companion pairs (with chosen binding cells).
\item If we have a functorial choice of companions and binding cells  and we only use the 
horizontal arrows that are taken in these choices in the construction of $\bbComp(\bbD)$,
the result is a double category with a thin structure as defined in \cite{BM}.
\end{enumerate}
\end{rmks}

\subsection{Pre-companions}\label{pre-comps}
We saw before that companions in weakly globular double categories correspond to quasi units in 
bicategories. 
We now introduce a class of horizontal arrows in weakly globular double categories that will correspond to  
equivalences in bicategories.

\begin{dfn}\label{precomp}\rm
Let $\bbX$ be a weakly globular double category.
\begin{enumerate}
 \item A horizontal arrow $\xymatrix@1{A\ar[r]^f&B}$  in $\bbX$ is a {\em left pre-companion}
if there are horizontal arrows $\xymatrix@1{A'\ar[r]^{f'} &B'}$ and $\xymatrix@1{B'\ar[r]^{r_f}&C}$
with a vertically invertible double cell
$$
\xymatrix{
A\ar[d]|\bullet\ar@{}[dr]|\varphi \ar[r]^f&B\ar[d]|\bullet\\
A'\ar[r]_{f'}&B'}
$$
such that $r_f\circ f'$ is a companion in $\bbX$.
\item
Dually, a horizontal arrow $\xymatrix@1{A\ar[r]^f&B}$  in $\bbX$ is a {\em right pre-companion}
if there are horizontal arrows $\xymatrix@1{A''\ar[r]^{f''} &B''}$ and $\xymatrix@1{D\ar[r]^{l_f}&A''}$
with a vertically invertible double cell
$$
\xymatrix{
A\ar[d]|\bullet\ar@{}[dr]|{\varphi'} \ar[r]^f&B\ar[d]|\bullet\\
A''\ar[r]_{f''}&B''}
$$
such that $f''\circ l_f$ is a companion in $\bbX$.  
\item
A horizontal arrow $\xymatrix@1{A\ar[r]^f&B}$  in $\bbX$ is a {\em pre-companion}
if it is both a left and a right pre-companion.
\end{enumerate}
\end{dfn}

We first prove some basic properties of pre-companions.

\begin{lma}\label{lr}
 If $f$ is a pre-companion with arrows $l_f$ and $r_f$ as in Definition \ref{precomp}, then 
there is a vertically invertible double cell
of the form
$$
\xymatrix@R=3em@C=3em{
\ar[d]|\bullet\ar@{}[dr]|{\nu_f}\ar[r]^{r_f} & \ar[d]|\bullet
\\
\ar[r]_{l_f} &\rlap{\quad.}
}
$$
\end{lma}

\begin{proof}
 Let $r_f\circ f'$ have vertical companion $v^r_f$ with binding cells
$$
\xymatrix{
\ar@{=}[d]\ar@{=}[rr]\ar@{}[drr]|{\psi_f^r} && \ar[d]^{v^r_f}|\bullet \ar@{}[drr]|{\mbox{and}} 
	&& \ar[r]^{f'}\ar[d]|\bullet_{v^r_f} \ar@{}[drr]|{\chi_f^r} & \ar[r]^{r_f} & \ar@{=}[d]
\\
\ar[r]_{f'}&\ar[r]_{r_f} & && \ar@{=}[rr] &&\rlap{\quad,}
}
$$
and let $f''\circ l_f$ have vertical companion $v^l_f$ with binding cells
$$
\xymatrix{
\ar@{=}[d]\ar@{=}[rr]\ar@{}[drr]|{\psi_f^l} && \ar[d]^{v^l_f}|\bullet \ar@{}[drr]|{\mbox{and}} 
	&& \ar[r]^{l_f}\ar[d]|\bullet_{v^l_f} \ar@{}[drr]|{\chi_f^l} & \ar[r]^{f''} & \ar@{=}[d]
\\
\ar[r]_{l_f}&\ar[r]_{f''} & && \ar@{=}[rr] &&\rlap{\quad.}
}
$$
Further, let $x=(v_f^l)^{-1}\cdot d_1\varphi'\cdot(d_1\varphi)^{-1}$
and $y=(d_0\varphi')\cdot(d_0\varphi)^{-1}\cdot (v^r_f)^{-1}$.
Then $\nu_f$ can be obtained 
as the following pasting of double cells:
$$\xymatrix@C=4em@R=2em{
\ar[d]|\bullet_x\ar@{=}[rr]\ar@{}[drr]|{\mathrm{id}} && \ar[d]|\bullet_x\ar[r]^{r_f}\ar@{}[4,1]|{1_{r_f}} & \ar@{=}[4,0]
\\
\ar@{=}[d]\ar@{=}[rr]\ar@{}[drr]|{\psi_f^l} && \ar[d]|\bullet^{v_f^l}
\\
\ar[r]|{l_f}\ar@{=}[4,0]\ar@{}[4,1]|{1_{l_f}} & \ar[d]|\bullet_{d_0\varphi'^{-1}}\ar[r]|{f''}\ar@{}[dr]|{\varphi'^{-1}} 
	& \ar[d]|\bullet^{d_1\varphi'^{-1}}
\\
& \ar[d]|\bullet_{d_0\varphi}\ar[r]|{f}\ar@{}[dr]|\varphi & \ar[d]|\bullet^{d_1\varphi}
\\
&\ar[d]_{v^r_f}\ar@{}[drr]|{\chi^r_f}\ar[r]|{f'} & \ar[r]|{r_f} &\ar@{=}[d]
\\
&\ar[d]|\bullet_y\ar@{=}[rr]\ar@{}[drr]|{\mathrm{id}} &&\ar[d]|\bullet^y
\\
\ar[r]_{l_f} & \ar@{=}[rr]&&\rlap{\quad .}
}$$
\end{proof}

It is well-known that companions, just like adjoints, are unique up to special invertible double cells.
A similar result applies to pre-companions and the proof is a 2-dimensional version of the proof
that the pseudo-inverse of an equivalence in a bicategory is unique up to invertible 2-cell.
 
\begin{lma}\label{pre-comp-uniq}
 For any pre-companion $f$ in a weakly globular double category, the
pre-companion structure given in Definition \ref{precomp} is unique
up to vertically invertible double cells.
\end{lma}

\begin{proof}
 Suppose that $\varphi_i$, $r_{f,i}$, $v_{f,i}^r$, $f'_i$, $\psi_{f,i}^r$, and $\chi_{f,i}^r$ give two right pre-companion structures for $i=1,2$,
and $\varphi'_i$,
$l_{f,i}, f''_i, v_{f,i}^l$, $\psi_{f,i}^l$ and $\chi_{f,i}^l$ give two left pre-companion structures for $i=1,2$.
Then there are vertically invertible double cells as follows:
$$
\xymatrix@C=4em{
\ar[r]^{f'_1}\ar[d]|\bullet\ar@{}[dr]|{\varphi_1^{-1}} & \ar[d]|\bullet &\ar[d]|\bullet\ar[r]^{f''_1}\ar@{}[dr]|{(\varphi'_1)^{-1}}& \ar[d]|\bullet
\\
\ar[r]|f \ar[d]|\bullet \ar@{}[dr]|{\varphi_2} & \ar[d]|\bullet & \ar[r]|f  \ar[d]|\bullet \ar@{}[dr]|{\varphi'_2} & \ar[d]|\bullet
\\
\ar[r]_{f'_2} & & \ar[r]_{f''_2} &
}$$
According to Lemma \ref{lr} there are vertically invertible double cells
$$
\xymatrix{
\ar[d]|\bullet\ar[r]^{r_{f,1i}}\ar@{}[dr]|{\nu_{f,ij}} &
\ar[d]|\bullet \ar@{}[drrr]|{\textstyle \mbox{for }i,j\in\{1,2\}.}
\\
\ar[r]_{l_{f,j}}&&&&
}$$
So we obtain vertically invertible cells
$$
\xymatrix@C=5em{
\ar[r]^{r_{f,1}} \ar[d]|\bullet \ar@{}[dr]|{\nu_{f,21}^{-1}\cdot\nu_{f,11}} 
	&\ar[d]|\bullet\ar@{}[dr]|{\mbox{and}} & \ar[r]^{l_{f,1}}\ar[d]|\bullet \ar@{}[dr]|{\nu_{f,22}\cdot\nu_{f,21}^{-1}}& \ar[d]|\bullet
\\
\ar[r]_{r_{f,2}} && \ar[r]_{l_{f,2}}&
}$$
\end{proof}

The next two propositions establish the relationship between 
pre-companions in weakly globular double categories and equivalences in bicategories.

\begin{prop}
 Let $\calB$ be a bicategory.
A horizontal arrow
\begin{equation}\label{arrow}
\xymatrix{A_0\ar[r]^{f_1}&\cdots\ar[r]^{f_{i_0}}
  &[A_{i_0}]\ar[r]^{f_{i_0+1}}&\cdots\ar[r]^{f_{i_1}}&[A_{i_1}]\ar[r]^{f_{i_1}+1}&\cdots\ar[r]^{f_{n}}&A_n}
\end{equation}
in $\Dbl(\calB)$ is a pre-companion if and only if the chosen composition $\varphi_{f_{i_0+1}\cdots f_{i_1}}$
is an equivalence in $\calB$.
\end{prop}

\begin{proof}
Suppose that the chosen composition $\varphi_{f_{i_0+1}\cdots f_{i_1}}$ 
is an equivalence in $\calB$. 
Denote the arrow (\ref{arrow}) in $\Dbl(\calB)$ by $f$.
Let $g\colon A_{i_1}\rightarrow A_{i_0}$ be a pseudo-inverse of $\varphi_{f_{i_0+1}\cdots f_{i_1}}$ in $\calB$.
Then we may take $f'$ in Definition \ref{precomp} to be the arrow
$$
\xymatrix@C=5em{
[A_{i_0}]\ar[r]^{\varphi_{f_{i_0+1}\cdots f_{i_1}}}&[A_{i_1}]\ar[r]^{g}&A_{i_0}
}$$
with vertically invertible cell 
$$
\xymatrix{
A_0\ar[r]^{f_1} & \cdots \ar[r]^{f_{i_0}}&[A_{i_0}]\ar@{=}[d] \ar@/_2ex/[rr]_{\varphi_{f_{i_0+1}\cdots f_{i_1}}}\ar[r]^{f_{i_0+1}} 
			    & \cdots \ar[r]^{f_{i_1}}\ar@{}[d]|(.8){=}
			    &[A_{i_1}] \ar[r]^{f_{i_1+1}}\ar@{=}[d] &\cdots\ar[r]^{f_{i_n}}&A_n
\\
&& [A_{i_0}]\ar[rr]_{\varphi_{f_{i_0+1}\cdots f_{i_1}}} && [A_{i_1}]\ar[rr]_{g} && A_{i_0}
}$$
and we may take $r_f$ in Definition \ref{precomp} to be the arrow
$$
\xymatrix@C=4em{
A_{i_0}\ar[r]^{\varphi_{f_{i_0+1}\cdots f_{i_1}}}&[A_{i_1}]\ar[r]^{g}&[A_{i_0}]\rlap{\,.}
}$$
The horizontal composition $r_f\circ f'$ is given by
$$
\xymatrix@C=4em{
[A_{i_0}]\ar[r]^{\varphi_{f_{i_0+1}\cdots f_{i_1}}}&A_{i_1}\ar[r]^{g}&[A_{i_0}]\rlap{\,,}
}$$
and this is a companion by Proposition \ref{comp-quasi}, since $g\circ \varphi_{f_{i_0+1}\cdots f_{i_1}}$ 
is a quasi unit in $\calB$.
So (\ref{arrow}) is a left pre-companion.
The proof that it is a right pre-companion goes similarly.

Now suppose that (\ref{arrow}) is a precompanion in $\Dbl(\calB)$.
So there are diagrams with vertically invertible double cells of the form
$$
  \xymatrix@R=4em{
A_0\ar[r]^{f_1} & \cdots \ar[r]^{f_{i_0}}
      &[A_{i_0}]\ar@/_2ex/[rr]_{\varphi_{f_{i_0+1}\cdots f_{i_1}}}\ar@{=}[d]\ar[r]^{f_{i_0+1}} 
      & \cdots \ar[r]^{f_{i_1}}\ar@{}[d]|{\varphi}
      &[A_{i_1}]\ar@{=}[d]\ar[r]^{f_{i_1+1}}&\cdots\ar[r]^{f_{n}}&A_n
\\
B_0\ar[r]_{g_1} &\cdots\ar[r]_{g_{j_0}} &[A_{i_0}]\ar@/^2ex/[rr]^{\varphi_{g_{j_0+1}\cdots g_{j_1}}}
	    \ar[r]_{g_{j_0+1}} & \cdots\ar[r]_{g_{j_1}} & [A_{i_1}]\ar[r]_{g_{j_1+1}} 
      & \cdots\ar[r]_{g_{j_2}} & C \ar[r]_{g_{j_2+1}}&\cdots\ar[r]_{g_{m}} & B_{m}
}$$
and
$$
  \xymatrix@R=4em{
&&A_0\ar[r]^{f_1} & \cdots \ar[r]^{f_{i_0}}
      &[A_{i_0}]\ar@/_2ex/[rr]_{\varphi_{f_{i_0+1}\cdots f_{i_1}}}\ar@{=}[d]\ar[r]^{f_{i_0+1}} 
      & \cdots \ar[r]^{f_{i_1}}\ar@{}[d]|{\varphi'}
      &[A_{i_1}]\ar@{=}[d]\ar[r]^{f_{i_1+1}}&\cdots\ar[r]^{f_{n}}&A_n
\\
B'_0\ar[r]_{h_1} &\cdots\ar[r]_{h_{k_0}} & D\ar[r]_{h_{k_0+1}} & \cdots \ar[r]_{h_{k_1}} 
      & [A_{i_0}]\ar@/^2ex/[rr]^{\varphi_{h_{k_1+1}\cdots h_{k_2}}} \ar[r]_{h_{k_1+1}}
      &\cdots \ar[r]_{h_{k_2}} & [A_{i_1}]\ar[r]_{h_{k_2+1}}&\cdots\ar[r]_{h_p} &B'_p
}
$$
such that the arrows 
$$
\xymatrix{
B_0\ar[r]_{g_1} &\cdots\ar[r]_{g_{j_0}} &[A_{i_0}]
	    \ar[r]_{g_{j_0+1}} & \cdots\ar[r]_{g_{j_1}} & A_{i_1}\ar[r]_{g_{j_1+1}} 
      & \cdots\ar[r]_{g_{j_2}} & [C] \ar[r]_{g_{j_2+1}}&\cdots\ar[r]_{g_{m}} & B_{m}
}$$
and
$$
\xymatrix{
B'_0\ar[r]_{h_1} &\cdots\ar[r]_{h_{k_0}} & [D]\ar[r]_{h_{k_0+1}} & \cdots \ar[r]_{h_{k_1}} 
      & A_{i_0} \ar[r]_{h_{k_1+1}}
      &\cdots \ar[r]_{h_{k_2}} & [A_{i_1}]\ar[r]_{h_{k_2+1}}&\cdots\ar[r]_{h_p} &B'_p
}$$
have companions in $\Dbl(\calB)$.
By Proposition \ref{comp-quasi} this implies that $A_{i_0}=C$, $D=A_{i_1}$ and
both chosen composites $\varphi_{g_{j_0+1}\cdots g_{j_2}}$ and $\varphi_{h_{k_0+1}\cdots h_{k_2}}$
are quasi units in $\calB$. So the chosen composite $\varphi_{g_{j_1+1}\cdots g_{j_2}}$
is a pseudo-inverse for $\varphi_{f_{i_0+1}\cdots f_{i_1}}$, making this arrow an equivalence in $\calB$. 
\end{proof}

\begin{prop}
 Let $\bbX$ be a weakly globular double category with a horizontal arrow $f\colon A\rightarrow B$.
Then the arrow $f\colon \bar{A}\rightarrow\bar{B}$
in $\Bic(\bbX)$ is an equivalence if and only if  $f$ is a pre-companion in $\bbX$.
\end{prop}

\begin{proof}
 Suppose that $f\colon \bar{A}\rightarrow\bar{B}$ has pseudo inverse $g\colon \bar{B}\rightarrow\bar{A}$.
The two  compositions of these arrows in $\Bic(\bbX)$ are given by the horizontal compositions
$\tilde{g}\circ \tilde{f}$  and $\hat{f}\circ\hat{g}$ as in the following diagrams of double cells
\begin{equation}\label{firstcomp}
\xymatrix{
\tilde{A} \ar[dd]|\bullet\ar[r]^{\tilde{f}}\ar@{}[ddr]|\cong 
	&\tilde{B}\ar[d]|\bullet\ar[r]^{\tilde{g}}\ar@{}[dr]|\cong 
	& \tilde{A}'\ar[d]|\bullet 
\\
&B'\ar[d]|\bullet\ar[r]_g&A'
\\
A\ar[r]_f & B
}\end{equation}
and
\begin{equation}\label{secondcomp}
\xymatrix{
\hat{B}' \ar[dd]|\bullet\ar[r]^{\hat{g}}\ar@{}[ddr]|\cong 
	&\hat{A}\ar[d]|\bullet\ar[r]^{\hat{f}}\ar@{}[dr]|\cong 
	& \hat{B}\ar[d]|\bullet 
\\
&A\ar[d]|\bullet\ar[r]_f&B
\\
B'\ar[r]_g & A'\rlap{\quad.}
}\end{equation}
By Proposition \ref{quasi-comp}, the compositions $\tilde{g}\circ\tilde{f}$ and $\hat{f}\circ\hat{g}$
are companions in $\bbX$ and the diagrams (\ref{firstcomp}) and (\ref{secondcomp}) show that $f$ is a pre-companion
with $f'=\tilde{f}$, $r_f=\tilde{g}$, $f''=\hat{f}$ and $l_f=\hat{g}$.

Now suppose that $f$ is a pre-companion in $\bbX$ with additional arrows and cells as in Definition \ref{precomp}:
$$
\xymatrix{
\ar[r]^f\ar[d]|\bullet\ar@{}[dr]|\varphi &\ar[d]|\bullet &&&\ar[r]^f\ar[d]|\bullet\ar@{}[dr]|{\varphi'}&\ar[d]|\bullet
\\
\ar[r]_{f'} &\ar[r]_{r_f} & &\ar[r]_{l_f} & \ar[r]_{f''} &\rlap{\quad.}
}$$
By Proposition \ref{quasi-comp}, $r_f\circ f'$ and $f''\circ l_f$ are quasi units in $\Bic(\bbX)$, say with 
invertible 2-cells $\alpha\colon r_f\circ f'\Rightarrow\mbox{Id}_{\bar{A}}$ and
$\beta\colon  f''\circ l_f\Rightarrow\mbox{Id}_{\bar{B}}$.
So we have the composites,
$r_f f\stackrel{r_f\varphi}{\Rightarrow} r_f\circ f'\stackrel{\alpha}{\Rightarrow}\mbox{Id}_{\bar{A}}$
and $f\,l_f\stackrel{\varphi' \,l_f}{\Rightarrow} f''\circ l_f\stackrel{\beta}{\Rightarrow}\mbox{Id}_{\bar{B}}$.
All these 2-cells are invertible and this implies that $f\colon \bar{A}\rightarrow\bar{B}$ in $\Bic(\bbX)$ is a 
an equivalence.
\end{proof}

\begin{prop}
 A pseudo-functor between weakly globular double categories preserves pre-companions.
\end{prop}

\begin{proof}
 Let $F\colon \bbX\rightarrow\bbY$ be a pseudo-functor between two weakly globular double 
categories $\bbX$ and $\bbY$, and let $\xymatrix@1{A\ar[r]^f&B}$ be a left pre-companion in 
$\bbX$, with 
$$
\xymatrix{
A\ar[r]^f\ar[d]|\bullet\ar@{}[dr]|\varphi &B\ar[d]|\bullet
\\
A'\ar[r]_{f'}&B'\ar[r]_g & C
}$$ 
as in Definition \ref{precomp}, companion $v_f$, and binding cells
$$
\xymatrix{
A'\ar@{=}[d]\ar@{=}[rr] \ar@{}[drr]|{\psi_f} && A'\ar[d]|\bullet^{v_f}\ar@{}[drr]|{\mbox{and}}
	&& A'\ar[r]^{f'}\ar[d]_{v_f} \ar@{}[drr]|{\chi_f} &B'\ar[r]^g & C\ar@{=}[d]
\\
A'\ar[r]^{f'}&B'\ar[r]^g & C && C\ar@{=}[rr]&& C\rlap{\,.}
}$$

Then we have the following diagram in $\bbY$ 
$$
\xymatrix@C=4em@R=2.5em{
\ar[r]^{Ff} \ar[d]|\bullet\ar@{}[dr]|{F\varphi}&\ar[d]|\bullet
\\
\ar[r]_{Ff'}\ar[dd]|\bullet\ar@{}[ddr]|\cong & \ar[d]|\bullet
\\
&\ar[r]^{Fg}\ar@{}[dr]|\cong\ar[d]|\bullet &\ar[d]|\bullet 
\\
\ar[d]|\bullet  \ar[r]_{\pi_2F_2(g,f')}&\ar@{}[d]|(.65){\mu_{g,f'}^{-1}} \ar[r]_{\pi_1F_2(g,f')}&\ar[d]|\bullet
\\
\ar[rr]_{F_1(gf')} &&\rlap{\quad.}
}$$
By Proposition \ref{comp-pres}, $F_1(gf')$ is a companion in $\bbY$, so  $\pi_2F_1(f',g)\circ\pi_1F_1(f',g)$
is a companion as well.

The fact that pseudo-functors preserve right pre-companions goes similarly.
\end{proof}

\section{The Weakly Globular Double Category of Fractions}\label{CW}

In this section we construct a weakly globular double category of fractions
for a category ${\bf C}$ and a class of arrows $W$ satisfying the conditions for a bicalculus of 
fractions, as spelled out below. 

Categories of fractions were first introduced by Gabriel and Zisman \cite{GZ}
in homotopy theory in order to localize a category with respect to a class of weak equivalences.
It has proved useful in contexts where there is some type of atlas or chart to represent an object.
For instance, the category of manifolds may be seen as a category of fractions of the category of atlas groupoids
with respect to atlas refinements, and the category of etendues and isomorphism classes of 
geometric morphisms is the category of fractions of the category of \'etale groupoids with 
respect to essential equivalences. The categories of topological, smooth, and algebraic stacks can be viewed 
as categories of fractions in a similar way. 
The category of fractions construction was generalized to a bicategory of fractions construction
by the second author in \cite{P1995} and this enables us to take the 2-categorical details
into account in these last examples, of etendues and stacks, and view them as bicategories of fractions.

In this section we introduce a weakly globular double category $\CW$ 
with the property that its fundamental bicategory $\Bic(\CW)$ is biequivalent to the bicategory 
of fractions $\bfC(W^{-1})$.
We could of course obtain such a weakly globular double category by taking the double category 
$\Dbl(\bfC(W^{-1}))$ of marked paths in
$\bfC(W^{-1})$. However, we will describe a weakly globular double category
which is (vertically) 2-equivalent to $\Dbl(\bfC(W^{-1}))$ and has an additional universal property as a double category
with companions and conjoints
which is not immediately derivable from the universal property of $\bfC(W^{-1})$ 
as we will discuss in the next section. Furthermore, one draw-back of both 
the category and bicategory of fractions constructions is that 
the resulting category or bicategory does not necessarily have small hom-sets. 
This is not the case for $\CW$: it has small homsets in both directions, both for arrows and double cells.

\subsection{Categories of fractions}
Let ${\bfC}$ be a category with a class $W\subseteq \bfC_1$ of arrows satisfying the conditions 
for a calculus of fractions
given by Gabriel and Zisman \cite{GZ} (which imply the existence of a category of fractions ${\bfC}[W^{-1}]$):
\begin{itemize}
\item{\bf CF1} $W$ contains all isomorphisms in ${\bfC}$ and is closed under composition;
\item{\bf CF2} For any diagram $\xymatrix@1{C\ar[r]^f&B&\ar[l]_w A}$ in $\bfC$ 
with $w\in W$, there exist arrows $\xymatrix@1{D \ar[r]^{\overline{f}} & A}$ and 
$\xymatrix@1{D\ar[r]^{\overline{w}} &C}$ with $\overline{w}\in W$ such that
$$
\xymatrix@R=1.8em@C=1.8em{
D\ar[r]^{\overline{f}}\ar[d]_{\overline{w}} & A\ar[d]^w
\\
C\ar[r]_f & B}
$$
commutes;
\item{\bf CF3} For any diagram $\xymatrix@1{A\ar@<-.5ex>[r]_f\ar@<.5ex>[r]^g &B\ar[r]^w& C}$ 
with $w\in W$, such that $wf=wg$, there exists an arrow $\xymatrix@1{X\ar[r]^{\tilde{w}}&A}$
in $W$ such that $f\tilde{w}=g\tilde{w}$.
\end{itemize}
Note that the conditions on a class of arrows $W$ in a bicategory ${\bfC}$ 
to form a bicategory of fractions coincide with these conditions when $\bfC$ is a category (cf.~\cite{P1995}).
So when a class $W$ of arrows satisfies these conditions, both the category ${\bfC}[W^{-1}]$
and the bicategory ${\bfC}(W^{-1})$ are defined (and the former can be obtained as a quotient of the latter).

\begin{eg}
 {\rm{\em Manifold Atlases} 
Consider manifolds with atlases such that the intersection of two atlas charts is 
either itself a chart or is covered by smaller atlas charts.
For any smooth map between two manifolds $f\colon M\rightarrow N$ 
there are such atlases $\calU$ for $M$ and $\calV$ for $N$ such that for each 
chart $U\in{\calU}$, there is a chart $V\in\calV$ with a smooth map $f_U\colon U\rightarrow V$
and such that for any pair of charts with an inclusion $i\colon U_1\hookrightarrow U_2$,  
there are charts $V_1$ and $V_2$
in $\calV$ with an inclusion $j\colon V_1\hookrightarrow V_2$ and such that $f_{U_i}\colon U_i\rightarrow V_i$
for $i=1,2$ and the diagram
$$\xymatrix@R=1.8em@C=1.8em{
U_1\ar[d]_i\ar[r]^{f_1} & V_1\ar[d]^j
\\
U_2\ar[r]_{f_2} & V_2
}
$$
commutes.
We may choose to start with manifolds represented by such atlases (without explicitly keeping track of 
the underlying spaces)
and consider atlas maps to be such families of smooth maps that are locally compatible as described above.
This gives us the category {\sf Atlases} of atlases and atlas maps.

To obtain the usual category {\sf Mfds} of manifolds from {\sf Atlases} we do the following.
First, note that if $\calU$ and $\calU'$ are two atlases for the same manifold and $\calU$ is a refinement 
of $\calU'$, then there is an induced atlas map $\calU\rightarrow\calU'$.
Furthermore, the class $W$ of atlas refinements satisfies the conditions {\bf CF1}, {\bf CF2} and {\bf CF3} above.
The corresponding category of fractions $\mbox{\sf Atlases}[W^{-1}]$ is categorically equivalent to {\sf Mfds}.
Arrows in this category can be thought of as first taking an atlas refinement and then mapping out.
In the category of fractions we consider two such maps to be the same if there is a common refinement
on which they would become the same (and this is the case precisely when they induce the same maps on the
underlying manifolds.) If we instead consider the bicategory of fractions, we consider two maps defined by 
different refinements as distinct but there is a unique invertible 2-cell between two such maps if
there is a common refinement where they become the same. 

All of this can also be expressed in the language of \'etale groupoids. Given an atlas for a manifold, there
is a canonical smooth \'etale groupoid corresponding to this atlas (where the space of objects is the disjoint 
union of the atlas charts). Atlas maps correspond to smooth functors between these groupoids.
A smooth functor between atlas groupoids corresponds to an atlas refinement if and only if it is an
essential equivalence of smooth groupoids. In this case the bicategory of fractions corresponds to the category 
of ringed toposes (of sheaves on the manifolds) and geometric morphisms of ringed toposes between them.
The details of this latter viewpoint can be found in \cite{P1995}, \cite{MP} and the second author's PhD-thesis
\cite{P-thesis}.
 
These references also show how all of this can be extended to orbifolds, but in that case we start with a 
2-category of atlases and atlas maps 
(where the 2-cells are coming from the group actions). 
We will address the consequences of this for the current set-up in terms of weakly globular double categories
in a sequel to this paper.}
\end{eg}

\subsection{Definitions and Basic Properties}\label{construction}
We define the weakly globular double category $\bfC\W$ as follows:
\begin{itemize}
\item {\em Objects} are arrows in $W$, $(w)=\left(\xymatrix@1{A\ar[r]^w&B}\right)$.
\item A {\em vertical arrow} $(u_1,C,u_2)\colon \xymatrix@1{(A_1\ar[r]^{w_1}&B)
\ar[r]|-\bullet &(A_2\ar[r]^{w_2}&B)}$ 
is an equivalence class of 
commutative diagrams
$$
\xymatrix@R=1.3em{A_1\ar[r]^{w_1} &B\ar@{=}[dd]
\\
C\ar[u]^{u_1}\ar[d]_{u_2}
\\
A_2\ar[r]_{w_2} &B}
$$
with $w_1u_1=w_2u_2$ in $W$.
Two such diagrams, $(u_1,C,u_2)$ and $(v_1,D,v_2)$, are equivalent (i.e., represent the same vertical arrow) 
when there are arrows $\xymatrix@1{C &E\ar[l]_{r_1}\ar[r]^{r_2}& D}$ 
such that 
$$
\xymatrix@C=1.5em{
&A_1
\\
C\ar[ur]^{u_1} \ar[dr]_{u_2} & E\ar[l]_{r_1}\ar[r]^{r_2} & D\ar[ul]_{v_1}\ar[dl]^{v_2}
\\
&A_2
}
$$
commutes and $w_1u_1r_1=w_1v_1r_2=w_2v_2r_2=w_2u_1r_1$ is in $W$.
\item A {\em horizontal arrow} $\xymatrix@1{(A\ar[r]^w&B)\ar[r]^f&(A'\ar[r]^{w'}&B')}$
is given by an arrow $\xymatrix@1{A\ar[r]^f&A'}$ in $\bfC$.
We will usually draw this as $\xymatrix@1{(B&\ar[l]_w A)\ar[r]^f&(A'\ar[r]^{w'}&B')}$.
\item
A {\em double cell} 
$$
\xymatrix{
(w_1)\ar[d]|\bullet_{[u_1,C,u_2]}\ar[r]^{f_1}\ar@{}[dr]|{\varphi} & (w_1')\ar[d]|\bullet^{[u_1',C',u_2']}\\
(w_2)\ar[r]_{f_2} & (w_2').
}$$
is an equivalence class of commutative diagrams of the form
\begin{equation}\label{square}
\xymatrix@C=5em@R=2em{
B\ar@{=}[dd] &A_1\ar[l]_{w_1}\ar[r]^{f_1} & A_1'\ar[r]^{w_1'} & B'\ar@{=}[dd]
\\
& C\ar[u]_{u_1}\ar[d]^{u_2}\ar[r]_{\varphi} &C' \ar[u]_{u_1'}\ar[d]^{u_2'} &
\\
B & A_2\ar[l]_{w_2}\ar[r]_{f_2} & A_2'\ar[r]_{w_2'} & B'
}
\end{equation}
The diagram (\ref{square}) is equivalent to the diagram
$$\xymatrix@C=5em@R=2em{
B\ar@{=}[dd] &A_1\ar[l]_{w_1}\ar[r]^{f_1} & A_1'\ar[r]^{w_1'} & B'\ar@{=}[dd]
\\
& D\ar[u]_{v_1}\ar[d]^{v_2}\ar[r]_{\psi} &D' \ar[u]_{v_1'}\ar[d]^{v_2'} &
\\
B & A_2\ar[l]_{w_2}\ar[r]_{f_2} & A_2'\ar[r]_{w_2'} & B'
}
$$
if and only if there are arrows $r,s,r',s'$ and $\chi$ 
as in 
$$
\xymatrix{
C\ar[rr]^\varphi&&C'\\
E\ar[u]^r\ar[d]_s\ar[rr]^\chi && E'\ar[u]_{r'}\ar[d]^{s'}\\
D\ar[rr]_\psi&&D'
}$$
such that $w_1u_1r\in W$, $w'_1u_1'r'\in W$,
and making the following diagram commute:
\begin{equation}\label{dblcellequiv}
\xymatrix@C=7ex{
\ar@{=}[4,0] && \ar[ll]_{w_1} \ar[rr]^{f_1} && \ar[rr]^{w_1'} && \ar@{=}[4,0]
\\
&&& \ar@/_1ex/[ul]^{v_1} \ar[rr]|(.22)\hole^(.58){\psi} \ar@/^1.5ex/[3,-1]_{v_2}|(.3)\hole|(.63)\hole 
		&& \ar@/_1ex/[ul]^{v_1'} \ar@/^1.5ex/[3,-1]^{v_2'} &&
\\
&& \ar[ur]_{s}\ar[dl]^r \ar[rr]_\chi|(.59)\hole &&\ar[ur]_{s'}\ar[dl]^{r'}
\\
& \ar@/^1.5ex/[-3,1]^{u_1}\ar@/_.5ex/[dr]_{u_2} \ar[rr]_{\varphi} 
		&&\ar@/^1.5ex/[-3,1]_{u_1'} \ar@/_.5ex/[dr]_{u_2'}
\\
&&\ar[ll]^{w_2}\ar[rr]_{f_2} && \ar[rr]_{w_2'}&&
}
\end{equation}
\end{itemize}

Note that a double cell may not have a representative for each combination of representatives of vertical arrows.
However, there is always a representative of the double cell with the given representative 
of the codomain vertical arrow  as in the following lemma.

\begin{lma}\label{reps}
 Given a representative of a double cell,
\begin{equation}\label{square1}
\xymatrix@C=5em@R=2em{
B\ar@{=}[dd] &A_1\ar[l]_{w_1}\ar[r]^{f_1} & A_1'\ar[r]^{w_1'} & B'\ar@{=}[dd]
\\
& C\ar[u]_{u_1}\ar[d]^{u_2} \ar[r]^\varphi &C' \ar[u]_{u_1'}\ar[d]^{u_2'} &
\\
B & A_2\ar[l]_{w_2}\ar[r]_{f_2} & A_2'\ar[r]_{w_2'} & B',
}
\end{equation}
and another representative $(v_1,D,v_2)$ of the codomain vertical arrow, then there are arrows $r$ and $\psi$
such that
\begin{equation}\label{square2}
\xymatrix@C=5em@R=2em{
B\ar@{=}[dd] &A_1\ar[l]_{w_1}\ar[r]^{f_1} & A_1'\ar[r]^{w_1'} & B'\ar@{=}[dd]
\\
& E\ar[u]_{u_1r}\ar[d]^{u_2r} \ar[r]^\psi &D \ar[u]_{v_1}\ar[d]^{v_2} &
\\
B & A_2\ar[l]_{w_2}\ar[r]_{f_2} & A_2'\ar[r]_{w_2'} & B',
}
\end{equation}
and (\ref{square1}) represent the same double cell.
\end{lma}

\begin{proof}
 Since 
 $$
 \xymatrix{
 A_1'\ar[r]^{w_1'}& B'\ar@{=}[dd] \ar@{}[2,2]|{\mbox{and}} && A_1'\ar[r]^{w_1'} & B'\ar@{=}[dd]
 \\
 C'\ar[u]_{u_1'}\ar[d]^{u_2'} &&&D\ar[u]_{v_1}\ar[d]^{v_2}
 \\
 A_2'\ar[r]_{w_2'}&B' && A_2'\ar[r]_{w_2'}&B'
 }
 $$
 represent the same vertical arrow, there are arrows $r_1$ and $r_2$ such that
 $w_1'u_1'r_1,w_1'v_1r_2\in W$ and
 $$
 \xymatrix{
 &A_1'
 \\
 C'\ar[ur]^{u_1'}\ar[dr]_{u_2'} & F\ar[l]_{r_1}\ar[r]^{r_2} & D\ar[ul]_{v_1}\ar[dl]^{v_2}
 \\
 &A_2'
 }
 $$
 commutes.
 By condition {\bf CF2} there are arrows $\bar{r}_1\in W$ and $\bar{\varphi}$  that make the following
 square commute,
 $$
 \xymatrix{
 \bar{F}\ar[d]_{\bar{r}_1}\ar[rr]^{\bar\varphi} && F\ar[d]^{w_2'v_2r_1}
 \\
 C\ar[r]_{\varphi} & C'\ar[r]_{w_2'u_2'} & B'.
 }
 $$
 By condition {\bf CF3} there exists an arrow $(\tilde{v}\colon E\rightarrow\bar{F})\in W$
 such that 
 $$
 \xymatrix{
 E\ar[d]_{\bar{r}_1\tilde{v}}\ar[r]^{\bar{\varphi}\tilde{v}}&F\ar[d]^{r_1}
 \\
 C\ar[r]_\varphi & C'
 }
 $$
 commutes.
 We can use all this to construct the following commutative diagram
 $$
 \xymatrix{
 B\ar@{=}[4,0] && A_1\ar[ll]_{w_1} \ar[r]^{f_1} & A_1'\ar[rr]^{w_1'} && B'\ar@{=}[4,0]
 \\
 &&C\ar[u]^{u_1}\ar[r]^{\varphi} &C'\ar[u]^{u_1'}
 \\
 &&E\ar[u]^{\bar{r}_1\tilde{v}}\ar[r]^{\bar{\varphi}\tilde{v}}\ar[d]_{\bar{r}_1\tilde{v}} 
 			& F\ar[u]^{r_1}\ar[d]_{r_1}\ar[r]^{r_2}&D\ar[uul]_{v_1}\ar[ddl]^{v_2}
 \\
 &&C\ar[r]_\varphi\ar[d]^{u_2} & C'\ar[d]_{u_2'}
 \\
 B&&A_2\ar[ll]^{w_2}\ar[r]_{f_2} & A_2'\ar[rr]_{w_2'}&& B'
 }
 $$
 So we obtain a double cell 
 $$
 \xymatrix{
 B\ar@{=}[dd]&A_1\ar[l]_{w_1}\ar[r]^{f_1} & A_1'\ar[r]^{w_1'} &B'\ar@{=}[dd]
 \\
 &E\ar[r]^{r_2\bar{\varphi}\tilde{v}} \ar[u]^{u_1\bar{r}_1\tilde{v}}\ar[d]_{u_2\bar{r}_1\tilde{v}}&D\ar[u]_{v_1}\ar[d]^{v_2}
 \\
 B&A_2\ar[l]^{w_2}\ar[r]_{f_2}& A_2'\ar[r]_{w_2'}& B'. 
 }
 $$
 This cell is equivalent to (\ref{square1}) since we have a commutative diagram
 $$
 \xymatrix{
 E\ar@{=}[d]\ar[r]^{r_2\bar{\varphi}\tilde{v}}&D
 \\
 E\ar[r]^{\bar{\varphi}\tilde{v}}\ar[d]_{\bar{r}_1\tilde{v}} & F\ar[u]_{r_2}\ar[d]^{r_1}
 \\
 C\ar[r]_{\varphi}& C'
 }
 $$
 which fits in a diagram like (\ref{dblcellequiv}).
 By taking $r=\bar{r}_1\tilde{v}$ and $\psi=r_2\bar{\varphi}\tilde{v}$ we see that we have indeed
  a cell as in (\ref{square2}).
\end{proof}

\begin{eg}
{\rm In the example where $\bfC=\mbox{\sf Atlases}$ and $W$ the class of atlas refinements,
the situation is  as follows.
The objects of $\CW$ are atlases with a refinement. There is a vertical arrow between any two such objects 
if there is a common refinement, and a horizontal arrow corresponds to an atlas map between the refinements.
Double cells correspond to common refinements where the two atlas maps corresponding to the horizontal arrows
agree.}
\end{eg}
 
The following lemma is equivalent to stating that the codomain map $d_1\colon\CW_1\rightarrow \CW_0$,
sending squares and horizontal arrows to their horizontal codomains is an isofibration.

\begin{lma}\label{isofib}
For every pair of a horizontal arrow $\xymatrix@1{(w_2)\ar[r]^g&(w_2')}$ and a vertical arrow 
$\xymatrix@1@C=3.5em{(w_1')\ar[r]|\bullet^{[v_1,D,v_2]}&(w_2')}$ in $\CW$, there is a double cell 
$$
\xymatrix{
(w_1)\ar[r]^f\ar[d]|\bullet_{[u_1,C,u_2]}\ar@{}[dr]|{(\varphi)} & (w_1')\ar[d]|\bullet^{[v_1,D,v_2]}
\\
(w_2)\ar[r]_g & (w_2')\rlap{\,.}
}
$$
\end{lma}

\begin{proof}
Let $\xymatrix@1{A_2\ar[r]^{w_2}&B}$ and $\xymatrix@1{A_i'\ar[r]^{w_i'}&B'}$, 
for $i=1,2$ in $W$.
By condition {\bf CF2} there exists a commutative square 
$$
\xymatrix{
E\ar[r]^{\varphi_E}\ar[d]_{\overline{v}_2}&D\ar[d]^{w_2'v_2}
\\
A_2\ar[r]_{w_2'g} & B'}
$$
in $\bfC$.
Since $w_2'\in W$, there is an arrow $(\xymatrix@1{C\ar[r]^{\tilde{w}'_2}&E})\in W$ 
such that
$$
\xymatrix{
C\ar[d]_{\overline{v}_2\tilde{w}_2'}\ar[r]^{\varphi_E\tilde{w}_2'} & D\ar[d]^{v_2}
\\
A_2\ar[r]_g&A_2'
}
$$
also commutes (by condition {\bf CF3}).
So let $\varphi_C=\varphi_E\tilde{w}_2'$ $u_2=\overline{v}_2\tilde{w}_2'$. Note that $u_2\in W$,
and the diagram
$$
\xymatrix@R=2em{
(B\ar@{=}[dd]&\ar[l]_{w_2u_2}C)\ar[r]^{v_1\varphi_C} & (A_1'\ar[r]^{w_1'}&B'\ar@{=}[dd])
\\
	& C\ar@{=}[u]\ar[d]^{u_2}\ar[r]^{\varphi_C}&D\ar[d]_{v_2}\ar[u]^{v_1}
\\
(B&A_2\ar[l]^{w_2})\ar[r]_g&(A_2'\ar[r]_{w_2'}&B')
}
$$
commutes in $\bfC$ and satisfies all the conditions to represent a double cell
in $\CW$.
\end{proof}

The following lemma implies that the vertical arrow category is a posetal groupoid.
 
\begin{lma}\label{connectedcomp}
 There is a vertical arrow $\xymatrix@1{(A\ar[r]^w&B)\ar[r]|\bullet &(A'\ar[r]^{w'}&B')}$
if and only if $B=B'$ and furthermore, this arrow is unique.
\end{lma}

\begin{proof}
It is obvious that the existence of such a vertical arrow implies that $B=B'$.

Now suppose that we have objects $\xymatrix@1{(A\ar[r]^w&B)}$ and $\xymatrix@1{(A'\ar[r]^{w'}&B)}$
in $\CW$.
Since $W$ satisfies condition {\bf CF2} above, there is a commutative square
$$
\xymatrix@R=1.5em{C\ar[r]^{u}\ar[d]_{u'}&A\ar[d]^{w}
\\
A'\ar[r]_{w'} &B
}$$
in $\bfC$ with $u'\in W$ and consequently, $wu=w'u'\in W$.
So $$\xymatrix@R=1.5em{A\ar[r]^w & B\ar@{=}[dd]
\\
C\ar[u]^{u}\ar[d]_{u'}
\\
A'\ar[r]_{w'}&B
}$$
represents a vertical arrow as required.

To show that this arrow is unique, suppose that we have two representatives
$$
\xymatrix@R=1.5em{
A\ar[r]^w & B\ar@{=}[dd] &&A\ar[r]^w & B\ar@{=}[dd]
\\
C\ar[u]^{u}\ar[d]_{u'}&&\mbox{and}&D\ar[u]^{v}\ar[d]_{v'}&
\\
A'\ar[r]_{w'}&B&&A'\ar[r]_{w'}&B.
}$$
By conditions {\bf CF1} and {\bf CF2}, there are arrows  $\xymatrix@1{C&E\ar[l]_r\ar[r]^s& D}$ 
such that $wur=wvs\in W$. Now it follows that $w'v's=wvs=wur=w'u'r$ and $w'\in W$. 
By condition {\bf CF3}, there is an arrow $\xymatrix@1{E'\ar[r]^{\tilde w}& E}$ such that
$v's\tilde{w}=u'r\tilde{w}$. 
So the following diagram commutes:
$$
\xymatrix{
&A&
\\
C \ar[ur]^{u}\ar[dr]_{u'} & E'\ar[l]_{r\tilde{w}}\ar[r]^{s\tilde{w}}&D\ar[ul]_{v}\ar[dl]^{v'}
\\
&A'
}
$$
and hence $[u,C,u']=[v,D,v']$.
\end{proof}

Note that this implies that $\pi_0$ of the vertical arrow category of $\bfC\W$
is isomorphic to the objects of $\bfC$. 

The category of double cells in $\CW$ with vertical composition (and horizontal arrows as objects) is also
posetal, as stated in the following lemma.

\begin{lma}\label{uniquedblcell}
For any pairs of horizontal arrows and vertical arrows fitting together as in
$$
\xymatrix{
(w_1)\ar[d]|\bullet_{[u_1,C,u_2]} \ar[r]^{f_1} & (w'_1) \ar[d]|\bullet^{[v_1,D,v_2]}
\\
(w_2)\ar[r]_{f_2} & (w_2')
}
$$
there is at most one double cell in $\bfC\W$ that fills this.
\end{lma}

\begin{proof}
Suppose that both $(\varphi)$ and $(\psi)$ fit in this frame.
By Lemma \ref{reps} we may assume that  $\varphi$ and $\psi$ 
are of the following form:
$$
\xymatrix{
\ar@{=}[dd] & \ar[l]_{w_1}\ar[r]^{f_1}&\ar[r]^{w_1'} & \ar@{=}[dd]&&\ar@{=}[dd] & \ar[l]_{w_1}\ar[r]^{f_1}&\ar[r]^{w_1'} & \ar@{=}[dd]
\\
&\ar[u]^{u_1'}\ar[d]_{u_2'}\ar[r]^\varphi & \ar[u]_{v_1}\ar[d]^{v_2}&&\mbox{and}&&\ar[u]^{u_1''}\ar[d]_{u_2''}\ar[r]^\psi & \ar[u]_{v_1}\ar[d]^{v_2}
\\
&\ar[l]^{w_2}\ar[r]_{f_2}&\ar[r]_{w_2'}& &&&\ar[l]^{w_2}\ar[r]_{f_2}&\ar[r]_{w_2'}&
}
$$
Since $[u_1',u_2']=[u_1,u_2]=[u_1'',u_2'']$, there are arrows
$r'$ and $r''$ such that $w_1u_1'r',w_1u_1''r''\in W$
and 
the following diagram commutes
$$
\xymatrix{
&
\\
\ar[ur]^{u_1'}\ar[dr]_{u_2'}&\ar[l]_{r'}\ar[r]^{r''}&\ar[ul]_{u_1''}\ar[dl]^{u_2''}
\\
&
}
$$
Note that $w_1'v_1\varphi r'=w_1'f_1u_1'r'=w_1'f_1u_1''r''=w_1'v_1\psi r''$.
Since $w_1'v_1\in W$, we can apply condition {\bf CF3} to obtain an arrow $\tilde{w}$
such that $\varphi r'\tilde{w}=\psi r''\tilde{w}$. Finally,
we see that 
$$
\xymatrix@C=3em{
\ar@{=}[dd] & \ar[l]_{w_1}\ar[r]^{f_1}&\ar[r]^{w_1'} & \ar@{=}[dd]
		&&\ar@{=}[dd] & \ar[l]_{w_1}\ar[r]^{f_1}&\ar[r]^{w_1'} & \ar@{=}[dd]
\\
&\ar[u]^{u_1'}\ar[d]_{u_2'}\ar[r]^\varphi & \ar[u]_{v_1}\ar[d]^{v_2}
		&&\sim&&\ar[u]^{u_1'r'\tilde{w}}\ar[d]_{u_2'r'\tilde{w}}\ar[r]^{\varphi r'\tilde{w}} & \ar[u]_{v_1}\ar[d]^{v_2}&&=
\\
&\ar[l]^{w_2}\ar[r]_{f_2}&\ar[r]_{w_2'}& &&&\ar[l]^{w_2}\ar[r]_{f_2}&\ar[r]_{w_2'}&
\\
\ar@{=}[dd] & \ar[l]_{w_1}\ar[r]^{f_1}&\ar[r]^{w_1'} & \ar@{=}[dd]
		&&\ar@{=}[dd] & \ar[l]_{w_1}\ar[r]^{f_1}&\ar[r]^{w_1'} & \ar@{=}[dd]
\\
&\ar[u]^{u_1''r''\tilde{w}}\ar[d]_{u_2''r''\tilde{w}}\ar[r]^{\psi r'' \tilde{w}} & \ar[u]_{v_1}\ar[d]^{v_2}&
		&\sim&&\ar[u]^{u_1''}\ar[d]_{u_2''}\ar[r]^\psi & \ar[u]_{v_1}\ar[d]^{v_2}
\\
&\ar[l]^{w_2}\ar[r]_{f_2}&\ar[r]_{w_2'}& 
		&&&\ar[l]^{w_2}\ar[r]_{f_2}&\ar[r]_{w_2'}&
}
$$
So $\varphi$ and $\psi$ represent the same double cell.
\end{proof}

\begin{rmk}\rm Since the vertical arrow category of $\CW$ is posetal and groupoidal, this lemma implies that the
vertical category $(\CW)_1$  of horizontal arrows and double cells is also posetal and groupoidal.
\end{rmk}

\begin{thm}
The double category $\CW$ is weakly globular.
\end{thm}

\begin{proof}
By Lemma \ref{connectedcomp} the vertical arrow category of $\CW$ is groupoidal
and posetal as required. Furthermore,  the
codomain $d_1\colon (\CW)_1\rightarrow (\CW)_0$ is an isofibration.
For any pair of a horizontal arrow  $\xymatrix@1{(w_2)\ar[r]^g&(w_2')}$ and a vertical arrow 
$\xymatrix@1@C=3.5em{(w_1')\ar[r]|\bullet^{[v_1,D,v_2]}&(w_2')}$, there is a double cell 
$$
\xymatrix{
(w_1)\ar[r]^f\ar[d]|\bullet_{[u_1,C,u_2]}\ar@{}[dr]|{(\varphi)} & (w_1')\ar[d]|\bullet^{[v_1,D,v_2]}
\\
(w_2)\ar[r]_g & (w_2')
}
$$
by Lemma \ref{isofib}. And this double cell is vertically invertible since the vertical category $(\CW)_1$ is posetal groupoidal
by Lemma \ref{uniquedblcell}.
\end{proof}

\subsection{Composition}
{\em Composition of vertical arrows} is defined using condition {\bf CF2}. 
For a pair of composable vertical arrows,
$$
\xymatrix@R=1.2em{A_1\ar[r]^{w_1} & B\ar@{=}[dd]
\\
C\ar[u]^{u_1}\ar[d]_{u_2}&
\\
A_2\ar[r]_{w_2} & B\ar@{=}[dd]
\\
D\ar[u]^{v_1}
\ar[d]_{v_2} &
\\
A_3\ar[r]_{w_3} & B,
}
$$
there is a commutative square in $\bfC$ of the form
$$
\xymatrix{
E\ar[d]_{\bar{u}_2}\ar[r]^{\bar{v}_1}&C\ar[d]^{w_2u_2}
\\
D\ar[r]_{w_2v_1} & A_2
}
$$ 
with $u_2\bar{v}_1=v_1\bar{u}_2\in W$ (by Condition {\bf CF2}).
So $w_1u_1\bar{v}_1=w_2u_2\bar{v}_1=w_2v_1\bar{u}_2=w_3v_2\bar{u}_2$, and
a representative for the vertical composition is given by
$$
\xymatrix@R=1.6em{
A_1\ar[r]^{w_1} & B\ar@{=}[dd]
\\
D\ar[u]^{u_1\bar{v}_1}\ar[d]_{v_2\bar{u}_2}
\\
A_3\ar[r]_{w_3} &B
}
$$

Note that this is well-defined on equivalence classes, doesn't depend on the choice of the
square in condition {\bf CF2}, and is associative and unital, because the vertical arrow category is posetal.

Analogously, to define {\em vertical composition of double cells}, 
$$
\xymatrix@R=1.6em{
(B\ar@{=}[dd] & A_1)\ar[l]_{w_1}\ar[r]^{f_1} & (A_1'\ar[r]^{w_1'} & B')\ar@{=}[dd]
\\
&C\ar[u]^{u_1}\ar[d]_{u_2}\ar@{}[r]|{(\varphi)}&C'\ar[u]^{u_1'}\ar[d]_{u_2'}
\\
(B\ar@{=}[dd]&A_2)\ar[l]^{w_2}\ar[r]_{f_2}&(A_2'\ar[r]_{w_2'} & B')\ar@{=}[dd]
\\
&D\ar[u]^{v_1}
\ar[d]_{v_2} \ar@{}[r]|{(\psi)}&D'\ar[u]^{v_1'}
\ar[d]_{v_2'}
\\
(B&A_3)\ar[l]^{w_3}\ar[r]_{f_3}&(A_3'\ar[r]_{w_3} & B)
}
$$
we only need to give a representative to fit the square with
the composed boundary arrows. Without loss of generality we may assume that
the double  cells we are composing are such that
$\varphi\colon C\rightarrow C'$ and $\psi\colon D\rightarrow D'$ fit in the diagrams given.
So suppose that
we have commutative squares
$$
\xymatrix{
E\ar[r]^{\bar{v}_1}\ar[d]_{\bar{u}_2} & C\ar[d]^{u_2} & E'\ar[r]^{\bar{v}_1'}\ar[d]_{\bar{u}_2'} & C\ar[d]^{u_2'}
\\
D\ar[r]_{v_1} & A_2 & D'\ar[r]_{v_1'}& A_2',
}
$$
giving the following composition of the vertical arrows involved:
$$
\xymatrix@C=4em@R=1.6em{
(B\ar@{=}[dd] & \ar[l]_{w_1} A_1)\ar[r]^{f_1} & (A_1'\ar[r]^{w_1'}& B')\ar@{=}[dd]
\\
& E\ar[u]^{u_1\bar{v}_1}\ar[d]_{v_2\bar{u}_2}\ar@{}[r]|{(\psi)\cdot(\varphi)} 
	& E'\ar[u]^{u_1'\bar{v}_1'}\ar[d]_{v_2'\bar{u}_2'}
\\
(B & \ar[l]^{w_3} A_3) \ar[r]_{f_3} & (A_3'\ar[r]_{w_3'} & B')
}
$$
We will now construct $(\xi)=(\psi)\cdot(\varphi)$.
Note that it is sufficient to find an arrow $\xymatrix@1{\bar{E}\ar[r]^r&E}$ in $W$
with an arrow $\xi\colon \bar{E}\rightarrow E'$
such that $\bar{v}_1'\xi =\varphi \bar{v}_1r$ and $\bar{u}_2'\xi=\psi\bar{u}_2r$.

Apply condition {\bf CF2} to obtain a commutative square
$$\xymatrix{
\tilde{E}\ar[d]_{\bar{\bar{v}}_1'}\ar[rr]^{\bar{\varphi}} &&E'\ar[d]^{\bar{v}_1'}
\\
E\ar[r]_{\bar{v}_1} & C\ar[r]_\varphi & C'.
}
$$ in $\bfC$ with $\bar{\bar{v}}_1'\in W$.
Then
\begin{eqnarray*}
w_2'v_1'\bar{u}_2'\bar{\varphi}&=&w_2'u_2'\bar{v}_1'\bar{\varphi}\\
&=&w_2'u_2'\varphi\bar{v}_1\bar{\bar{v}}_1'\\
&=&w_2'f_2u_2\bar{v}_1 \bar{\bar{v}}_1'\\
&=&w_2'f_2v_1\bar{u}_2\bar{\bar{v}}_1'\\
&=&w_2'v_1'\psi\bar{u}_2\bar{\bar{v}}_1'
\end{eqnarray*}
Since $w_2'v_1'\in W$, there is an arrow $s\colon \bar{E}\rightarrow \tilde{E}$
such that $\bar{u}_2'\bar{\varphi}s=\psi\bar{u}_2\bar{\bar{v}}_1's$ (by condition {\bf CF3}).
So let $r=\bar{\bar{v}}_1's$ and $\xi=\bar{\varphi}s$.

{\em Composition of horizontal arrows} in $\bfC\W$ is defined directly in terms of composition of arrows in $\bfC$:
the horizontal composition of $\xymatrix{(B&\ar[l]_{w}A)\ar[r]^f &(A'\ar[r]^{w'}&B')}$
and $\xymatrix{(B'& \ar[l]_{w'}A)\ar[r]^{f'} &(A''\ar[r]^{w''}&B'')}$
is $\xymatrix{(B&\ar[l]_{w}A)\ar[r]^{f'f} &(A'\ar[r]^{w''}&B'')}$.

{\em Horizontal composition of double cells} is defined as follows.
Let 
$$
\xymatrix{
\ar@{=}[dd] & \ar[l]_{w_1}\ar[r]^{f_1}&\ar[r]^{w_1'} & \ar@{=}[dd]
		&&\ar@{=}[dd] & \ar[l]_{w_1'}\ar[r]^{g_1}&\ar[r]^{w_1''} & \ar@{=}[dd]
\\
&\ar[u]^{u_1}\ar[d]_{u_2}\ar[r]^\varphi & \ar[u]_{u_1'}\ar[d]^{u_2'}&&\mbox{and}&&\ar[u]^{v_1}\ar[d]_{v_2}\ar[r]^\psi & \ar[u]_{v_1'}\ar[d]^{v_2'}
\\
&\ar[l]^{w_2}\ar[r]_{f_2}&\ar[r]_{w_2'}& &&&\ar[l]^{w_2'}\ar[r]_{g_2}&\ar[r]_{w_2''}&
}
$$
be two horizontally composable double cells.
This means that $[u_1',u_2']=[v_1,v_2]$, so there exist arrows $r$ and $s$ such that
$w_1'u_1'r\in W$ and the following diagram commutes,
$$
\xymatrix{
&
\\
\ar[ur]^{u_1'} \ar[dr]_{u_2'}& \ar[l]_r\ar[r]^s & \ar[ul]_{v_1}\ar[dl]^{v_2}
\\
&
}
$$
By conditions {\bf CF2} and {\bf CF3} there is a commutative square of the form
$$
\xymatrix{
\ar[d]_{\bar{r}}\ar[r]^{\bar{\varphi}} & \ar[d]^r\\
\ar[r]_\varphi&
}
$$
with $\bar{r}\in W$.
Now it is not difficult to check that
$$
\xymatrix{
\ar@{=}[dd] & \ar[l]_{w_1}\ar[r]^{f_1}&\ar[r]^{w_1'} & \ar@{=}[dd]
		&&\ar@{=}[dd] & \ar[l]_{w_1}\ar[r]^{f_1}&\ar[r]^{w_1'} & \ar@{=}[dd]
\\
&\ar[u]^{u_1}\ar[d]_{u_2}\ar[r]^\varphi & \ar[u]_{u_1'}\ar[d]^{u_2'}&
		&\sim&&\ar[u]^{u_1\bar{r}}\ar[d]_{u_2\bar{r}}\ar[r]^{\bar{\varphi}} & \ar[u]_{u_1'r}\ar[d]^{u_2'r}
\\
&\ar[l]^{w_2}\ar[r]_{f_2}&\ar[r]_{w_2'}& &&&\ar[l]^{w_2}\ar[r]_{f_2}&\ar[r]_{w_2'}&
}
$$
and
$$
\xymatrix{
\ar@{=}[dd] & \ar[l]_{w_1'}\ar[r]^{g_1}&\ar[r]^{w_1''} & \ar@{=}[dd]
		&&\ar@{=}[dd] & \ar[l]_{w_1'}\ar[r]^{g_1}&\ar[r]^{w_1''} & \ar@{=}[dd]
\\
&\ar[u]^{v_1}\ar[d]_{v_2}\ar[r]^\psi & \ar[u]_{v_1'}\ar[d]^{v_2'}&
		&\sim&&\ar[u]^{v_1s}\ar[d]_{v_2s}\ar[r]^{\psi s} & \ar[u]_{v_1'}\ar[d]^{v_2'}
\\
&\ar[l]^{w_2'}\ar[r]_{g_2}&\ar[r]_{w_2''}& &&&\ar[l]^{w_2'}\ar[r]_{g_2}&\ar[r]_{w_2''}&
}
$$
Since $u_1'r=v_1s$ and $u_2'r=v_2s$, the horizontal composition is given by
$$
\xymatrix{
\ar@{=}[dd] &\ar[l]_{w_1}\ar[r]^{g_1}{f_1} & \ar[r]^{w_1''}&\ar@{=}[dd]
\\
&\ar[u]^{u_1\bar{r}}\ar[d]_{u_2\bar{r}}\ar[r]^{\psi s\bar{\varphi}}& \ar[u]_{v_1'}\ar[d]^{v_2'}
\\
&\ar[l]^{w_2}\ar[r]_{g_2f_2}&\ar[r]_{w_2''} &\rlap{\,.}
}
$$

Middle four interchange holds because the double cells are posetal in the vertical direction.

\begin{rmk}
{\rm For any objects $w_1,w_2\in \CW_{00}$, the hom-set $\Hom_{\CW,h}(w_1,w_2)$ is small 
since it is isomorphic to $\Hom_{\bfC}(d_0(w_1),d_0(w_2))$  and $\Hom_{\CW,v}(A,B)$ is small since
it contains at most one element. Furthermore, for any given pair of horizontal or vertical arrows,
there is a set of double cells  with those arrows as domain and codomain, since given any frame 
of horizontal and vertical arrows, there is at most one double cell that fills it.
This shows that although $\bfC(W^{-1})$ is generally not a bicategory with small hom-sets (and consequently, 
its weakly globular double category of marked paths doesn't have small horizontal hom-sets either),
the double category $\CW$ does. This is one of the advantages of $\CW$ over $\bfC(W^{-1})$, for instance when one wants to consider the classifying space of
such a category.}
\end{rmk}

\subsection{The fundamental bicategory of $\CW$}
In this section we establish the relationship between the weakly globular double category $\bfC\W$
and the bicategory of fractions $\bfC(W^{-1})$. We will construct a biequivalence of bicategories
$$\omega\colon\Bic(\CW)\rightarrow \bfC(W^{-1}).$$

The objects of the associated bicategory $\Bic(\CW)$ are the connected components 
of the vertical arrow category of 
$\bfC\W$. By Lemma \ref{connectedcomp} the map 
$$\omega_0\colon\Bic(\CW)_0\rightarrow \bfC(W^{-1})_0=\bfC_0,$$
sending the connected component of $\xymatrix@1{(A\ar[r]^w&B)}$ to $B$ is an isomorphism.

An arrow in $\Bic(\CW)$ corresponds to a horizontal arrow in $\CW$, so it is given by a diagram
$\xymatrix@1{B&A\ar[l]_{w}\ar[r]^f & A'\ar[r]^{w'}&B'}$ in $\bfC$ with $w,w'\in W$.
We define 
$$\omega_1(\xymatrix{B&A\ar[l]_{w}\ar[r]^f & A'\ar[r]^{w'}&B'})=(\xymatrix{B&A\ar[l]_-w\ar[r]^-{w'f}&B'}).$$

A 2-cell $\varphi$ from $\xymatrix@1{B&A_1\ar[l]_{w_1}\ar[r]^{f_1} & A_1'\ar[r]^{w_1'}&B'}$
to $\xymatrix@1{B&A_2\ar[l]_{w_2}\ar[r]^{f_2} & A_2'\ar[r]^{w_2'}&B'}$
in $\Bic(\CW)$
corresponds to an equivalence class of commutative diagrams of the form
$$
\xymatrix@R=1.5em{
B\ar@{=}[dd]&A_1\ar[l]_{w_1}\ar[r]^{f_1} & A_1'\ar[r]^{w_1'}&B'\ar@{=}[dd]
\\
&C\ar[u]^{u_1}\ar[d]_{u_2}\ar[r]^\varphi & C'\ar[u]_{v_1}\ar[d]^{v_2}
\\
B&A_2\ar[l]_{w_2}\ar[r]^{f_2} & A_2'\ar[r]^{w_2'}&B'
}$$
This means that $w_1'f_1u_1=w_1'v_1\varphi=w_2'v_2\varphi=w_2'f_2u_2$ and $w_1u_1=w_2u_2$, so
we define 
$$\omega_2(\varphi)=\left(\raisebox{3.2em}{$
\xymatrix@R=1.8em{ & A_1\ar[dl]_{w_1}\ar[dr]^{w_1'f_1}
\\
B&C\ar[u]^{u_1}\ar[d]_{u_2} & B'
\\
&A_2\ar[ul]^{w_2}\ar[ur]_{w_2'f_2}
}$}\right)$$
The fact that this is well defined on equivalence classes of double cells
follows from the following lemma about the 2-cells in a bicategory of fractions of a category.

\begin{lma}
For  any category $\bfC$ with a class of arrows $W$ satisfying the conditions {\bf CF1}, {\bf CF2}, and {\bf CF3} 
above, the bicategory of fractions $\bfC(W^{-1})$ has at most one 2-cell between any two arrows.
\end{lma}

\begin{proof}
The arrows in this bicategory are spans $\xymatrix@1{A&S\ar[l]_w\ar[r]^f & B}$ with $w\in W$.
A 2-cell from  $\xymatrix@1{A&S_1\ar[l]_{w_1}\ar[r]^{f_1} & B}$
to $\xymatrix@1{A&S_2\ar[l]_{w_2}\ar[r]^{f_2} & B}$ is represented by a commutative diagram of the form
$$
\xymatrix@R=1.3em{
&S_1\ar[dl]_{w_1}\ar[dr]^{f_1}\\
A &T\ar[u]_{u_1}\ar[d]^{u_2}&B\\
&S_2\ar[ul]^{w_2}\ar[ur]_{f_2}
}
$$
Two such diagrams,
\begin{equation}\label{two cells}
\xymatrix@R=1.3em{
&S_1\ar[dl]_{w_1}\ar[dr]^{f_1}&&& &S_1\ar[dl]_{w_1}\ar[dr]^{f_1}\\
A &T\ar[u]_{u_1}\ar[d]^{u_2}&B&\mbox{and} & A &T'\ar[u]_{u'_1}\ar[d]^{u'_2}&B\\
&S_2\ar[ul]^{w_2}\ar[ur]_{f_2} &&& &S_2\ar[ul]^{w_2}\ar[ur]_{f_2}
}
\end{equation}
 represent the same 2-cell when there exist arrows $t$ and $t'$ such that the
 diagram 
\begin{equation}\label{equivalence}
\xymatrix@R=1.3em{
&S_1
\\
T\ar[ur]^{u_1} \ar[dr]_{u_2}& \bar{T}\ar[l]_t\ar[r]^{t'} & T'\ar[ul]_{u_1'}
\ar[dl]^{u_2'} 
\\
&S_2
}
\end{equation}
commutes and $w_1u_1t\in W$.
For any two 2-cells as in (\ref{two cells}), we can find $t$ and $t'$ as in (\ref{equivalence})
by first using condition {\bf CF2} to find a commutative square
$$\xymatrix{
\ar[r]^r\ar[d]_{r'} &\ar[d]^{w_1u_1}
\\
\ar[r]_{w_1u_1'}&\rlap{\,.}
}
$$
We have then also that $w_2u_2r=w_1u_1r=w_1u_1'r'=w_2u_2'r'$.
By condition {\bf CF3} applied to both $w_1$ and $w_2$ we find that there
is an arrow $\tilde{w}\in W$ such that  $u_1r\tilde{w}=u_1'r'\tilde{w}$ and $u_2r\tilde{w}=u_2'r'\tilde{w}$.
So $t=r\tilde{w}$ and $t'=r'\tilde{w}$ fit in the diagram (\ref{equivalence}).
\end{proof}

\begin{thm}\label{equiv}
The bicategories, $\Bic(\bfC\W)$ and $\bfC(W^{-1})$ are equivalent in the 2-category
$\NHom$.
\end{thm} 

\begin{proof}
We first show that the functions $\omega_0$, $\omega_1$ and $\omega_2$ defined above give 
a homomorphism of bicategories:
$\omega\colon\Bic(\CW)\rightarrow\bfC(W^{-1})$.
To prove this we only need to give the comparison 2-cells for units and composition. 
(They will automatically be coherent since the hom-categories in $\bfC(W^{-1})$ are posetal.)

We will represent each connected component of $(\CW)_0$ by the identity arrow in it.
So an identity arrow in $\Bic(\CW)$ has the form $\xymatrix@1{B&B\ar[l]_{1_B}\ar[r]^{1_B} & B\ar[r]^{1_B}&B}$,
and $$\omega_1\left(\xymatrix@1{B&B\ar[l]_{1_B}\ar[r]^{1_B} & B\ar[r]^{1_B}&B}\right)=
\left(\xymatrix@1{B&B\ar[l]_{1_B}\ar[r]^{1_B}&B}\right).$$ So $\omega_1$ preserves identities strictly.

For any two composable arrows 
\begin{equation}\label{composable}
 \xymatrix@1{B_1&\ar[l]_{w_1}A_1\ar[r]^{f}&A_2\ar[r]^{w_2}&B_2}
\mbox{ and }\xymatrix@1{B_2&A_3\ar[l]_{w_3}\ar[r]^{g}&A_4\ar[r]^{w_4}&B_3}
\end{equation}
in $\Bic(\CW)$ ,
the composition is found by first constructing the following double cell in $\CW$,
$$
\xymatrix@R=1.5em{
(B_1\ar@{=}[dd] 
	&C)\ar[l]_{w_1\bar{\bar{w_3}}}\ar@{=}[d]\ar[r]^{\bar{w}_2\bar{f}} 
	& (A_3\ar[r]^{w_3}&B_2)\ar@{=}[dd]
\\
&C\ar[d]_{\bar{\bar{w}}_3}\ar[r]^{\bar{f}} &D\ar[u]_{\bar{w}_2}\ar[d]^{\bar{w}_3} 
\\
(B_1&\ar[l]_{w_1}A_1)\ar[r]^{f}&(A_2\ar[r]^{w_2}&B_2)
}$$
(using chosen squares in $\bfC$ to obtain $\bar{w}_2$, $\bar{w}_3$, $\bar{\bar{w}}_3$ and $\bar{f}$)
and then composing the domain horizontal arrow with $\xymatrix@1{(B_2&A_3)\ar[l]_{w_3}\ar[r]^{g}&(A_4\ar[r]^{w_4}&B_3)}$ to get
$$
 \xymatrix@1{(B_1&\ar[l]_{w_1\bar{\bar{w}}_3}C)\ar[rr]^{g\bar{w}_2\bar{f}}&&(A_4\ar[r]^{w_4}&B_3)\rlap{\,.}}
$$
So $\omega_1$ applied to the composite gives
\begin{equation}\label{imageofcomp}
 \xymatrix@C=4em{B_1&C\ar[l]_{w_1\bar{\bar{w}}_3}\ar[r]^{w_4g\bar{w}_2\bar{f}}&B_3\rlap{\,.}}
\end{equation}
On the other hand, the composite of the images of the arrows in (\ref{composable}) under $\omega$
is
\begin{equation}\label{compofimages}
 \xymatrix@C=4em{B_1&E\ar[l]_{w_1\hat{w}_3}\ar[r]^{w_4g\widehat{w_2f}}&B_3\rlap{\,,}}
\end{equation}
where $\hat{w}_3$ and $\widehat{w_2f}$ are obtained using chosen squares.

To construct a 2-cell from (\ref{imageofcomp}) to (\ref{compofimages}) we consider the following diagram
containing the three chosen squares used above:
$$
\xymatrix{
E\ar@/_1ex/[ddr]_{\hat{w}_3}\ar@/^2ex/[drrr]^{\widehat{w_2f}}
\\
&C\ar[d]_{\bar{\bar{w}}_3}\ar[r]^{\bar{f}} & D\ar[d]_{\bar{w}_3}\ar[r]^{\bar{w}_2} & A_3\ar[d]^{w_3}
\\
& A_1\ar[r]_f&A_2\ar[r]_{w_2} & B_2.
}
$$
Using Condition {\bf CF2} we find a commutative square 
$$
\xymatrix{
F\ar[r]^{\bar{\hat{w}}_3}\ar[d]_{\hat{\bar{\bar{w}}}_3} & C\ar[d]^{\bar{\bar{w}}_3}
\\
E\ar[r]_{\hat{w}_3} & A_1}
$$
and using Condition {\bf CF3} we find an arrow $\xymatrix@1{G\ar[r]^{\tilde{w}_3}&F}$ in $W$
such that all parts of the following diagram commute:
$$
\xymatrix@R=1.8em{
E\ar@/_2ex/[3,2]_{\hat{w}_3}\ar@/^2ex/[2,4]^{\widehat{w_2f}}
\\
& G\ar[ul]|{\hat{\bar{\bar{w}}}_3\tilde{w}_3}\ar[dr]|{\bar{\hat{w}}_3\tilde{w}_3}
\\
&&C\ar[d]_{\bar{\bar{w}}_3}\ar[r]^{\bar{f}} & D\ar[d]_{\bar{w}_3}\ar[r]^{\bar{w}_2} & A_3\ar[d]^{w_3}
\\
&& A_1\ar[r]_f&A_2\ar[r]_{w_2} & B_2\rlap{\,.}
}
$$
So the required comparison 2-cell is given by
$$
\xymatrix@C=6em@R=1.8em{
&C\ar[dl]_{w_1\bar{\bar{w}}_3}\ar[dr]^{w_4g\bar{w}_2\bar{f}}
\\
B_1 & G\ar[u]_{\bar{\hat{w}}_3\tilde{w}_3}\ar[d]^{\hat{\bar{\bar{w}}}_3\tilde{w}_3} & B_3
\\
&E\ar[ul]^{w_1\hat{w}_3}\ar[ur]_{w_4g\widehat{w_2f}}&.
}$$

We will now construct a pseudo-inverse $\alpha\colon\bfC(W^{-1})\rightarrow\Bic(\bfC\{W\})$
for $\omega$. 

On objects, $\alpha(A)=(1_A)$. On arrows, $$\alpha(\xymatrix@1{A&C\ar[l]_w\ar[r]^f & B})=
( \xymatrix@1{A&C\ar[l]_w\ar[r]^f &B\ar[r]^{1_B} & B}),$$ and on 2-cells,
$$
\alpha\left(\raisebox{3.1em}{$\xymatrix{&\ar[dl]_w\ar[dr]^f\\&\ar[u]^v\ar[d]_{v'}&\\&\ar[ul]^{w'}\ar[ur]_{f'}}$}\right)=
\raisebox{3.5em}{$\xymatrix{\ar@{=}[dd]&\ar[l]_w\ar[r]^f & \ar@{=}[r]&\ar@{=}[dd]
\\ &\ar[r]^{fv}\ar[u]_{v}\ar[d]^{v'} &\ar@{=}[d]\ar@{=}[u] \\&\ar[l]^{w'}\ar[r]_{f'} & \ar@{=}[r] &\rlap{\,.}}$}
$$
We will show that $\alpha$ is a homomorphism of bicategories.
First, note that $\alpha$ preserves units strictly:
$\alpha(\xymatrix@1{&\ar[l]_{1_A}\ar[r]^{1_A}&})=(\xymatrix@1{&\ar[l]_{1_A}\ar[r]^{1_A} & \ar[r]^{1_A}&})$.
Composition is also preserved strictly as long as we choose the same chosen squares.
We will now show this.
So let $\xymatrix@1{A&\ar[l]_{w_1}S\ar[r]^{f_1} & B}$ and $\xymatrix@1{B&T\ar[l]_{w_2}\ar[r]^{f_2}& C}$
be a composable pair of arrows in $\bfC(W^{-1})$. Let 
\begin{equation}\label{chosen}
\xymatrix{
U\ar[r]^{\bar{f}_1}\ar[d]_{\bar{w}_2} & T\ar[d]^{w_2}
\\
S\ar[r]_{f_1} & B
}\end{equation}
be a chosen square. Then the composition in $\bfC(W^{-1})$ is 
$\xymatrix@1{A & U\ar[l]_{w_1\bar{w}_2}\ar[r]^{f_2\bar{f}_1} & C}$,
and $$\alpha(\xymatrix@1{A & U\ar[l]_{w_1\bar{w}_2}\ar[r]^{f_2\bar{f}_1} & C})=
(\xymatrix@1{A&U\ar[l]_{w_1\bar{w}_2}\ar[r]^{f_2\bar{f}_1}&C\ar@{=}[r]&C}).$$
To calculate the composition in $\Bic(\CW)$ of 
$$\alpha(\xymatrix@1{A&\ar[l]_{w_1}S\ar[r]^{f_1} & B})=
(\xymatrix@1{A&S\ar[l]_{w_1}\ar[r]^{f_1}&B\ar@{=}[r]&B})$$ and 
$$\alpha(\xymatrix@1{B&\ar[l]_{w_2}T\ar[r]^{f_1} & C})=
(\xymatrix@1{B&T\ar[l]_{w_2}\ar[r]^{f_2}&C\ar@{=}[r]&C}),$$ we note that we can use 
the same chosen square (\ref{chosen})
to construct a pair of horizontal arrows in $\CW$ that is composable there:
there is a double cell
$$
\xymatrix{
(\ar@{=}[dd]&)\ar[l]_{w_1\bar{w}_2}\ar[r]^{\bar{f}_1} & (\ar[r]^{w_2}&)\ar@{=}[dd]
\\
&\ar[d]_{\bar{w}_2}\ar@{=}[u]\ar[r]^{\bar{f}_1} &\ar[d]^{w_2}\ar@{=}[u]
\\
(&\ar[l]^{w_1})\ar[r]_{f_1} & (\ar@{=}[r]&)
}
$$
Composing the top row of this double cell with $\xymatrix@1{(&)\ar[l]_{w_2}\ar[r]^{f_2}&(\ar@{=}[r]&)}$
gives us
$$\xymatrix{(&)\ar[l]_{w_1\bar{w}_2}\ar[r]^{f_2\bar{f}_1} & (\ar@{=}[r]&)}.$$

Now it remains for us to calculate the composites $\alpha\circ\omega$ and $\omega\circ\alpha$.
First, a straightforward calculation shows that $\omega\alpha=\mbox{Id}_{\bfC(W^{-1})}$.
The other composition requires a bit more attention.

On objects, $\alpha\omega(\overline{(\stackrel{w}{\rightarrow})})=\overline{(\stackrel{1_B}{\rightarrow})}$
for $A\stackrel{w}{\rightarrow}B$, and $\overline{(\stackrel{w}{\rightarrow})}=\overline{(\stackrel{1_B}{\rightarrow})}$
in $\Bic(\CW)$. So  $\alpha\omega$ is the identity on objects.

On arrows, $\alpha\omega(\xymatrix@1{&\ar[l]_{w_1}\ar[r]^f&\ar[r]^{w_2}&})=
(\xymatrix@1{&\ar[l]_{w_1}\ar[r]^{w_2f}&\ar@{=}[r]&})$. And on 2-cells, 
\begin{eqnarray*}
\alpha\omega \left(\raisebox{3.5em}{$\xymatrix{\ar@{=}[dd]&\ar[l]_{w_1}\ar[r]^{f} &\ar[r]^{w_2}&\ar@{=}[dd]
\\
&\ar[u]_{u_1}\ar[r]^h\ar[d]^{v_1} & \ar[u]^{u_2}\ar[d]^{v_2}
\\
&\ar[l]_{w_1'}\ar[r]_{f'} & \ar[r]_{w_2'}& }$}\right)&=&
\alpha\left(\raisebox{3.1em}{$\xymatrix{&\ar[dl]_{w_1}\ar[dr]^{w_2f}
\\
&\ar[u]_{u_1}\ar[d]^{v_1}&
\\
&\ar[ul]^{w_1'}\ar[ur]_{w_2'f'}}$}\right)
\\
&=& \raisebox{3.5em}{$\xymatrix{\ar@{=}[dd]&\ar[l]_{w_1}\ar[r]^{w_2f} &\ar@{=}[r]&\ar@{=}[dd]
\\
&\ar[u]_{u_1}\ar[r]^{w_2fu_1}\ar[d]^{v_1} & \ar@{=}[u]\ar@{=}[d]
\\
&\ar[l]_{w_1'}\ar[r]_{w_2'f'} & \ar@{=}[r]&\rlap{\,.}
}$}
\end{eqnarray*}
There is an invertible icon $\zeta\colon\alpha\omega\Rightarrow\mbox{Id}_{\smBic(\bfC(W)^{-1})}$
such that the component of $\zeta$ at $\xymatrix@1{&\ar[l]_{w_1}\ar[r]^f&\ar[r]^{w_2}&}$ is
$$
\xymatrix{
\ar@{=}[dd]&\ar[l]_{w_1}\ar[r]^{w_2f} &\ar@{=}[r]&\ar@{=}[dd]
\\
&\ar@{=}[u]\ar[r]^f\ar@{=}[d] & \ar[u]^{w_2}\ar@{=}[d]
\\
&\ar[l]_{w_1}\ar[r]_{f} & \ar[r]_{w_2}& 
}
$$

The fact that the naturality squares of 2-cells commute follows from the fact that this bicategory is
locally posetal. 
\end{proof}

\begin{cor}\label{2-equiv}
 $\CW$ and $\Dbl(\bfC(W^{-1}))$ are equivalent in the 2-category $\WGDbl_{\mbox{\scriptsize\rm ps}}$.
\end{cor}

\begin{proof}
Apply $\Dbl$ to the equivalence of Theorem \ref{equiv} and then compose the resulting vertical equivalence of 
weakly globular double categories with the equivalence of $\CW$ with $\Dbl\Bic(\CW)$.
\end{proof}

\section{The Universal Property of $\CW$} \label{fracns}

The goal of this section is to describe the vertical and horizontal universal properties 
of the weakly globular double category $\CW$ we constructed in the previous section. 
First we derive what the universal property is that $\CW$ inherits from
$\bfC(W^{-1})$ by Corollary \ref{2-equiv}. This is a property in terms of 
pseudo-functors and vertical transformations.
The second universal property we give for $\CW$ is in terms of strict double functors 
and horizontal transformations. The two properties together determine $\CW$ 
up to horizontal and vertical equivalence.

\subsection{The vertical universal property of $\bfC(W^{-1})$} 
In this subsection we will describe the universal property $\CW$ inherits from $\bfC(W^{-1})$.
This universal property will be described in terms of pseudo-functors and vertical transformations between them.

The bicategory $\bfC(W^{-1})$ has its universal property as a bicategory of fractions as 
described in \cite{P1995}. 

\begin{thm} There is a homomorphism of bicategories $U\colon\bfC\rightarrow\bfC(W^{-1})$
such that composition with $U$ induces a biequivalence of bicategories
$$
\Hom(\bfC(W^{-1}),\calB)\simeq\Hom_W(\bfC,\calB).
$$
\end{thm}

In this result, the category on the left consists of homomorphisms of bicategories, lax natural transformations,
and modifications, and the category on the right consists of those homomorphisms, $W$-homomorphisms for short,
which send the elements of $W$ to equivalences in $\calB$. The natural transformations and modifications 
on the right are 
represented by $W$-homomorphisms into $\Cyl(\calB)$ and $\Cyl(\Cyl(\calB))$ respectively.

The proof of this result in \cite{P1995} shows implicitly that composition with $U$ induces 
also an equivalence of categories
$$
\Hom_{\mathsf{Bicat}_{\mathrm{icon}}}(\bfC(W^{-1}),\calB)\simeq\Hom_{\mathsf{Bicat}_{\mathrm{icon}},W}(\bfC,\calB),
$$
where we just consider homomorphisms of bicategories and icons (transformations with identity arrows as components)
as 2-cells.

The functor $\Dbl$ can be applied to this equivalence of categories
and we obtain the following
$$
\Hom_{\smWGDbl_{\ps}}(\Dbl(\bfC(W^{-1})),\Dbl\calB)\simeq\Hom_{\smWGDbl_{\ps},W}(\Dbl(\bfC),\Dbl\calB)
$$
where $\Hom_{\smWGDbl_{\ps}}$ refers to pseudo-functors of weakly globular double categories with 
vertical transformations of pseudo-functors,
and $\Hom_{\WGDbl_{\ps},W}$ consists of those pseudo-functors and vertical transformations
that are in the image of the 
$W$-homomorphisms and $W$-transformations under $\Dbl$.

Since pre-companions in weakly globular double categories correspond to equivalences in bicategories
as shown in Proposition \ref{quasi-comp} and Proposition \ref{comp-quasi}, we have the following result.

\begin{lma}\label{W-hom}
 Let $F\colon \bfC\rightarrow \calB$ be a $W$-homomorphism in $\Bicat$.
Then $\Dbl(F)$ sends a horizontal arrow in $\Dbl(\bfC)$ of the form 
$$
 \xymatrix{
A_0\ar[r]^{f_1} & \cdots \ar[r]^{f_{i_0}}&[A_{i_0}]\ar[r]^{f_{i_0+1}} & \cdots \ar[r]^{f_{i_1}}
      &[A_{i_1}]\ar[r]^{f_{i_1+1}}&\cdots\ar[r]^{f_{i_n}}&A_n
}
$$
with $f_{i_1}\circ\cdots\circ f_{i_0+1}\in W$,
to a pre-companion in $\Dbl(\calB)$.
\end{lma}

We also need a description of the vertical transformations in the image of the $W$-transformations
under $\Dbl$. We claim that they are the following class of vertical transformations.

\begin{dfn}\label{W-trafo}
{\rm A vertical transformation $\alpha\colon F\Rightarrow G\colon \Dbl(\bfC)\rightrightarrows \Dbl(\calB)$ 
between $W$-pseudo-functors  is a {\em $W$-transformation} if 
given the pre-companion structure for $Fw$ and $Gw$ with $w$ a horizontal arrow in $\Dbl(\bfC)$
that is related to $W$ as in Lemma \ref{W-hom},
there are double cells 
$$
\xymatrix{
\ar[r]^{l_{Fw}}\ar[d]|\bullet \ar@{}[dr]|{\alpha_w^l}& \ar[d]|\bullet
\ar@{}[drr]|{\mbox{and}} && \ar[r]^{r_{Fw}}\ar[d]|\bullet \ar@{}[dr]|{\alpha_w^r} & \ar[d]|\bullet
\\
\ar[r]_{l_{Gw}}&&& \ar[r]_{r_{Gw}}&
}$$
such that the following four equations hold:
$$
\xymatrix{
\ar[d]|\bullet\ar[r]^{l_{Fw}}\ar@{}[dr]|{\alpha_w^l}
	&\ar[d]|\bullet\ar@{}[dr]|{\alpha_w}\ar[r]^{Fw} 
	&\ar[d]|\bullet^{d_1\alpha_w} 
	&& \ar[d]|\bullet_{v^l_{Fw}} \ar[r]^{l_{Fw}}\ar@{}[drr]|{\chi^l_{Fw}}
	&\ar[r]^{Fw} 
	&\ar@{=}[d]
\\
\ar[r]_{l_{Gw}}\ar[d]|\bullet_{v_{Gw}^l}\ar@{}[drr]|{\chi^l_{Gw}} & \ar[r]_{Gw} & \ar@{=}[d]
&=& \ar@{=}[rr]\ar[d]|\bullet_{d_1\alpha_w} \ar@{}[drr]|{\mathrm{id}} && \ar[d]|\bullet^{d_1\alpha_w}
\\
\ar@{=}[rr] && && \ar@{=}[rr] &&
}$$
$$\xymatrix{
\ar@{=}[rr]\ar[d]|\bullet_{d_0\alpha_w^l}\ar@{}[drr]|{\mathrm{id}} &&\ar[d]|\bullet^{d_0\alpha_w^l} 
	&& \ar@{=}[d]\ar@{=}[rr]\ar@{}[drr]|{\psi^l_{Fw}} && \ar[d]|\bullet^{v^l_{Fw}}
\\
\ar@{=}[rr]\ar@{=}[d]\ar@{}[drr]|{\psi^l_{Gw}} && \ar[d]|\bullet^{v^l_{Gw}} 
	&=&\ar[d]|\bullet_{d_0\alpha^l_w} \ar@{}[dr]|{\alpha_w^l}\ar[r]_{l_{Fw}} 
	& \ar[d]|\bullet \ar[r]_{Fw}\ar@{}[dr]|{\alpha_w} & \ar[d]|\bullet 
\\
\ar[r]_{l_{Gw}}&\ar[r]_{Gw} & && \ar[r]_{l_{Gw}}&\ar[r]_{Gw} &
}$$
$$
\xymatrix{
\ar[r]^{Fw}\ar[d]|\bullet \ar@{}[dr]|{\alpha_w} & \ar[d]|\bullet \ar[r]^{r_{Fw}}\ar@{}[dr]|{\alpha^r_w}
	& \ar[d]|\bullet^{d_1\alpha^r_w} && \ar[r]^{Fw}\ar[d]|\bullet_{v^r_{Fw}}\ar@{}[drr]|{\chi^r_{Fw}} 
	& \ar[r]^{r_{Fw}} & \ar@{=}[d]
\\
\ar[d]|\bullet_{v^r_{Gw}}\ar@{}[drr]|{\chi^r_{Gw}} \ar[r]_{Gw} & \ar[r]_{r_{Gw}} & \ar@{=}[d] 
	&=& \ar[d]|\bullet_{d_1\alpha^r_w}\ar@{}[drr]|{\mathrm{id}} \ar@{=}[rr] && \ar[d]|\bullet^{d_1\alpha^r_w}
\\
\ar@{=}[rr] && && \ar@{=}[rr] &&
}
$$
$$\xymatrix{
\ar@{=}[d]\ar@{=}[rr]\ar@{}[drr]|{\psi^r_{Fw}} && \ar[d]|\bullet^{v_{Fw}^r} 
	&& \ar[d]|\bullet_{d_0\alpha_w}\ar@{}[drr]|{\mathrm{id}} \ar@{=}[rr]&& \ar[d]|\bullet^{d_0\alpha_w}
\\
\ar[d]|\bullet_{d_0\alpha_w}\ar[r]_{Fw}\ar@{}[dr]|{\alpha_w} 
	& \ar[d]|\bullet \ar[r]_{r_{Fw}}\ar@{}[dr]|{\alpha^r_w} & \ar[d]|\bullet 
	&=& \ar@{=}[d]\ar@{=}[rr]\ar@{}[drr]|{\psi^r_{Gw}} && \ar[d]|\bullet^{v^r_{Gw}}
\\
\ar[r]_{Gw}&\ar[r]_{r_{Gw}} & &&\ar[r]_{Gw}&\ar[r]_{r_{Gw}} &
}$$}
\end{dfn}

 \begin{notation}\rm
We will write $\Hom_{\smWGDbl_{\ps},W}(\Dbl\bfC,\bbD)$ for the category of $W$-pseudo-functors 
and $W$-transformations.
 \end{notation}

We compose $\Dbl(U)$ with the equivalence $\xymatrix@1{\Dbl(\bfC(W^{-1}))\ar[r]^-{\sim}&\CW}$
to get a pseudo functor $\tilde{U}\colon \Dbl(\bfC)\rightarrow\CW$.
Since every weakly globular double category  $\bbD$ is equivalent to 
$\Dbl(\Bic(\bbD))$ 
in the 2-category $\WGDbl_{\mathrm{ps}}$
we get the following universal property:

\begin{thm}
Composition with the pseudo-functor 
$\tilde{U}\colon \Dbl(\bfC)\rightarrow \Dbl(\bfC(W^{-1})),$
gives rise to an equivalence of categories,
$$
\Hom_{\smWGDbl_{\ps}}(\CW,\bbD)\simeq\Hom_{\smWGDbl_{\ps},W}(\Dbl\bfC,\bbD).
$$
\end{thm}

Since furthermore, $H(\bfC)$ is equivalent to $\Dbl(\bfC)$, we may compose $\tilde{U}$ 
with this equivalence and obtain 
$\calJ_{\bfC}\colon H(\bfC)\rightarrow \CW$.
We obtain the following corollary.

\begin{cor}
 Composition with $\calJ_{\bfC}\colon H\bfC\rightarrow \Dbl(\bfC(W^{-1}))$ gives 
rise to an equivalence of categories,
$$
\Hom_{\smWGDbl_{\ps}}(\CW,\bbD)\simeq\Hom_{\smWGDbl_{\ps},W}(H\bfC,\bbD).
$$
\end{cor}

\begin{rmks}\rm
 \begin{enumerate}
 \item This characterization of $\CW$ determines this weakly globular double category up to a vertical equivalence.
\item
It is straightforward to check that the composition $\calJ_{\bfC}\colon H\bfC\rightarrow\CW$ 
used in the last corollary
is the obvious inclusion functor, sending and object of $C$ to the object in $\CW$ represented 
by its identity arrow. 
It is clear that this is a strict functor of weakly globular double categories. 
 \end{enumerate}
\end{rmks}

In the rest of this section we will discuss 
the universal property of the functor $\calJ_{\bfC}\colon h\bfC\rightarrow\CW$  
in terms of 2-categories of strict functors and horizontal transformations.
This will then characterize $\CW$ up to a horizontal equivalence.
This characterization will be in terms of companions rather than pre-companions, 
since strict functors do preserve companions.

\subsection{Companions and conjoints in $\CW$}
As we saw in an Section \ref{compquasi}, companions in weakly globular double categories are closely related to 
quasi units in bicategories. Requiring in the universal property  that the elements of $W$ themselves become 
companions is too strong, 
but requiring that they become companions after composition with a horizontal isomorphism 
turns out to be just right for a universal property with respect to strict functors and horizontal transformations.

For this part we will make the added assumption on the class $W$ that it satisfies the 2 out of 3 property:
if $f,g,h$ are arrows in $\bfC$ such that $gf=h$ and two of $f,g$ and $h$ are in $W$, so is the third.

The following lemma characterizes the horizontal and vertical arrows in $\bfC\W$ that have 
companions or conjoints.

\begin{lma}\label{compchar}
\begin{enumerate}
 \item 
A vertical arrow in $\CW$ has a horizontal companion if and only if it is of the form
\begin{equation}\label{vertcomp}
\xymatrix@R=1.5em{
(A_1\ar[r]^{wu} & B)\ar@{=}[dd]
\\
C\ar@{=}[u]\ar[d]_u 
\\
(A_2\ar[r]_{w}&B)
}
\end{equation}
\item
A vertical arrow in $\CW$ has a horizontal conjoint if and only if it is of the form
\begin{equation}\label{vertconj}
\xymatrix@R=1.5em{
(A_1\ar[r]^{w} & B)\ar@{=}[dd]
\\
A_2\ar@{=}[d]\ar[u]^u 
\\
(A_2\ar[r]_{wu}&B)
}
\end{equation}
\item
A horizontal arrow in $\CW$ has a vertical companion and conjoint if and only if
it is of the form
$$
 \xymatrix{
(B&\ar[l]_{wu} A_2)\ar[r]^u&(A_1\ar[r]^{w} & B)\rlap{\,.}
}
$$ 
\end{enumerate}
\end{lma}

\begin{proof}
To prove the first part, the companion for (\ref{vertcomp}) is
the horizontal arrow
$$
\xymatrix@1{(B&\ar[l]_{wu} A_1)\ar[r]^{u}&(A_2\ar[r]^w&B)}
$$
with binding cells
\begin{equation}\label{bindingcells}
\xymatrix@R=1.5em{
&(B\ar@{=}[dd] & A_1)\ar[l]_{wu}\ar@{=}[r] & (A_1\ar[r]^{wu}&B)\ar@{=}[dd]
\\
\ar@{}[r]|(0.2){\textstyle \psi_{u,w}=}&&A_1\ar@{=}[d]\ar@{=}[u] \ar@{=}[r] & A_1\ar@{=}[u]\ar[d]_{u}&
\\
&(B&\ar[l]^{wu}A_1)\ar[r]_{u} & (A_2\ar[r]_w & B)
\\
&(B\ar@{=}[dd] & A_1)\ar[l]_{wu}\ar[r]^u & (A_2\ar[r]^w & B)\ar@{=}[dd]
\\
\ar@{}[r]|(0.2){\textstyle \chi_{u,w}=}& & A_1\ar[d]_u\ar@{=}[u] \ar[r]^u & A_2\ar@{=}[u]\ar@{=}[d]&
\\
&(B&\ar[l]^{w}A_2)\ar@{=}[r] & (A_2\ar[r]_w & B)
}
\end{equation}
We leave it up to the reader to check that these cells satisfy the equations (\ref{compcondns}).

Conversely, suppose that a vertical arrow
\begin{equation}\label{vertarrow}
\xymatrix@R=1.5em{
A_1\ar[r]^{w_1} & B\ar@{=}[dd]
\\
C\ar[u]^{u_1}\ar[d]_{u_2} 
\\
A_2\ar[r]_{w_2}&B
}
\end{equation}
has a companion $\xymatrix@1{(B&\ar[l]_{w_1}A_1)\ar[r]^{f}&(A_2\ar[r]^{w_2}&B)}$. Then there is 
a binding cell
$$
\xymatrix@R=1.5em{
(B\ar@{=}[dd] & A_1)\ar[l]_{w_1}\ar[r]^f & (A_2\ar[r]^{w_2} & B)\ar@{=}[dd]
\\
 & C\ar[d]_{u_2}\ar[u]^{u_1} \ar@{}[r]|{(\chi)} & A_2\ar@{=}[u]\ar@{=}[d]&
\\
(B&\ar[l]^{w_2}A_2)\ar@{=}[r] & (A_2\ar[r]_{w_2} & B)
}
$$
So there is an arrow $w\in W$ such that the component $\chi_{u_1w,u_2w,1_{A_2},1_{A_2}}$
exists. So $fu_1w= \chi_{u_1w,u_2w,1_{A_2},1_{A_2}}=u_2w$. It follows that 
(\ref{vertarrow}) is equivalent to
$$
\xymatrix@R=1.5em{
A_1\ar[r]^{w_1} & B\ar@{=}[dd]
\\
A_1\ar@{=}[u]\ar[d]_{f} 
\\
A_2\ar[r]_{w_2}&B
}
$$
This completes the proof of the first part.

The second part follows from the first part, together with Remark \ref{compconj}:
a vertical arrow in $\CW$ has a horizontal conjoint if and only if its inverse has this
horizontal arrow as companion.

The last part follows from the proofs of the previous two parts.
\end{proof}

It is easy to see that the vertical companions and conjoints generate the category of vertical arrows:
$$
\xymatrix@R=1.5em{
\ar@{=}[dd]B&\ar[l]_{w_1} A_1
\\
&C\ar[u]_{u_1}\ar[d]^{u_2}
\\
B&\ar[l]^{w_2} A_2
}$$
is the vertical composition of 
$$
\xymatrix@R=1.5em{
\ar@{=}[dd]B&\ar[l]_{w_1} A_1 && \ar@{=}[dd]B&\ar[l]_{w_2u_2} C 
\\
&C\ar[u]_{u_1}\ar@{=}[d]&\mbox{and} & &C\ar@{=}[u]\ar[d]^{u_2}
\\
B&\ar[l]^{w_1u_1} C&& B&\ar[l]^{w_2} A_2
}$$
which are composable since $w_1u_1=w_2u_2$.

The following lemma shows that the companion binding cells and 
their vertical inverses (the conjoint binding cells)
generate all the cells of $\CW$.

\begin{lma}\label{dblcellfac}
Each double cell in $\CW$ can be written as a pasting diagram of companion binding 
cells as in (\ref{bindingcells}) and their
vertical inverses.
\end{lma}

\begin{proof}
Let 
\begin{equation}\label{generic}
\xymatrix@R=1.5em{
(B\ar@{=}[dd] & \ar[l]_{w_1}A_1)\ar[r]^{f_1}&(A_1'\ar[r]^{w_1'} & B')\ar@{=}[dd]
\\
&C\ar[u]_{u_1}\ar[r]^{\xi}\ar[d]^{u_2}&C'\ar[u]_{u_1'}\ar[d]^{u_2'} &
\\
(B & A_2)\ar[l]^{w_2}\ar[r]_{f_2}&( A_2'\ar[r]_{w_2'}& B')
}
\end{equation}
be a component of a double cell in $\CW$.
This cell can be written as a pasting of the following 
array of double cells by first composing the double cells in each row
horizontally and then composing the resulting double cells vertically.
$$
\xymatrix@R=1.5em{
(B\ar@{=}[dd] & \ar[l]_{w_1}A_1)\ar@{=}[r]&(A_1\ar[r]^{w_1} & B)\ar@{=}[dd]
	&(B\ar@{=}[dd] & \ar[l]_{w_1}A_1)\ar[r]^{f_1}&(A_1'\ar[r]^{w_1'} & B')\ar@{=}[dd]
\\
&C\ar[u]_{u_1}\ar[r]^{u_1}\ar@{=}[d]&A_1\ar@{=}[u]\ar@{=}[d]&
	&&A_1\ar@{=}[u]\ar@[=][d]\ar[r]^{f_1}&A_1'\ar@{=}[u]\ar@{=}[d] &
\\
(B & C)\ar[l]^{w_1u_1}\ar[r]_{u_1}& (A_1'\ar[r]_{w_1}& B)
	& (B & A_1)\ar[l]^{w_1}\ar[r]_{f_1}& (A_1'\ar[r]_{w_1'}& B')
}
$$
$$\xymatrix@R=1.5em{
(B\ar@{=}[dd] & \ar[l]_{w_1u_1} C)\ar[r]^{\xi}&(C'\ar[r]^{w_1'u_1'} & B')\ar@{=}[dd]
	& (B'\ar@{=}[dd] & \ar[l]_{w_1'u_1'}C')\ar[r]^{u_1'}&(A_1'\ar[r]^{w_1'} & B')\ar@{=}[dd]
\\
&C\ar@{=}[u]\ar[r]^{\xi}\ar@{=}[d]&C'\ar@{=}[u]\ar@{=}[d]&
	&&C'\ar@{=}[u]\ar@{=}[r]\ar@{=}[d]&C'\ar[u]_{u_1'}\ar@{=}[d] &
\\
(B & C)\ar[l]^{w_1u_1}\ar[r]_{\xi}& (C'\ar[r]_{w_1'u_1'}& B')
	& (B' & C')\ar[l]^{w_1'u_1'}\ar@{=}[r]& (C'\ar[r]_{w_1'u_1'}& B')
}
$$
$$\xymatrix@R=1.5em{
(B\ar@{=}[dd] & \ar[l]_{w_1u_1}C)\ar[r]^{\xi}&(C'\ar[r]^{w_2'u_2'} & B')\ar@{=}[dd]
	& (B'\ar@{=}[dd] & \ar[l]_{w_2'u_2'}C')\ar@{=}[r]&(C'\ar[r]^{w_2'u_2'} & B')\ar@{=}[dd]
\\
&C\ar@{=}[u]\ar[r]^{\xi}\ar@{=}[d]&C'\ar@{=}[u]\ar@{=}[d] &
	&&C'\ar@{=}[u]\ar@{=}[r]\ar@{=}[d]&C'\ar@{=}[u]\ar[d]^{u_2'} &
\\
(B & C)\ar[l]^{w_2u_2}\ar[r]_{\xi}& (C'\ar[r]_{w_2'u_2'}& B')
	& (B' & C')\ar[l]^{w_2'u_2'}\ar[r]_{u_2'}& (A_2'\ar[r]_{w_2'}& B')
}
$$
$$\xymatrix@R=1.5em{
(B\ar@{=}[dd] & \ar[l]_{w_2u_2}C)\ar[r]^{u_2}&(A_2\ar[r]^{w_2} & B)\ar@{=}[dd]
	& (B\ar@{=}[dd] & \ar[l]_{w_2}A_2)\ar[r]^{f_2}&(A_2'\ar[r]^{w_2'} & B')\ar@{=}[dd]
\\
&C\ar@{=}[u]\ar[r]^{u_2}\ar[d]^{u_2}&A_2\ar@{=}[u]\ar@{=}[d] &
	&& A_2\ar@{=}[u]\ar[r]^{f_2}\ar@{=}[d]&A_2'\ar@{=}[u]\ar@{=}[d] &
\\
(B & A_2)\ar[l]^{w_2} \ar@{=}[r]& (A_2\ar[r]_{w_2}& B)
	& (B & A_2)\ar[l]^{w_2} \ar[r]_{f_2}& (A_2'\ar[r]_{w_2'}& B')\rlap{\,.}
}
$$ 
Each one of these cells is either a horizontal or vertical identity cell or a companion binding cell or the 
inverse of a companion binding cell.
\end{proof}

\subsection{$W$-friendly functors and transformations}
As we have seen in 
Lemma \ref{compchar} and Lemma \ref{dblcellfac},  the strict inclusion functor $\calJ_{\bfC}\colon H\bfC\rightarrow \CW$
adds companions and conjoints to $H\bfC$.
We want to make precise in what sense it does so freely.

The first question is for which arrows it adds the companions and conjoints.
It is tempting to think that it does this for the horizontal arrows in the image of $W$.
However, those would be the arrows of the form
$$\xymatrix@1{(1_A)\ar[r]^w&(1_B)}$$
and the arrows that obtain companions and conjoints are of the form
$$
\xymatrix@1{(uw)\ar[r]^w&(u)\rlap{\,.}}
$$
In particular, there are companions and conjoints for arrows
of the form 
\begin{equation}\label{w}
\xymatrix@1{(w)\ar[r]^w&(1_B)\rlap{\,,}}
\end{equation}
but not for
$$\xymatrix@1{(1_A)\ar[r]^w&(1_B)}$$
unless $w=1_B$.
The arrow in (\ref{w}) can be factored as
$$
\xymatrix@1{(B&A\ar[l]_w)\ar[r]^{1_A}&(A\ar[r]^{1_A}&A)\ar[r]^w &(B\ar[r]^{1_B}&B)}
$$
We will denote the first arrow in this factorization by $\varphi_w\colon (w)\rightarrow(1_A)$.
Note that this is a horizontal isomorphism in $\CW$.
Our first step will be to show that the family $\{\varphi_w\}$ for $w\in W$ form the components of an invertible 
natural transformation.
To express this property we need the following comma category derived from $W$.

\begin{dfn}\label{nabla}
{\rm Let $\bfC$ be a category with a class $W$ of arrows.
We define $\nabla W$ to be the category with arrows of $W$ as objects and an 
arrow from $\xymatrix@1{A\ar[r]^w&B}$
to $\xymatrix@1{A'\ar[r]^{w'}&B}$ is given by a commutative triangle of arrows in $\bfC$:
$$
\xymatrix{
A\ar[dr]_w\ar[rr]^v&&A'\ar[dl]^{w'}
\\
&B
}
$$
We denote this arrow by $(v,w')\colon w\rightarrow w'$.
There is a functor $D_0\colon \nabla W\rightarrow \bfC$ defined by $D_0(w)=d_0(w)$ (the domain)
on objects and $D_0(v,w')=v$.}
\end{dfn}

\begin{rmk}
 {\rm Since $W$ satisfies the 2 out of 3 property, the commutativity of the triangle in Definition
\ref{nabla} implies that $v\in W$. Furthermore, by condition {\bf CF1}, $W$ forms a subcategory of $\bfC$ 
which contains all objects of $\bfC$. So we can view $\nabla W$ as the comma category $(1_W\downarrow J_W)$,
where $J_W\colon C_0^d\rightarrow W$ is the inclusion functor of the discrete category on the objects of $\bfC$
into the category $W$. In this notation, $D_0$ is just the first projection functor 
$(1_W\downarrow J_W)\rightarrow W$ followed by the inclusion into $\bfC$.}
\end{rmk}

Let $\Phi\colon\nabla W\rightarrow \Comp(\CW)$
be the functor that sends an arrow
$$
\xymatrix@R=1.5em{
A\ar[dr]_{wu}\ar[rr]^u&&A'\ar[dl]^{w}
\\
&B
}
$$
to the horizontal companion arrow $\xymatrix@1{(B&A\ar[l]_{wu})\ar[r]^u&(A'\ar[r]^{w}&B)}$
with the vertical companion and binding cells in $\CW$ as indicated in the proof of Lemma \ref{compchar}.
Now we can view $\varphi$ as a natural transformation in the following diagram
$$
\xymatrix{
\nabla W\ar[d]_{D_0}\ar[r]^-\Phi\ar@{}[dr]|{\stackrel{\varphi}{\Rightarrow}} & \Comp(\CW)\ar[d]^{h_-}
\\
\bfC\ar[r]_{h\JC}&h\CW\rlap{\,.}
}
$$
This is almost enough to describe $\JC$ as a $W$-friendly functor.
The problem is that the information given thus far is not enough to describe 
$W$-friendly transformations between  $W$-friendly functors.
Such transformations need to have components that interact well with the companion binding cells.
In order to ensure this, $\Phi$ needs to be viewed as a double functor $V\nabla W\rightarrow \bbComp(\CW)$.
This leads us to the following definition.

\begin{dfn}
{\rm A {\em $W$-friendly structure} $(\Gamma,\gamma)$ for a functor $G\colon H\bfC\rightarrow \bbD$ consists of 
a functor $\Gamma\colon V\nabla W\rightarrow \bbComp(\bbD)$ together with 
an invertible natural transformation $\gamma$,
$$
\xymatrix{
\nabla W\ar[d]_{D_0}\ar[r]^-{v\Gamma}\ar@{}[dr]|{\stackrel{\,\,\sim}{\Leftarrow}}_(.57)\gamma 
  & v\bbComp(\bbD)\ar[d]^{h_-}
\\
\bfC\ar[r]_{hG} & h\bbD,}
$$
where $h_-$ is the identity on objects and takes the horizontal component of each arrow (companion pair).
We will also refer to $G$ (and to the pair $(G,\Gamma)$) as a {\em $W$-friendly functor}.}
\end{dfn}

\begin{prop}
There is a canonical $W$-friendly structure $$(\Phi\colon V\nabla W\rightarrow\bbComp(\CW), \varphi)$$ 
for the functor $\calJ_{\bfC}\colon H\bfC\rightarrow\CW$.
\end{prop}

\begin{proof}
On objects, $\Phi(w)=(w)$. On horizontal arrows, $\Phi$ is the identity.
On vertical arrows,
$\Phi((u,w))=(h_{(u,w)},v_{(u,w)},\psi_{(u,w)},\chi_{(u,w)})$,
where 
\begin{eqnarray*}
h_{(u,w)}&=&\xymatrix@1{(&)\ar[l]_{wu}\ar[r]^u&(\ar[r]^w&)}\\
v_{(u,w)}&=&\xymatrix@R=1.4em{(\ar[r]^{wu}\ar@{=}[d] &)\ar@{=}[dd]\\\ar[d]_u\\(\ar[r]_w&)}
\end{eqnarray*}
and the binding cells are
$$
\xymatrix@R=1.4em{
(\ar@{=}[dd]&\ar[l]_{wu})\ar@{=}[r] & (\ar[r]^{wu}& \ar@{=}[dd])
      && (\ar@{=}[dd]&\ar[l]_{wu})\ar[r]^u&(\ar[r]^w&\ar@{=}[dd])
\\
&\ar@{=}[u]\ar@{=}[d]\ar@{=}[r] & \ar@{=}[u]\ar[d]^u &&\mbox{and}
      & & \ar@{=}[u]\ar[r]^u\ar[d]^u & \ar@{=}[u]\ar@{=}[d]
\\
(&\ar[l]^{wu})\ar[r]_u&(\ar[r]_w&) && (&\ar[l]^w)\ar@{=}[r] & (\ar[r]_w&)\rlap{\,.}
}
$$
Furthermore, the invertible natural transformation $\varphi$
in 
$$
\xymatrix@R=1.9em{
\nabla W\ar[d]_{D_0}\ar[r]^-{v\Phi} \ar@{}[dr]|{\stackrel{\,\,\sim}{\Leftarrow}}_(.57)\varphi 
  & v\bbComp(\CW)\ar[d]^{h_-}
\\
\bfC\ar[r]_{hF} & h\CW}
$$
has components $\varphi_w$,
$$
\xymatrix{(&\ar[l]_w)\ar@{=}[r] &(\ar@{=}[r]&).}
$$
\end{proof}

Let $L\colon\CW\rightarrow \bbD$ be a functor of double categories.
Composition with $\calJ_{\bfC}\colon H\bfC\rightarrow \CW$ gives rise to a
functor $L\calJ_{\bfC}\colon H\bfC\rightarrow \bbD$ with the $W$-friendly structure
$(\bbComp(L)\circ\Phi,\lambda)$. Here, 
$\bbComp(L)\circ\Phi\colon V\nabla W\rightarrow \bbComp(\bbD)$
sends $w$ to $L(w)$ and a vertical arrow $(u,v)\colon\xymatrix@1{vu\ar[r]|\bullet^u&v}$
to 
$$\left(L\left(\xymatrix{&\ar[l]_{vu}\ar[r]^u & \ar[r]^v&}\right),
L\left(\raisebox{1.95em}{$\xymatrix@R=1.4em{\ar@{=}[dd]&\ar[l]_{vu}\\&\ar@{=}[u]\ar[d]^u \\ 
&\ar[l]^{v}}$}\right),
L\left(\raisebox{1.95em}{$\xymatrix@R=1.4em{\ar@{=}[dd]& \ar[l]_{vu}\ar@{=}[r]& \ar[r]^{vu}&\ar@{=}[dd]\\
& \ar@{=}[u]\ar@{=}[d]
		&\ar@{=}[u]\ar[d]^u \\ &\ar[l]^{vu}\ar[r]_u&\ar[r]_v&}$}\right),\right.
		$$
		$$\left.
L\left(\raisebox{1.95em}{$\xymatrix@R=1.4em{\ar@{=}[dd]& \ar[l]_{vu}\ar[r]^u
		& \ar[r]^{v}&\ar@{=}[dd]\\& \ar@{=}[u]\ar[d]^u
		&\ar@{=}[u]\ar@{=}[d] \\ &\ar[l]^{v}\ar@{=}[r]&\ar[r]_v&}$}\right)\right).$$
Furthermore, the natural transformation $\lambda$ has components 
$\lambda_w=L\left(\xymatrix@1@C=1.5em{&\ar[l]_w\ar@{=}[r]&\ar@{=}[r] &}\right)$.

The appropiate horizontal transformations between $W$-friendly functors 
are described in the following definition.

\begin{dfn}
{\rm For $W$-friendly functors
$$(G,\Gamma,\gamma), (L,\Lambda,\lambda)\colon 
H\bfC\rightrightarrows\bbD, V\nabla W\rightrightarrows \bbComp(\bbD),$$
a {\em $W$-friendly horizontal transformation} is a pair 
$(a\colon G\Rightarrow L,\alpha\colon\Gamma\Rightarrow\Lambda)$
of horizontal transformations
such that the following square commutes: 
\begin{equation}\label{W-friendly}
\xymatrix@R=1.7em{
h_-\circ v\Gamma\ar[d]_\gamma\ar[r]^{h_-\circ v\alpha}& h_-\circ v\Lambda\ar[d]^{v\lambda}
\\
hG\circ D_0\ar[r]_{ha\circ D_0} & hL\circ D_0.
}
\end{equation}}
\end{dfn}

\subsection{The horizontal universal property of $\bfC(W^{-1})$} 
 \begin{thm}\label{main}
 Composition with $\calJ_{\bfC}\colon H\bfC\rightarrow \CW$ induces an equivalence of 
 categories $$\Hom_{\smWGDbl_h}(\CW,\bbD)\simeq\Hom_{\smWGDbl_h,W}(H\bfC,\bbD),$$ where $\Hom_{\smWGDbl_h,W}(H\bfC,\bbD)$
 is the category of $W$-friendly functors and $W$-friendly horizontal transformations.
 \end{thm}
 
 \begin{proof}
Composition with $\calJ_{\bfC}$ 
sends a functor $L\colon\CW\rightarrow \bbD$ of weakly globular double categories
to the triple $(L\calJ_{\bfC}, L\circ\Phi,L\varphi)$, i.e., the functor $L\calJ_{\bfC}$ with $W$-friendly 
structure $(L\Phi,L\varphi)$, as described above.

 To show that composition with $\calJ_{\bfC}$ is essentially surjective on objects,
 we show how to lift a functor $G\colon H\bfC\rightarrow \bbD$ with 
 a $W$-friendly structure $(\Gamma,\gamma)$ to a functor $\tilde{G}\colon\CW\rightarrow \bbD$.
 
 Define $\tilde G$ on objects by $\tilde{G}(\xymatrix@1{A\ar[r]^w&B})=\Gamma(w)$.
 On vertical arrows, $\tilde G$ is defined by
 $$\tilde{G}\left(\raisebox{2.65em}{$\xymatrix@R=1.5em{(\ar[r]^{w_1}&)\ar@{=}[dd]
\\
\ar[u]^{u_1}\ar[d]_{u_2} &
\\
(\ar[r]_{w_2}&)}$}\right)=v_{\Gamma(u_2,w_2)}\cdot(v_{\Gamma(u_1,w_1)})^{-1}
 $$
 and on horizontal arrows, $\tilde G$ is defined by
 $\tilde{G}\left(\xymatrix@1{(&\ar[l]_w)\ar[r]^f&(\ar[r]^v&)}\right)=\gamma_v^{-1}G(f)\gamma_w.$
 To define $\tilde{G}$ on double cells, we use the factorization of a generic double
 cell in $\CW$ given in the proof of  Lemma \ref{dblcellfac}. 
The result of applying $\tilde{G}$ to (\ref{generic}) is given in Figure \ref{imagedblcell}.
\begin{figure}[ht]
 $$\xymatrix@R=1.4em{
 \Gamma(w_1)\ar[d]|\bullet^{v^{-1}_{\Gamma(u_1,w_1)}}\ar@{=}[rrr]\ar@{}[drrr]|{\chi^{-1}_{\Gamma(u_1,w_1)}} 
		&&&\Gamma(w_1)\ar[r]^{\gamma_{w_1}} \ar@{=}[d]& GA_1 \ar[r]^{Gf_1} 
		& GA_1'\ar[r]^{\gamma_{w_1'}^{-1}} 
		& \Gamma(w_1') \ar@{=}[d]
\\
\Gamma(w_1u_1) \ar[rrr]_{h_{\Gamma(u_1,w_1)}}  \ar@{=}[d] 
		&&& \Gamma(w_1) \ar[r]^{\gamma_{w_1}} \ar@{=}[d]
		& GA_1 \ar[r]^{Gf_1} & GA_1'\ar[r]^{\gamma_{w_1'}^{-1}} 
		& \Gamma(w_1') \ar@{=}[d]
\\
\Gamma(w_1u_1) \ar[r]^-{\gamma_{w_1u_1}} \ar@{=}[d] & GC \ar@{=}[d] \ar[r]^{Gu_1} 
		&GA_1\ar[r]^{\gamma_{w_1}^{-1}}
		& \Gamma(w_1) \ar[r]^{\gamma_{w_1}} &  GA_1 \ar[r]^{Gf_1} 
		& GA_1' \ar@{=}[d] \ar[r]^{\gamma_{w_1'}^{-1}} 
		& \Gamma(w_1') \ar@{=}[d]
\\
\Gamma(w_1u_1) \ar[r]^-{\gamma_{w_1u_1}} \ar@{=}[5,0] & GC \ar@{=}[5,0] \ar[r]^{G\xi} 
		& GC' \ar@{=}[dd] \ar[r]^-{\gamma^{-1}_{w_1'u_1'}}
		& \Gamma(w_1'u_1') \ar@{=}[d] \ar[r]^-{\gamma_{w_1'u_1'}} 
		& GC' \ar[r]^{Gu_1'} & GA_1' \ar[r]^{\gamma_{w_1'}^{-1}} 
		& \Gamma(w_1') \ar@{=}[d]
\\
&&& \Gamma(w_1'u_1')\ar@{=}[d] \ar[rrr]^{h_{\Gamma(u_1',w_1')}}\ar@{}[drrr]|{\psi^{-1}_{\Gamma(w_1')}}
		&&& \Gamma(w_1') \ar[d]|\bullet^{v^{-1}_{\Gamma(w_1')}}
\\
&& GC'\ar@{=}[d] \ar[r]^-{\gamma_{w_1'u_1'}^{-1}} &  \Gamma(w_1'u_1') \ar@{=}[d] \ar@{=}[rrr] 
		&&& \Gamma(w_1'u_1') \ar@{=}[d]
\\
&& GC'\ar@{=}[dd] \ar[r]^-{\gamma_{w_2'u_2'}^{-1}} 
		&  \Gamma(w_2'u_2') \ar@{=}[d] \ar@{=}[rrr] \ar@{}[drrr]|{\psi_{\Gamma(u_2',w_2')}} 
		&&& \Gamma(w_2'u_2') \ar[d]|\bullet^{v_{\Gamma(u_2',w_2')}}
\\
&&& \Gamma(w_2'u_2')\ar@{=}[d] \ar[rrr]_{h_{\Gamma(u_2',w_2')}} &&& \Gamma(w_2')\ar@{=}[d]
\\
\Gamma(w_2u_2) \ar[r]^-{\gamma_{w_2u_2}} \ar@{=}[d]& GC \ar[r]^{G\xi} \ar@{=}[d] 
		& GC'\ar@{=}[d] \ar[r]^-{\gamma_{w_2'u_2'}^{-1}} 
		& \Gamma(w_2'u_2') \ar[r]^-{\gamma_{w_2'u_2'}} & GC' \ar[r]^{Gu_2'} 
		&GA_2'\ar[r]^{\gamma_{w_2'}^{-1}} \ar@{=}[d] &\Gamma(w_2') \ar@{=}[d]
\\
\Gamma(w_2u_2) \ar[r]^-{\gamma_{w_2u_2}} \ar[d]|\bullet^{v_{\Gamma(u_2,w_2)}} 
			  \ar@{}[drrr]|{\chi_{\Gamma(u_2,w_2)}} 
		& GC \ar[r]^{Gu_2} & GA_2 \ar[r]^{\gamma_{w_2}^{-1}} 
		& \Gamma(w_2) \ar@{=}[d] \ar[r]^{\gamma_{w_2}} 
		& GA_2\ar[r]^{Gf_2} & GA_2' \ar[r]^{\gamma_{w_2'}^{-1}} 
		& \Gamma(w_2')\ar@{=}[d]
\\
\Gamma(w_2)\ar@{=}[rrr] &&& \Gamma(w_2) \ar[r]^{\gamma_{w_2}} & GA_2\ar[r]^{Gf_2} 
		& GA_2' \ar[r]^{\gamma_{w_2'}^{-1}} & \Gamma(w_2')
 }$$
 \caption{The image of a generic double cell in $\CW$ under $\tilde{G}$.}\label{imagedblcell}
\end{figure}
 It is not hard to see that there is an invertible horizontal transformation
 $\bar\gamma\colon \tilde{G}\circ \JC\Rightarrow G$ with components $\bar{\gamma}_{A}=\gamma_{1_A}$.
 Furthermore, it is straightforward to check that $\tilde{G}\Phi=\Gamma$, and $(\bar\gamma,\mbox{id})$ 
is an invertible 
$W$-friendly transformation from $(\tilde{G}\JC,\tilde{G}\Phi,\tilde{G}\varphi)$ to $(G,\Gamma,\gamma)$.
 
We now want to show that composition with $\JC$ is fully faithful on arrows, i.e., horizontal transformations.
So we will show that horizontal transformations $$b\colon G\Rightarrow L\colon \CW\rightrightarrows\bbD$$
are in one-to-one correspondence with
$W$-friendly transformations $$(b\JC,\beta)\colon (G\JC,G\Phi,G\varphi)\Rightarrow(L\JC,L\Phi,L\varphi)$$ 
 where $\beta\colon G\Phi\Rightarrow L\Phi$ is defined by $\beta_w=b_w$.
 The condition that the square in (\ref{W-friendly}) commutes for $(b\JC,\beta)$
 is equivalent to the following square of horizontal arrows commuting for each $w\in W$:
 $$
 \xymatrix@C=3.5em{
 G(w)\ar[d]_{G(\stackrel{v}{\longleftarrow}\Longrightarrow\Longrightarrow)}\ar[r]^{b_v}
	&L(v)\ar[d]^{G(\stackrel{v}{\longleftarrow}\Longrightarrow\Longrightarrow)}
 \\
 G(\mbox{id}_A)\ar[r]_{b_{\mbox{\scriptsize id}_A}} &L(\mbox{id}_A)
 }
 $$
 and the commutativity of this diagram follows immediately from the vertical naturality of $b$. 
 
Conversely, given a $W$-friendly horizontal transformation 
$(a\colon G\JC\Rightarrow L\calJ_{\bfC},\alpha\colon G\Phi\Rightarrow L\Phi)$,
we get $\tilde{a}\colon G\Rightarrow L$ with components $\tilde{a}_w=\alpha_w$ and
$$
\tilde{a}\left(\raisebox{3.2em}{$\xymatrix{(\ar[r]^{wu}&)\ar@{=}[dd]
\\
\ar@{=}[u]\ar[d]_{u} &
\\
(\ar[r]_{w}&)}$}\right) = \raisebox{4.2em}{$\xymatrix@C=3.9em@R=1.25em{G(wu)\ar[rr]|{h_{\alpha(wu)}} \ar@{=}[d] 
		&& L(wu) \ar@{=}[d]\ar@{=}[r] \ar@{}[dr]|{L(\psi_{\Phi(u,v)})} & L(wu)\ar[d]|\bullet^{L(v_{\Phi(u,w)})}
\\
G(wu)  \ar@{=}[d]\ar@{=}[r] & G(wu)\ar@{=}[d]  \ar[r]|{h_{\alpha(wu)}} & L(wu)\ar[r]|{L(h_{\Phi(u,w)})}
	& L(w)\ar@{=}[d]
\\
G(wu) \ar@{=}[r] \ar[d]|\bullet_{G(v_{\Phi(u,v)})} 
	& G(wu) \ar[d]|\bullet_{G(v_{\Phi(u,w)})} \ar@{}[dr]|{G(\chi_{\Phi(u,w)})} \ar[r]|{G(h_{\Phi(u,w)})} 
	& G(w)\ar@{=}[d]\ar[r]|{h_{\alpha(w)}}&L(w)\ar@{=}[d]\\
G(w)\ar@{=}[r] & G(w)\ar@{=}[r] & G(w)\ar[r]|{h_{\alpha(w)}} &L(w)}$}
$$
Note that it is sufficient to define $\tilde{a}$ on these particular vertical arrows, since all others are 
generated by these and their inverses.
 \end{proof}

\begin{rmks}{\rm
\begin{enumerate}
 \item  One might be tempted to think that if we take $W=\bfC_1$,  the class of all arrows in $\bfC$,
then every horizontal arrow in $\bfC\{W\}$
will have  a companion and a conjoint. Unfortunately that is not the case: only arrows
of the form $\xymatrix{(&\ar[l]_{uv})\ar[r]^u &(\ar[r]^v&)}$
get a companion and a conjoint. However, if we restrict ourselves to such horizontal arrows (for any chosen class $W$),
take all vertical arrows and the full sub double category of double cells, 
we do get a double category in which every horizontal arrow has a companion and a conjoint.
That double category is the transpose of what Shulman \cite{Shul} has called a fibrant double category and
what was called a gregarious double category in \cite{DPP-spans2}.
\item
The vertical and horizontal property together determine $\CW$ up to both horizontal 
and vertical equivalence of weakly globular double categories,
but note that for the vertical equivalence the arrows are pseudo-functors, where for
the horizontal equivalence the arrows are strict functors.
\item The construction of $\CW$  given in Section \ref{construction}
can be extended to 2-categories, where the universal properties of $\calC\{W\}$
can be formulated using the horizontal double category $H\calC$. It can even be extended 
to arbitrary weakly globular double categories, but we choose to present these results in a sequel to this paper,
as they involve a lot of technical details.
\end{enumerate}}

\end{rmks}

\section{Groupoidal weakly globular double categories}\label{grpdl}

In this section we give a homotopical application of weakly globular double
categories. We introduce a subcategory of the latter, whose objects we call
groupoidal weakly globular double categories.

We show that this category provides an algebraic model of 2-types. This
generalizes a result of \cite{bp}, where weakly globular double groupoids are
proved to model $2$-types.

\begin{definition}\rm{\cite{tam}}\;\it\label{tam3}
{\rm The category $\Tm$ of Tamsamani weak $2$-groupoids is the full subcategory of
$\Ta$ whose objects $X$ are such that, for all $a,b \in X_0$, $X_{(a,b)}$ and
$\Pi_0 X$ are groupoids.}
\end{definition}

It is shown in \cite{tam} that $\Tm$ is an algebraic model of $2$-types; that
is, there is an equivalence of categories $\Tm\bsim^2\,\;\simeq\;\calH
o(\text{2-types})$. From the results of \cite{lp}, it follows that, if we
enlarge the morphisms of $\Tm$ to include the pseudo-morphims, we still obtain
an equivalence of categories $(\Tm)_{\ps}\bsim^2\,\;\simeq\;\calH
o(\text{2-types})$.

\begin{definition}\label{se5.def1}{\rm
The category $\GWGDbl$ of groupoidal weakly globular double categories is the
full subcategory of $\WGDbl$ whose objects $\bbX$ are such that
\begin{itemize}
  \item  [i)] For all $a,b\in \bbX^d_{0}$, $\bbX_{(a,b)}$ is a groupoid.
  \item [ii)] $\Pi_0 \bbX$ is a groupoid, where $\Pi_0:\WGDbl\rw\Cat$ is as in Lemma
  \ref{s4.lem2}.
\end{itemize}
A morphism in $\GWGDbl$ is a 2-equivalence if and only if it is a 2-equivalence
in $\WGDbl$.

The category $(\GWGDbl)_{\ps}$ is the full subcategory of $(\WGDbl)_{\ps}$ whose
objects are in $\GWGDbl$.}
\end{definition}

\begin{remark}\rm\label{s5.rem1}
The notion of groupoidal weakly globular double category is more general than
the one of weakly globular double groupoid introduced in \cite{bp}. In
particular, objects of $\GWGDbl$ are not necessarily double groupoids, as
the horizontal categories
$\bbX_{*0}$ and $\bbX_{*1}$
are not required to be groupoids if $\bbX\in\GWGDbl$.
This notion is similar to the one of Tamsamani weak 2-groupoids, where
inverses of horizontal morphisms exist only in a weak sense. Indeed, the next
proposition establishes a comparison with Tamsamani's model.
\end{remark}

\begin{proposition}\label{s5.pro1}
    The functors $Q:(\Ta)_{\ps}\rw (\WGDbl)_{\ps}$ and $D:(\WGDbl)_{\ps}\rw(\Ta)_{\ps}$
    restrict to functors $Q:(\Tm)_{\ps}\rw (\GWGDbl)_{\ps}$ and
    $D:(\GWGDbl)_{\ps}\rw(\Tm)_{\ps}$ and induce an equivalence of categories:
    \begin{equation*}
        (\GWGDbl)_{\ps}\bsim^2\,\;\simeq\;(\Tm)_{\ps}\bsim^2\,\;\simeq\;\calH o(\text{2-types})
    \end{equation*}
\end{proposition}

\begin{proof}
Let $X\in(\Tm)_{\ps}$. As in the proof of Proposition \ref{s4.pro3} for each
$a,b\in X_{0*}$ there is a weak equivalence $(\alpha_x)_{(a,b)}:(QX)_{(a,b)}\rw
X_{(a,b)}$, and there is an isomorphism $\Pi_0 QX\cong\Pi_0 X$. Since
$X_{(a,b)}$ and $\Pi_0 X$ are groupoids, such are $(QX)_{(a,b)}$ and $\Pi_0
QX$. This proves that $QX\in(\GWGDbl)_{\ps}$.

Let $\bbX \in(\GWGDbl)_{\ps}$. Then, by Proposition \ref{s4.pro2}, there is a
morphism $D\bbX\rw N_h \bbX$ in $\Ps[\dop,\Cat]$  which is a levelwise
categorical equivalence. This implies that, for each $a,b\in \bbX_0^d$, there
is a categorical equivalence $(D\bbX)_{(a,b)}\rw \bbX_{(a,b)}$ and an isomorphism
$\Pi_0 D\bbX\cong \Pi_0 \bbX$. Since $\Pi_0 \bbX$ and $\bbX_{(a,b)}$ are
groupoids, such are $(D\bbX)_{(a,b)}$ and $\Pi_0 D\bbX$. This proves that
$D\bbX\in(\Tm)_{\ps}$. The rest follows immediately from Proposition
\ref{s4.pro3} and \cite{tam}.
\end{proof}

\begin{remark}\rm\label{s5.rem2}
     For every object $\bbX$ of $\GWGDbl$ it is possible, as in the case of the weakly
     globular double groupoids of \cite{bp}, to give an algebraic description
     of the homotopy groups and of the Postnikov decomposition of the
     classifying space $B\bbX$, where $B$ is the composite functor
     \begin{equation*}
        B:(\GWGDbl)_{\ps}\rw(\Tm)_{\ps}\rw \tp
     \end{equation*}
    as in Proposition \ref{s5.pro1}. This is done as follows. Let $B\Pi_0 \bbX$ be
    the classifying space of the groupoid $\Pi_0 \bbX$; for each $x\in \bbX_0^d$ let
    $\id_x\in \bbX_{10}$ be $\id_x = \sigma_0\gamma'(x)$, where $\gamma':\bbX^d_{0}\rw
    \bbX_0$ is the inverse of $\gamma:\bbX_0\rw \bbX^d_{0}$ and $\sigma_0:\bbX_{00}\rw
    \bbX_{10}$ is the face operator. Then it follows from Proposition
    \ref{s5.pro1} and from \cite{tam} that
    \begin{equation*}
        \begin{split}
            & \pi_i(B\bbX,x)=\pi_i(B\Pi_0 \bbX,x)\quad i=0,1 \\
            & \pi_2(B\bbX,x)=\mathrm{Hom}_{\bbX_{1}}(\id_x,\id_x).
        \end{split}
    \end{equation*}
    Further, there is a morphism $\bbX\rw c\Pi_0 \bbX$, where $c\Pi_0 \bbX$ is the
    double category which, as a category object in $\Cat$ in direction 2,
    is discrete with object of objects $\Pi_0 \bbX$.
    From above, we see that the morphism $\bbX\rw c\Pi_0 \bbX$ induces
    isomorphisms of homotopy groups $\pi_i (B \bbX)\cong \pi_i (Bc\Pi_0 \bbX)\;$ $i=0,1$.
    Hence this gives algebraically the Postnikov decomposition of $B\bbX$.
\end{remark}

\end{document}